\documentclass[a4paper,reqno]{amsart}

\usepackage[english]{babel}
\usepackage{a4wide}
\usepackage{graphicx}
\usepackage{amsmath,amssymb,amsthm, amsgen}
\usepackage[ansinew]{inputenc}
\usepackage{listings}
\usepackage{color}
\usepackage{hhline}
\usepackage{multirow}
\usepackage{enumerate}
\usepackage{mathtools}
\usepackage[bookmarks]{}
\usepackage{booktabs}
\usepackage{amsfonts}   
\usepackage{bbm}
\usepackage{tikz}
\usetikzlibrary{patterns, decorations.pathreplacing}
\usepackage{url}
\usepackage{stmaryrd}
\usepackage{longtable} 

\newtheorem{thm}{Theorem}[section]
\newtheorem{prop}[thm]{Proposition}
\newtheorem{ass}[thm]{Assumption}
\newtheorem{lem}[thm]{Lemma}

\newtheorem{defn}[thm]{Definition}

\newtheorem{rem}[thm]{Remark}

\newcommand{\argmax}{\operatorname*{arg\,max}}

\newcommand{\bmu}{\boldsymbol \mu}
\newcommand{\bnu}{\boldsymbol \nu}
\newcommand{\bone}{\boldsymbol 1}
\newcommand{\brho}{\boldsymbol \rho} 
\newcommand{\brhomix}{\boldsymbol \rho_{\operatorname{mix}}}
\newcommand{\brhosep}{\boldsymbol \rho_{\operatorname{sep}}} 
\newcommand{\rhosep}{\rho_{\operatorname{sep}}} 
\newcommand{\rhomix}{\rho_{\operatorname{mix}}} 
\newcommand{\bsigma}{\boldsymbol \sigma}
\newcommand{\bW}{\mathbf W}
\newcommand{\bw}{\mathbf w}
\newcommand{\bzero}{\boldsymbol 0}

\renewcommand{\div}{\operatorname{div}}

\newcommand{\indicatornoacc}[1]{ \mathbbm{1}_{ #1 } }
\newcommand{\N}{\mathbb N}
\newcommand{\R}{\mathbb R}
\newcommand{\supp}{\operatorname{supp}}
\newcommand{\Veff}{V_{\operatorname{eff}}}
\newcommand{\Veffce}{(V^{**})_{\operatorname{eff}}}
\newcommand{\Vreg}{V_{\operatorname{reg}}}
\newcommand{\Vsing}{V_{\operatorname{sing}}} 
\newcommand{\Weff}{W_{\operatorname{eff}}}
\newcommand{\weakto}{ \rightharpoonup }
\newcommand{\xto}[1]{\xrightarrow{#1}}

\def\XXint#1#2#3{{\setbox0=\hbox{$#1{#2#3}{\int}$}
     \vcenter{\hbox{$#2#3$}}\kern-.5\wd0}}

\begin{document}

\title{Many-particle limits and non-convergence of dislocation wall pile-ups} 

\author{Patrick van Meurs}


\begin{abstract}
The starting point of our analysis is a class of one-dimensional interacting particle systems with two species. The particles are confined to an interval and exert a nonlocal, repelling force on each other, resulting in a nontrivial equilibrium configuration. This class of particle systems covers the setting of pile-ups of dislocation walls, which is an idealised setup for studying the microscopic origin of several dislocation density models in the literature. Such density models are used to construct constitutive relations in plasticity models.

Our aim is to pass to the many-particle limit. The main challenge is the combination of the nonlocal nature of the interactions, the singularity of the interaction potential between particles of the same type, the non-convexity of the the interaction potential between particles of the opposite type, and the interplay between the length-scale of the domain with the length-scale $\ell_n$ of the decay of the interaction potential. Our main results are the $\Gamma$-convergence of the energy of the particle positions, the evolutionary convergence of the related gradient flows for $\ell_n$ sufficiently large, and the non-convergence of the gradient flows for $\ell_n$ sufficiently small.
\end{abstract}

\maketitle


\section{Introduction}
\label{sec:intro}

\newcommand{\Enu}[1]{\tilde E^{#1}}
\newcommand{\Enitau}{\tilde E^{(\operatorname{int})}}
\newcommand{\En}[1]{ E_n^{#1}}
\newcommand{\Enita}{ E_n^{(\operatorname{int})}}
\newcommand{\Enmix}{ \hat E_n^{+-}}
\newcommand{\E}[1]{ E^{#1}}
\newcommand{\Eita}{ E^{(\operatorname{int})}}
\newcommand{\Egita}{ \hat E^{(\operatorname{int})}}
\newcommand{\eita}{ \psi }
\newcommand{\eenita}{ \eita_{\varepsilon \alpha_n} }
\newcommand{\elita}{ \eita_\ell }
\newcommand{\emita}{ \eita_m }
\newcommand{\einfita}{ \eita_\infty }
\newcommand{\emkita}{ \eita_{m_k} }
\newcommand{\Glimf}{ \operatorname*{\Gamma-lim\,inf} }

Plasticity of metals is facilitated by many \emph{dislocations} (i.e., line defects in the crystallographic lattice) interacting on a small length-scale. Since it is undesirable and computationally heavy to model plasticity by keeping track of individual dislocations, there is a large community which develops models for the dislocation density; see, e.g., \cite{GromaBalogh99, 
GromaCsikorZaiser03, 
HochrainerZaiserGumbsch07,
KooimanHuetterGeers15}
for several such models. However, due to the complexity of the dislocation interactions, these models lack a mathematically precise connection with the dislocation interactions on the microscale. In this paper we seek such a connection for a family of simplified models for interacting dislocations, parametrised by two parameters. We both prove such connections for certain scaling regimes of the parameters (Theorem \ref{thm:main} and Theorem \ref{t:dyncs}), and identify \emph{non-convergence} in other scaling regimes (Proposition \ref{pp:sep} and Proposition \ref{pp:Esep}).

The simplified model for the interacting dislocations which we consider in this paper is a one-dimensional interacting particle system with two species (see \S \ref{ss:PS}). Interacting particle systems with multiple species are of rapidly increasing interest; see, e.g., 
\cite{ChapmanXiangZhu15, 
DiFrancescoFagioli13, 
BerendsenBurgerPietschmann16ArXiv, 
EversFetecauKolokolnikov16ArXiv} 
and the references therein for applications to dislocation networks, cellular aggregation, granular media, pedestrian movement, opinion formation and predator-pray models. A common challenge in these particle systems is the passage to the many-particle limit. The complexity of this limit passage lies in the high sensitivity of the particle system on the type of interactions between particles of the same species and the type of interactions between particles of different species. Our aim is therefore to impose minimal assumptions on the interaction potential.

\subsection{The particle system}
\label{ss:PS} 

The starting point of our analysis is a more general version of the one-dimensional particle system posed in \cite{DoggePeerlingsGeers15b}. In \S \ref{s:appl} we describe how the system in \cite{DoggePeerlingsGeers15b} models interacting dislocations. Here, the state of the particle system is characterised by a one-dimensional chain of $n^+ \in \N$ positive particles and $n^- \in \N$ negative particles with positions
\begin{equation} \label{for:defn:Omeganpm}
  x^{n,\pm} 
  := \big(x_1^\pm, \ldots, x_{n^\pm}^\pm \big)^T
  \in \Omega_n^\pm 
  := \big\{ y \in \R^{n^\pm} : 0 \leq y_1 < \ldots < y_{n^\pm} \leq 1 \big\},
\end{equation}
where $n := n^+ + n^-$ is the total number of particles. Figure \ref{fg:PS} illustrates an example.

\begin{figure}[h]
\centering
\begin{tikzpicture}[scale=0.7, >= latex]
	\draw (0,-0.3) node [anchor = north west] {$x_1^+$}; 
	\draw (2,-0.3) node [below] {$x_2^+$};  
	\draw (5,-0.3) node [below] {$x_3^+$};  
	\draw[gray] (6,0.3) node [above] {$x_1^-$};
	\draw[gray] (9,0.3) node [above] {$x_2^-$};
	\draw (13,-0.3) node [below] {$x_4^+$};
	\draw[gray] (16,0.3) node [above] {$x_3^-$};
	\draw[very thick] (0,-1.2) node [below] {$0$} -- (0,1.2);
	\draw (0,0) -- (18,0);
	\draw[very thick] (18,-1.2) node [below] {$1$} -- (18,1.2);
	\fill[pattern = north east lines] (-1, -1.2) rectangle (0, 1.2);
	\fill[pattern = north east lines] (18, -1.2) rectangle (19, 1.2);	
	     
    \foreach \x in {0,2,5,13}
      {
      \fill (\x, 0) circle (0.25);
      }
    \foreach \x in {6,9,16}
      {
      \fill[gray] (\x, 0) circle (0.25);
      }
\end{tikzpicture}
\caption{Example of $x^{n,+}$ (black) and $x^{n,-}$ (gray) for $n^+ = 4$ and $n^- = 3$.}\label{fg:PS}
\end{figure}
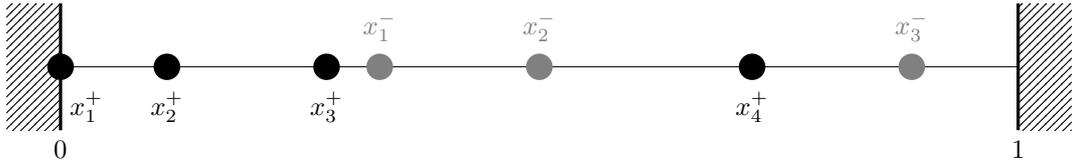

We consider the energy $E_n : \Omega_n^+ \times \Omega_n^- \to \R$ given by
\begin{align} \notag
    E_n (x^{n,+}, x^{n,-})
	&= \frac1{n^2} \sum_{i=1}^{n^+} \sum_{j = 1}^{i-1} \alpha_n V \big( \alpha_n (x_i^+ - x_j^+) \big)
        + \frac1{n^2} \sum_{i=1}^{n^-} \sum_{j = 1}^{i-1} \alpha_n V \big( \alpha_n (x_i^- - x_j^-) \big) \\\label{for:defn:En}
        &\quad+ \frac1{n^2} \sum_{i=1}^{n^+} \sum_{j = 1}^{n^-} \alpha_n W \big( \alpha_n (x_i^+ - x_j^-) \big)
        + \gamma_n^2 \frac1n \sum_{i = 1}^{n^+} x_i^+
        + \gamma_n^2 \frac1n \sum_{i = 1}^{n^-} (1 - x_i^-).
\end{align}
The parameter $\gamma_n \geq 0$ regulates the strength of the affine external potential (which models a constant applied force on the particles), which favours the positive particles to cluster at the left barrier at $x = 0$ and the negative particles to cluster at the right barrier at $x = 1$. The parameter $\alpha_n = \ell_n^{-1} > 0$ is the inverse of the length-scale of the decay of the interactions between particles. The corresponding interaction potential for particles of the same type is given by $V : \R \to (-\infty, \infty]$, and $W : \R \to \R$ denotes the interaction potential for particles of opposite type. Figure \ref{fig:VW} illustrates prototypical examples for $V$ and $W$. Minimal assumption on $V$ and $W$ are that $W \in L^\infty (\R)$, and $V$ is bounded from below on $\R$, bounded from above on compact sets of $\R \setminus \{0\}$ while $V(r) \to \infty$ as $r \to 0$. The singularity of $V$ at $0$ and the boundedness properties of $V$ and $W$ prevent the particles from clustering, which therefore results in a nontrivial interplay with the external loading term. We leave the precise assumptions on $V$ and $W$ to Theorem \ref{thm:main} and Theorem \ref{t:dyncs}.

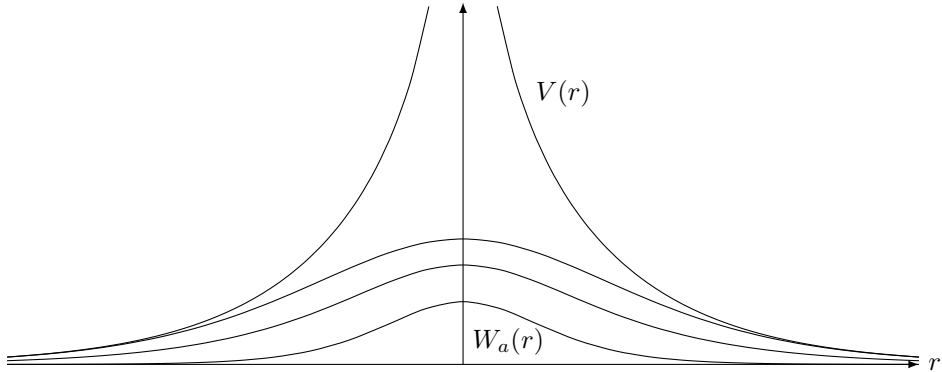
\begin{figure}[h]
\centering
\begin{tikzpicture}[scale=1.2, >= latex]    
\def \w {5}
        
\draw[->] (-\w,0) -- (\w,0) node[right] {$r$};
\draw[->] (0,0) -- (0,4);
\foreach \y in {90,120,180}
{
\draw[domain=-\w:\w, smooth] plot (\x,{ -\x * sinh(\x) / (cosh(\x) - cos(\y)) + ln(cosh(\x) - cos(\y)) + ln(2) });
}
\draw[domain=0.375:\w, smooth] plot (\x,{ \x * sinh(\x) / (cosh(\x) - 1) - ln(cosh(\x) - 1) - ln(2) });
\draw[domain=-\w:-0.375, smooth] plot (\x,{ \x * sinh(\x) / (cosh(\x) - 1) - ln(cosh(\x) - 1) - ln(2) });
\draw (1.1,3) node {$V(r)$};
\draw (0,0) node[anchor = south west] {$W_a(r)$};
\end{tikzpicture} \\
\caption{Plots of $V$ and three choices for $W$ (labelled $W_a$ for $a = 0, \tfrac12, 1$, from bottom to top) corresponding to dislocation walls (see \eqref{for:defn:V} and \eqref{for:defn:W} for the precise expressions).}
\label{fig:VW}
\end{figure}

\subsection{The parameters $\alpha_n$ and $\gamma_n$ in the single-species case}
\label{ss:al:ga}

Even in the single-specie case ($n^- = 0$), the asymptotic behaviour of the parameters $\alpha_n$ and $\gamma_n$ as $n \to \infty$ is crucial for the features of the many-particle limit. These many particle-limits were established first in \cite{GeersPeerlingsPeletierScardia13} on the half-infinite domain $[0, \infty)$ with $\gamma_n = 1$ to keep the particles confined. The corresponding energy reads
\begin{align} \notag
    \tilde E_n(x^n) = \frac1{n^2} \sum_{i=1}^{n} \sum_{j = 1}^{i-1} \alpha_n V \big( \alpha_n (x_i - x_j) \big) + \frac1n \sum_{i = 1}^{n} x_i.
\end{align}
Depending on the asymptotic behaviour of $\alpha_n$, five different limiting energies for the particle density were derived. Two of them are characterised by $\alpha_n \to \alpha > 0$ and $\tfrac1n \alpha_n \to \alpha > 0$ as $n \to \infty$, and the other three treat the case in which $\alpha_n$ is either asymptotically larger or smaller then any of these two limiting cases. 

Many particle systems in the literature fit to $\alpha_n \to \alpha >0$, which corresponds to a setting where the length-scale of $V$ is proportional to the length-scale of the support of the particle density. In particular, if the particles remain in a bounded set, then $\tilde E_n$ is independent of the tails of $V$. The many-particle limit of such systems is studied extensively; see, e.g., \cite{Schochet96, 
PetracheSerfaty14,
Duerinckx16,
Hauray09,
MoraPeletierScardia14ArXiv,
GarroniLeoniPonsiglione10,
CanizoPatacchini16ArXiv}
for a wide range of potentials, corrector estimates, gradient flows, and higher dimensional domains.

Another important class of particle systems corresponding to $\tfrac1n \alpha_n \to \alpha >0$ is given by atoms interacting by a Lennard-Jones potential, and is therefore studied in many different contexts \cite{Hudson13, HudsonOrtner14, BraidesGelli04, BlancLeBrisLions07, HallHudsonVanMeurs16ArXiv, ForcadelImbertMonneau09}. One characteristic of the related particle system is that the length-scale of $V$ is proportional to the length-scale of neighbouring particles, which typically scales as $\tfrac1n$ times the length-scale of the support of the particle density. As a result, $\tilde E_n$ is hardly sensitive to changes in the singularity of $V$ at zero, while the tail behaviour of $V$ shapes the limiting energy. Consequently, the proofs that establish the many-particle limit in this scaling regimes differ completely from the regime in which $\alpha_n \to \alpha >0$.

The limiting energy of the intermediate regime $1 \ll \alpha_n \ll n$ has the porous medium equation as its gradient flow (see \cite{Oelschlager90} for $V$ regular enough). The minimisers exhibits intricate boundary layers \cite{HallChapmanOckendon10, GarroniVanMeursPeletierScardia16}, and again the proof for the many-particle limit differs substantially. A more detailed discussion on all scaling regimes of $\alpha_n$ and their connection is given in \cite{GeersPeerlingsPeletierScardia13, ScardiaPeerlingsPeletierGeers14}.

In \cite{VanMeursMunteanPeletier14} the single-type scenario in \cite{GeersPeerlingsPeletierScardia13} is extended by considering finite domains. The corresponding energy is given by $E_n$ as in \eqref{for:defn:En} with $n^- = 0$. The main result in \cite{VanMeursMunteanPeletier14} is that the asymptotic behaviours of $\alpha_n$ and $\gamma_n$ can be treated independently in the many-particle limit, given a convenient rescaling of the particle positions (and $\tfrac1n \alpha_n$ bounded). Consequently, the asymptotic behaviour of $\alpha_n$ has the same effect on the interactions in the limit $n \to \infty$ as in \cite{GeersPeerlingsPeletierScardia13}. Moreover, the asymptotic behaviour of $\gamma_n$ determines whether the particles in the limit $n \to \infty$ are confined by either the finite domain ($\gamma_n \to 0$), the external force ($\gamma_n \to \infty$), or both effects ($\gamma_n \to \gamma > 0$).

\subsection{Connection to the literature on multiple-species models}

While there are many results in the literature on many-particle limits of $E_n$ as in \eqref{for:defn:En} for the single-type scenario (see \S \ref{ss:al:ga}), few results exist on many-particle limits of multiple specie models. In \cite{DiFrancescoFagioli16, Zinsl16} a stability result is proven for solutions of the continuum gradient flow related to $\alpha_n \to \alpha > 0$. Since these results require $V$ to be regular (in particular $V(0) < \infty$), the discrete gradient flow satisfies the same weak equation as the continuum gradient flow, and thus the stability result includes the many particle limit. Since we wish to include potentials $V$ that are singular at $0$, we need to seek other proof methods.

In \cite{ChapmanXiangZhu15} the upscaling of many dislocation dipoles (pairs of positive and negative dislocations) is studied in a dynamical setting with formal asymptotic techniques. While their setting relates to ours by setting $V = -\log$ and $W_n \to -V$, the unboundedness of $W_n(0)$ as $n \to \infty$ results in intricate effects (e.g., a fast time-scale describing dipole formation) which we do not address here. 

Therefore, in the proof of Theorem \ref{thm:main}, we rely on the methods developed in the literature of single-type scenarios. However, in the cases when $\alpha_n \to \infty$, we need to develop a new proof strategy, because the techniques in \cite{BraidesGelli04} and \cite{GeersPeerlingsPeletierScardia13} do not apply. This strategy seem flexible for extensions to higher dimensions, and is therefore one of the main contributions of this paper.

\subsection{Theorem \ref{thm:main}: $\Gamma$-convergence of $E_n$}

Similarly to \cite{VanMeursMunteanPeletier14}, we identify the vectors $x^{n, \pm}$ by their empirical distribution 
\begin{equation} \label{for:defn:munpm}
    \mu_n^\pm := \frac1n \sum_{j=1}^{n^\pm} \delta_{x_j^\pm},
\end{equation}
which have total mass $n^\pm / n \in [0,1]$. We say that $\mu_n^\pm$ converges in the \emph{narrow topology} to a measure $\mu^\pm$ if
\begin{equation} \label{for:defn:narrow:conv}
    \int_0^1 \varphi \, d \mu_n^\pm 
    \xto{ n \to \infty } \int_0^1 \varphi \, d\mu^\pm
    \quad \text{for all } \varphi \in C ([0,1]),
\end{equation}
where the integrals are taken over the closed interval $[0,1]$.

To state the main result (Theorem \ref{thm:main}) we extend $E_n$ to apply to measures by setting
\begin{equation*} 
    E_n (\mu^+, \mu^-) = \left\{ \begin{aligned}
      &E_n (x^{n,+}, x^{n,-}), 
      &&\text{if } \mu^\pm = \frac1n \sum_{j=1}^{n^\pm} \delta_{x_j^\pm}, \\
      &\infty, 
      &&\text{otherwise.}
    \end{aligned} \right.
\end{equation*}

\begin{thm}[$\Gamma$-convergence] \label{thm:main}
Let $\gamma_n \to \gamma \in [0,\infty)$ and $\alpha_n > 0$ be such that either $\alpha_n \rightarrow \alpha$, $1 \ll \alpha_n \ll n$ or $\tfrac1n {\alpha_n} \rightarrow \alpha$as $n \to \infty$ for some $\alpha  \in (0, \infty)$. Depending on the scaling regime of $\alpha_n$, let $V$, $W$ and $\Eita$ be as in Table \ref{tab:Eita}. Then $E_n$  $\Gamma$-converges with respect to the narrow topology to
\begin{equation*}
    E (\mu^+, \mu^-) := \Eita (\mu^+, \mu^-) + \gamma^2 \int_0^1 x \, d\mu^+(x) + \gamma^2 \int_0^1 (1-x) \, d\mu^-(x).
\end{equation*}
\end{thm}

\begin{table}[b]
\centering 
\begin{tabular}{lll}
  \toprule
  Regime & $V, W$ satisfy & $\Eita (\mu^+, \mu^-)$ 
  \\
  \midrule
  $\alpha_n \rightarrow \alpha$ 
  & Assumption \ref{ass:VW:L2}
  & $\begin{aligned}
    &\frac12 \iint_{[0,1]^2} \alpha V (\alpha (x-y)) \, d\big( \mu^+ \otimes \mu^+ + \mu^- \otimes \mu^- \big)(x,y) \\
       &+ \iint_{[0,1]^2} \alpha W \big( \alpha ( x - y ) \big) \, d( \mu^+ \otimes \mu^- )(x,y)
  \end{aligned}$ 
  \\ \midrule
  $1 \ll \alpha_n \ll n$ 
  & Assumption \ref{ass:VW:L3}
  & $\begin{aligned} 
     &\bigg( \int_0^\infty V \bigg) \int_0^1 \big( \rho^+ (x)^2 + \rho^- (x)^2 \big) \, dx \\
     &+ \bigg( \int_0^\infty W \bigg) \int_0^1 2 \rho^+ (x) \rho^- (x) \, dx
     \end{aligned}$ 
  \\ \midrule
  $\dfrac{\alpha_n}n \rightarrow \alpha$ 
  & Assumption \ref{ass:VW:L4}
  & $\displaystyle \int_0^1 \eita \big( \rho^+ (x), \rho^- (x) \big) \, dx,$
  \\
  \bottomrule
\end{tabular}
\caption{Expressions for $\Eita$, the interaction part of the limit energy in Theorem \ref{thm:main}. For $\alpha_n \gg 1$ the expressions are valid when $\mu^+$ and $\mu^-$ are absolutely continuous (see \eqref{for:defn:abs:cont}) with density $\rho^+, \rho^- \in L^2(0,1)$ respectively; otherwise $\Eita (\mu^+, \mu^-) = \infty$. The function $\eita \in C([0, \infty)^2)$ is implicitly determined as a limit of cell-problems \eqref{for:defn:eita}.}
\label{tab:Eita}
\end{table}

\subsection{Comments on Theorem \ref{thm:main}}

\emph{Compactness}. Compactness is for free since the space of non-negative Borel measures on $[0,1]$ with total variation bounded by $1$ is compact with respect to the narrow topology. 
\smallskip

\emph{Assumptions on $V$ and $W$}. The assumptions on $V$ and $W$ are such that the case of dislocation walls in \S \ref{s:appl} is covered for all scaling regimes of $\alpha_n$ as in Table \ref{tab:Eita}. Moreover, the assumptions on $V$ extend the setting in \cite{GeersPeerlingsPeletierScardia13}. In particular, we relax convexity of $V |_{(0,\infty)}$.
\smallskip

\emph{Proof strategy}. In all scaling regimes of $\alpha_n$, we pass to the limit $n \to \infty$ in the setting of measures \eqref{for:defn:munpm}. While this is a common approach when $\alpha_n \to \alpha > 0$, most literature on many-particle limits for $\tfrac1n \alpha_n \to \alpha > 0$ deals with the displacement function $u_n : [0,1] \to [0,1]$, which maps the reference lattice of equispaced points $\tfrac1n, \tfrac2n, \ldots, 1$ to $x_1, x_2, \ldots, x_n$. In this approach, the argument for the many-particle limit relies on the ordering of the particles by 
\begin{equation} \label{fd:un}
  u_n'(s) = n (x_{i+1} - x_i)
  \quad \text{with $i$ such that } x_i < s \leq x_{i+1}. 
\end{equation}
While we assume $x_n^+$ and $x_n^-$ to be ordered vectors, there is no natural ordering in the combined vector $x^n$, and it is therefore not clear how to write $x_i^+ - x_j^-$ conveniently in terms of $u_n^+$ and $u_n^-$. Moreover, while the expressions for $\Eita$ as in Table \ref{tab:Eita} describing the interactions between particles of the same type can be conveniently cast in terms of a displacement map $u^\pm$ \cite{GeersPeerlingsPeletierScardia13}, it is unclear how the terms describing the interactions between particles of opposite type can be written in terms of $u^+$ and $u^-$ (except for the case $\alpha_n \to \alpha > 0$). Nonetheless, our proof is much inspired by \cite{BraidesGelli04}, in which $E_n$ is approximated by a sum of cell problems.

As a consequence of using the setting of measures, where the ordering of the particles is not inherently used, a generalisation to higher dimensions is within reach. We keep the present setting one-dimensional to simplify the arguments and to cover the model of dislocations described in \S \ref{s:appl}.
\smallskip

\emph{Other scaling regimes}. Several scaling regimes of the parameters $\gamma_n$ and $\alpha_n$ are excluded in Theorem \ref{thm:main}. We comment on the meaning of these scaling regimes and the corresponding many-particle limit of $E_n$:
\begin{itemize}
  \item the scaling regime $\gamma_n \to \infty$ corresponds to a scenario in which the external forcing term is large enough to separate the positive particles from the negative particles into pile-ups at the barriers, whose length scale is asymptotically smaller than $L_n$. Consequently, the scaling introduced in \eqref{for:defn:Omeganpm} does not conserve information on the particle distribution in the limit $n \to \infty$. In \S \ref{s:gninf}, we introduce a different scaling (equivalent to the scaling in \cite{GeersPeerlingsPeletierScardia13}), and prove a $\Gamma$-convergence result on the corresponding energy (see Theorem \ref{thm:Gconv:gninf}). The $\Gamma$-limit $\tilde E$ decouples the dependence on $\mu^+$ and $\mu^-$, i.e. 
\begin{equation} \label{f:Gl:decou}
  \tilde E(\mu^+, \mu^-) = \tilde E^+(\mu^+) + \tilde E^-(\mu^-),  
\end{equation}  
where $\tilde E^\pm$ are equivalent to the energy studied in \cite{GeersPeerlingsPeletierScardia13} except for $\mu^\pm$ not having unit mass;
  
  \item the scaling regime $\alpha_n \to 0$ treats the case in which all particle interactions are given by the asymptotic behaviour of the singularity of $V$ at $0$. Consequently, any useful scaling of the energy depends on the type of singularity. For a logarithmic singularity of $V$, it is shown in \cite{VanMeursMunteanPeletier14} that the scaling
  \begin{equation*} 
    \frac{1}{ \alpha_n } E_n (x^{n,+}, x^{n,-}) + \frac{ (n^+)^2 + (n^-)^2 - n }{2 n^2} \log \alpha_n
  \end{equation*} 
  results in a non-trivial $\Gamma$-limit whenever $n^- = 0$. It is easy to extend this result to the case of mixed particles under the assumption that $W$ is continuous at $0$. Indeed, since $|x_i^+ - x_j^-| \leq 1$, it holds that
  \begin{equation*}
    W \big( \alpha_n (x_i^+ - x_j^-) \big)  
    \xto{n \to \infty} W(0)
  \end{equation*}
  uniformly in $i$ and $j$. Hence, the interactions between positive and negative particles are a continuous perturbation to the energy, and their contribution converges to a constant in the $\Gamma$-limit. Thus, the $\Gamma$-limit effectively decouples as in \eqref{f:Gl:decou}.
  
Other types of singularities of $V$ need to be dealt with in a slightly different way. Consider for example $V(r) = r^{-s}$ with $0 < s < 1$. Then, the right scaling of the energy is $\alpha_n^{s-1} E_n (x^{n,+}, x^{n,-})$, and the contribution of the interactions between particles of opposite type vanishes in the limit $n \to \infty$. 
  
  \item the scaling regime $\alpha_n \gg n$ describes the opposite effect of $\alpha_n \to 0$, i.e., all particle interactions are described by the tail-behaviour of $V$ instead. Again, any useful scaling of the energy depends on a detailed description of these tails. For example, $V$ and $W$ given by \eqref{for:defn:V} and \eqref{for:defn:W} have tails which decay exponentially fast. To preserve the effect of the interactions in $E_n$ in the limit $n \to \infty$, we need an exponentially large prefactor to the energy. For $n^- = 0$ this regime is treated in full detail in \cite{VanMeursMunteanPeletier14}. 
\end{itemize}

\subsection{Asymptotic behaviour of the gradient flows of $E_n$}

To treat the many-particle limit of the gradient flow dynamics, we introduce an alternative representation for the particle positions. Given $n^\pm$ and $x^{n,\pm}$, we collect the particle positions $x_i^\pm$ in an ordered vector $x^n \in [0,1]^n$, i.e., $0 \leq x_1 \leq \ldots \leq x_n \leq 1$, and keep track of their sign by a vector $b^n := (b_1, \ldots, b_n) \in \{-1,1\}^n$. There is an obvious isomorphism between $x^n, b^n$ and $x^{n,\pm}$, and thus we switch between both descriptions whenever convenient.

The gradient flow of $E_n$ is given by
\begin{equation} \label{fd:GFn}
  \left\{ \begin{aligned}
	\frac d{dt} x^n (t) &= - n \nabla E_n(x^n (t)),
	&& t > 0, \\
	x^n(0) &= x_\circ^n. &&
	\end{aligned} \right.
\end{equation}
where $x_\circ^{n, \pm} \in \Omega_n^\pm$ is a suitable initial condition. Given Theorem \ref{thm:main}, it is natural to investigate the possibility to pass to the limit $n \to \infty$ in \eqref{fd:GFn}. Such limit passage is obtained in \cite{VanMeursMuntean14} in the single-type case with $V |_{(0,\infty)}$ convex. However, since we consider non-convex $W$, we need to resort to different methods. To this aim, we discuss several evolutionary convergence techniques in the literature:
\begin{enumerate}
  \item if $E_n$ is $\lambda$-convex for an $n$-independent $\lambda$, then the theory in \cite[Chap.~4]{AmbrosioGigliSavare08} applies, and the evolutionary convergence method in \cite[Thm.~2.17]{DaneriSavareLN10} is within reach. We show in \S \ref{s:dyn} that, under mild conditions on $V$ and $W$, $E_n$ is $\lambda_n$-convex with $\mathcal O(-\lambda_n) = \alpha_n^3$. Then, in the scaling regime $\alpha_n \to \alpha > 0$, we show how \eqref{fd:GFn} fits to \cite[Thm.~2.17]{DaneriSavareLN10}. Theorem \ref{t:dyncs} states the evolutionary convergence result of \eqref{fd:GFn};
  \item the general framework by \cite{SandierSerfaty04} requires a characterisation of slopes or upper gradients of $E_n$ and $E$ with a related liminf-inequality. This characterisation is difficult to prove, except for the $\lambda$-convex case (see \cite{Ortner05}), which is considered above;
  \item if $x V'(x)$ is bounded and $\alpha_n \to \alpha > 0$, then Schochet's symmetrisation argument in \cite{Schochet96} can be used to pass to the limit in the weak form of \eqref{fd:GFn}. We apply this method at the end of \S \ref{s:dyn} to describe the limiting gradient flow by a PDE (see \eqref{f:GF});
  \item the scaling regimes $1 \ll \alpha_n \ll n$ and $\tfrac 1n \alpha_n \to \alpha > 0$ are treated in \cite{Oelschlager90} for the single-type case with strong regularity conditions on $V$, including $V (0) < \infty$. The proof method appears difficult to extend to unbounded $V$, let alone the extension to multiple species.
\end{enumerate}

Using the first method listed above, we prove evolutionary convergence of \eqref{fd:GFn} in Theorem \ref{t:dyncs} in the scaling regime $\alpha_n \to \alpha > 0$. The limiting gradient flow for $\alpha = 1$ is given by
\begin{equation} \label{f:GF}
\left\{ \begin{aligned}
    \frac{ \partial \mu^+}{ \partial t }
  &= \big( \mu^+ \big[ (V_\alpha * \mu^+)' + (W_\alpha * \mu_n^-)' + \gamma^2 \big] \big)'
  && \text{in } \mathcal D'((0,\infty) \times (0,1)), \\
  \frac{ \partial \mu^-}{ \partial t }
  &= \big( \mu^- \big[ (V_\alpha * \mu^-)' + (W_\alpha * \mu_n^+)' - \gamma^2 \big] \big)'
  && \text{in } \mathcal D'((0,\infty) \times (0,1)),
\end{aligned} \right.
\end{equation}
where $\mathcal D'$ denotes the space of distributions on the time-space product space, and $V_\alpha := \alpha V( \alpha \, \cdot )$ and $W_\alpha := \alpha W( \alpha \, \cdot )$. The coupled system of continuity equations in \eqref{f:GF} is similar to those studied in \cite{DiFrancescoFagioli16, Zinsl16} for regular $V$.
\smallskip

In the scaling regimes where $\alpha_n \to \infty$, the $\lambda$-convexity property of $E_n$ vanishes in the limit $n \to \infty$, and our proof method of Theorem \ref{t:dyncs} breaks down. To get insight in the solution of the gradient flow \eqref{fd:GFn} for large $n$, we extend in \S \ref{s:num} the numerical simulations of \cite{DoggePeerlingsGeers15b} to larger values of $n$ and different values of $\alpha_n$. We put $\gamma_n = 0$, and take as initial condition a \emph{fully separated} state, i.e.,
\begin{equation} \label{f:full:sep}
  0 = x_1^+ < \ldots < x_{n^+}^+ < x_1^- < \ldots < x_{n^-}^- = 1.
\end{equation}
This setup fits in the framework of dislocations to \emph{interlacing}, which was first considered by \cite{Head59}. The question is whether the particles remain fully separated during the gradient flow dynamics, and if not, to which extend they `interlace', i.e., the number of particle that swap position.

For $V$ and $W$ as in \eqref{for:defn:V} and \eqref{for:defn:W}, the simulations in \S \ref{ss:num} suggest that in the scaling regime $\tfrac1n \alpha_n \to \alpha > 0$ there is a critical value of $\alpha$ beyond which there exist local minima of $E_n$ for $n$ large enough which exhibit full separation \eqref{f:full:sep}. In \S \ref{ss:al:big} we prove the existence of such minimisers (Proposition \ref{pp:sep}). The idea behind the proof is that for increasing $\alpha_n$, the interaction forces between particles of the same type becomes smaller, whereas the force needed to push $x_{n^+}^+$ beyond $x_1^-$ becomes larger.

\begin{prop} \label{pp:sep} Let $E_n$ be as in \eqref{for:defn:En} with $\gamma_n = 0$, $\tfrac1n \alpha_n = \alpha > 0$, and $V$ and $W$ as in \eqref{for:defn:V} and \eqref{for:defn:W}. If $\alpha$ large enough, then for all $n^\pm \geq 2$ the energy $E_n$ admits a local minimiser which is fully separated (see \eqref{f:full:sep}).
\end{prop}

However, at the `continuum equivalent' of the local minimiser in Proposition \ref{pp:sep} for $n^+ = n^- = \tfrac n2$ given by
\begin{equation} \label{fd:brhosep}
  \rhosep^+(x) := \left\{ \begin{array}{ll}
    1
    & x < \tfrac12 \\
    0
    & x > \tfrac12,
  \end{array} \right.
  \quad
  \rhosep^-(x) := \left\{ \begin{array}{ll}
    0
    & x < \tfrac12 \\
    1
    & x > \tfrac12,
  \end{array} \right.
\end{equation}
the $\Gamma$-limit $E$ has infinite slope (Proposition \ref{pp:sep}), which would imply that the evolutionary limit of \eqref{fd:GFn} in the scaling regime $\tfrac1n \alpha_n \to \alpha$ with $\alpha$ large enough, if it exists, is \emph{not} given by the Wasserstein gradient flow of $E$. This reasoning relies on the conjecture that the local minimisers of Proposition \ref{pp:sep} converge to \eqref{fd:brhosep} as $n \to \infty$ with $n^+ = n^- = \tfrac n2$. This conjecture is based on the numerical results in \S \ref{ss:al:big} listed in Table \ref{tb:al2}.

Coming back to the criticality of some $\alpha  > 0$ in the scaling regime $\tfrac1n \alpha_n \to \alpha > 0$, the numerical results for $\alpha = \tfrac12$ in Table \ref{tb:al05} in \S \ref{ss:num} suggest that $x^{n,\pm} (t)$ converges in time to a `mixed state', which on the continuum scale reads as $\rhomix^+ (x) := \tfrac12 =: \rhomix^- (x)$. This observation was also made in \cite{DoggePeerlingsGeers15b} for large values of $n$ and for $\alpha_n = C \sqrt n$ with $C$ fixed. However, both findings are not very quantitative, and thus the asymptotic behaviour for large $n$ and $1 \ll \alpha_n \ll n$ remains illusive. 

\subsection{Discussion and conclusion}

Motivated by missing rigorous micro-to-macro connections for dislocation density models, we consider a class of interacting particle system (given by the energy $E_n$ \eqref{for:defn:En}) consisting of two species, which is also related to other applications \cite{DiFrancescoFagioli13, 
BerendsenBurgerPietschmann16ArXiv, 
EversFetecauKolokolnikov16ArXiv}. For the physically interesting scaling regimes of the parameters $\alpha_n$ and $\gamma_n$, we prove that $E_n$ $\Gamma$-converges to $E$ (Theorem \ref{thm:main}). Our proof method is novel and suited for extension to higher spatial dimensions. Regarding the gradient flows of $E_n$, we prove evolutionary convergence (Theorem \ref{t:dyncs}) for the scaling regime $\alpha_n \to \alpha > 0$ as $n \to \infty$ in which the nonlocality of the interactions is preserved in the limiting gradient flow \eqref{f:GF}. However, in the scaling regimes where the limiting energy becomes local (see Table \ref{tab:Eita}), the existence of evolutionary convergence is far from obvious, because the nonconvexity of the interactions may create local minima in $E_n$ (Proposition \ref{pp:sep}) which are seemingly not preserved by the limiting \emph{local} energy (see Proposition \ref{pp:Esep} and Table \ref{tb:al2} in \S \ref{ss:num}). 

For more complex multi-species interacting particle systems (for instance, in higher-spatial dimensions and more complex interactions), our findings therefore imply that it may be possible to prove a $\Gamma$-convergence result (and hence convergence of global minimisers), but that there may not be any evolutionary convergence result in the sense of Theorem \ref{t:dyncs}. This adds to the findings of \cite{ChapmanXiangZhu15} that for $n$-dependent $W_n$ which become singular in the limit, the limiting gradient flow (if it exists) is much more subtle than the (Wasserstein) gradient flow of the limiting energy $E$, and may not be expressed in terms of the dislocation densities alone (in addition, one may need internal variables accounting for microstructures such as dipoles). For the current dislocation density models mentioned in the introduction, which are stated in terms of the dislocation densities alone, this statement implies that it is not clear at all whether there exists a precise micro-to-macro connection between these models and the underlying dynamics of individual dislocations. This doubt on the existence of micro-to-macro connections leads to three kinds of future challenges on multi-species particle systems:
\begin{itemize}
  \item Under which geometric restrictions on the particle system does evolutionary convergence hold in the sense of Theorem \ref{t:dyncs}?
  \item What kind of microscopic particle configurations lead to a \emph{different} evolution of the macroscopic particle density than predicted by the continuum models in the literature?
  \item If evolutionary convergence in the sense of Theorem \ref{t:dyncs} does not hold for a certain interacting particle system, then can we develop an alternative, satisfactory mathematical statement for `evolutionary convergence'? 
\end{itemize}
In a forthcoming paper we give an answer to the first two questions for the celebrated two-species dislocation density model in \cite{GromaBalogh99}. The third question remains open.
\smallskip

The remainder of the paper is organised as follows. In \S \ref{s:appl} we show how $E_n$ captures the setting of dislocation walls, and how the parameters $\alpha_n$ and $\gamma_n$ can be computed from physical quantities. In \S \ref{s:not} we introduce the mathematical framework. In \S \ref{s:pf} we prove Theorem \ref{thm:main}. In \S \ref{s:gninf} we extend Theorem \ref{thm:main} to the case $\gamma_n \to \infty$. In \S \ref{s:dyn} we prove the evolutionary convergence result (Theorem \ref{t:dyncs}) for the gradient flow \eqref{fd:GFn} in the case $\alpha_n \to \alpha > 0$. In \S \ref{s:num} we rely on both analysis and numerical observations to give a convincing argument that the $n$-dependent gradient flows \eqref{fd:GFn} in the case $\tfrac1n \alpha_n \to \alpha$ for $\alpha$ large enough do not converge to the (Wasserstein) gradient flow of the $\Gamma$-limit $E$. In the appendices we perform those parts of the proofs in this paper that are computationally heavy without containing interesting novel insight.

\section{Application to dislocation walls}
\label{s:appl}

Figure \ref{fg:walls} shows the setting of dislocation walls in $[0, L] \times h \mathbb T$, where $L, h > 0$ and $\mathbb T$ is the one-dimensional torus. Dislocation walls are vertically periodic arrays of edge dislocations which are a distance $h$ apart. We consider both walls of `positive' edge dislocations and walls of `negative' edge dislocations, where the sign is related to the orientation of the dislocations. While in \cite{DoggePeerlingsGeers15b} the negative walls have vertically a phase shift of $\phi = \tfrac12 h$, we allow for any phase shift $\phi \in [\tfrac14 h, \tfrac34 h]$. For any such $\phi$, the force between walls of opposite sign is always repelling.

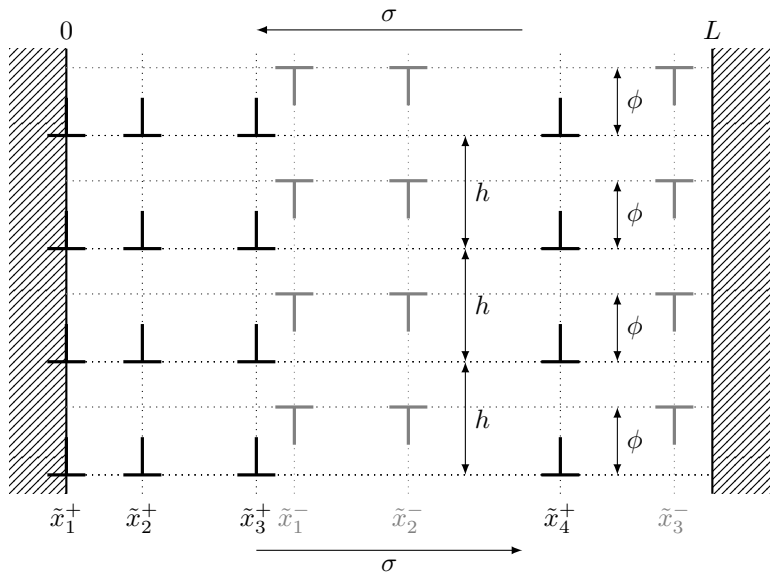
\begin{figure}[h]
\centering
\begin{tikzpicture}[scale=0.5, >= latex] 
  \def \h {3}
  \def \dlc {0.5}
  \def \a {0.6*\h}
  \def \wall {1.5}
  \def \si {5}
  
  \def \xb {0}
  \def \xa {\xb - \wall}
  \def \xc {0}
  \def \xd {2}
  \def \xe {5}
  \def \xf {13}
  \def \xg {17}
  \def \xh {\xg + \wall}
  \def \xma {6}
  \def \xmb {9}
  \def \xmc {16}
  
  \def \yb {0}
  \def \ya {\yb - \dlc}
  \def \yc {\yb + \h}
  \def \yd {\yc + \h}
  \def \ye {\yd + \h}
  \def \yma {\yb + \a}
  \def \ymb {\yc + \a}
  \def \ymc {\yd + \a}
  \def \ymd {\ye + \a}
  \def \yf {\ymd + \dlc}

	\draw (\xc, \ya) node [below] {$\tilde x_1^+$}; 
	\draw (\xd, \ya) node [below] {$\tilde x_2^+$};  
	\draw (\xe, \ya) node [below] {$\tilde x_3^+$}; 
	\draw (\xf, \ya) node [below] {$\tilde x_4^+$};  
	\draw[gray] (\xma, \ya) node [below] {$\tilde x_1^-$};
	\draw[gray] (\xmb, \ya) node [below] {$\tilde x_2^-$};
	\draw[gray] (\xmc, \ya) node [below] {$\tilde x_3^-$};
	\draw[thick] (\xb, \ya) -- (\xb, \yf) node [above] {$0$};
	\draw[thick] (\xg, \ya) -- (\xg, \yf) node [above] {$L$};
	\fill[pattern = north east lines] (\xa, \ya) rectangle (\xb, \yf);
	\fill[pattern = north east lines] (\xg, \ya) rectangle (\xh, \yf);	
	   
	\draw[->] (\xb + \si, \ya - 1.5) -- (\xg - \si, \ya - 1.5) node[midway, below] {$\sigma$};
	\draw[<-] (\xb + \si, \yf + 0.5) -- (\xg - \si, \yf + 0.5) node[midway, above] {$\sigma$};
	\foreach \y in {\yb, \yc, \yd}
      {
      \draw[<->] (\xmb + 1.5, \y) -- (\xmb + 1.5, \y + \h) node[midway, right] {$h$};
      } 
    \foreach \y in {\yb, \yc, \yd, \ye}
      {
      \draw[<->] (\xf + 1.5, \y) -- (\xf + 1.5, \y + \a) node[midway, right] {$\phi$};
      }   
	     
    \foreach \x in {\xc, \xd,\xe,\xf}
      {
      \draw[dotted] (\x, \ya) -- (\x, \yf);
      \foreach \y in {\yb, \yc, \yd, \ye} 
        {
        \draw[dotted] (\xb, \y) -- (\xg, \y);
          \begin{scope}[shift={(\x, \y)}, scale=1]
             \draw[very thick] (-\dlc, 0) -- (\dlc, 0);
             \draw[very thick] (0,0) -- (0,2*\dlc);
          \end{scope} 
        }
      }
    \foreach \x in {\xma, \xmb, \xmc}
      {
      \draw[gray, dotted] (\x, \ya) -- (\x, \yf);
      \foreach \y in {\yma, \ymb, \ymc, \ymd}
        {
        \draw[gray, dotted] (\xb, \y) -- (\xg, \y);
          \begin{scope}[shift={(\x, \y)}, scale=1]
             \draw[gray, very thick] (-\dlc, 0) -- (\dlc, 0);
             \draw[gray, very thick] (0,0) -- (0,-2*\dlc);
          \end{scope} 
        }
      }
\end{tikzpicture}
\caption{Example of $n^+ = 4$ positive dislocation walls and $n^- = 3$ negative dislocation walls. The domain is vertically periodic with period $h$. Its dimensionless equivalent is illustrated in Figure \ref{fg:PS}.}\label{fg:walls}
\end{figure}

Taking the dislocation model of Volterra from 1907, the following energy accounts for all dislocation interactions, including the effect of a constant applied shear stress $\sigma$:
\begin{multline*} 
  	  \Enu{} (\tilde x^+, \tilde x^-) 
    := K \sum_{i=1}^{n^+} \sum_{j = 1}^{i-1} V \bigg( \pi \frac{ \tilde x_i^+ - \tilde x_j^+ }h \bigg)
        + K \sum_{i=1}^{n^-} \sum_{j = 1}^{i-1} V \bigg( \pi \frac{ \tilde x_i^- - \tilde x_j^- }h \bigg)\\
        + K \sum_{i=1}^{n^+} \sum_{j = 1}^{n^-} W_a \bigg( \pi \frac{ \tilde x_i^+ - \tilde x_j^- }h \bigg)
        + \sigma \sum_{i = 1}^{n^+} \tilde x_i^+
        + \sigma \sum_{i = 1}^{n^-} (L - \tilde x_i^-).
\end{multline*}
Here, $\tilde x^\pm := (\tilde x_1^\pm, \ldots, \tilde x_{n^\pm}^\pm) \in \R^{n^\pm}$, $K$ is a material constant, and the interaction potential $V$ is given by
\begin{equation} \label{for:defn:V}
V(r) := r \coth r - \log | 2 \sinh r |,
\qquad r\in \R{}.
\end{equation}
Figure \ref{fig:VW} illustrates $V$. This potential is a special case of the more general potential in \eqref{for:defn:W}, which was first derived in \cite[(19-75)]{HirthLothe82} (an alternative derivation can be found in \cite[Prop.~A.2.2]{VanMeurs15}). We note that, in the special case where only one type of dislocation walls is considered (i.e., $n^- = 0$ or $n^+ = 0$), $\Enu{}$ is the same energy as the one studied in \cite{VanMeursMunteanPeletier14}. The interaction potential $W_a$ describes the interaction between positive and negative walls. It is given by
\begin{equation} \label{for:defn:W}
W_a(r) := \frac12 \log \big( 2 ( \cosh (2r) + a ) \big) - \frac{r \sinh (2r)}{\cosh (2r) + a}, 
\qquad r\in \R{}, \ a \in [-1,1].
\end{equation}
Figure \ref{fig:VW} illustrates $W_a$. The parameter $a$ is related to the phase shift $\phi \in [0, h)$ by $a = - \cos 2 \pi \tfrac\phi h$. We note that $a \in [0,1]$ corresponds to a phase shift between $\tfrac14 h$ and $\tfrac34 h$, and that $W_{-1}(r) = -V(r)$, which is consistent with the fact that dislocations of opposite sign interact with opposite force.

\begin{prop}[Properties of $W_a$] \label{prop:props:Wa} 
It holds that
\begin{enumerate}[(i)]
  \item $W_a \in \mathcal S (\R)$ for all $-1 < a \leq 1$;
  \item $0 < W_0 \leq W_a < W_b \leq W_1 < V$ on $\R$ for all $0 \leq a < b \leq 1$;
  \item $V - W_a$ is strictly convex on $(0, \infty)$ for all $0 \leq a \leq 1$;
  \item $\widehat{W_a} > 0$ for all $0 \leq a \leq 1$.
\end{enumerate}
\end{prop}

We are interested in the many-particle limit of $\Enu{}$. To find a meaningful limit, it is essential to find a proper rescaling of the wall positions $\tilde x^\pm$ and of the energy $\Enu{}$ in terms of the physical parameters $K_n$, $\sigma_n$ and $h_n$. The precise dependence of these parameters on $n$ is a modelling choice, which we choose to keep general. We use the same scaling as in \cite{ScardiaPeerlingsPeletierGeers14} and \cite{VanMeursMunteanPeletier14}. Setting
\begin{equation*}
  \alpha_n := \frac{\pi L_n}{h_n},
  \quad \gamma_n := L_n \sqrt{ \frac{ \pi \sigma_n }{ n K_n h_n } }
\end{equation*}
as dimensionless parameters\footnote{The only difference with \cite{VanMeursMunteanPeletier14} is that we take $\alpha_n$ $n$ times larger}, and rescaling the particle positions and energy as
\begin{equation*} \notag
  x_i^\pm := \frac{ \tilde x_i^\pm }{ L_n }
  \quad \text{and} \quad
    E_n (x^{n,+}, x^{n,-})
    := \frac{\pi L_n}{ n^2 K_n h_n } \Enu{} (L_n x^{n,+}, L_n x^{n,-}),
\end{equation*}
we obtain that rescaled energy $E_n$ is given by \eqref{for:defn:En}.

Regarding dislocation dynamics, we rely on the simplest but widely used relation given by Orowan's linear drag law \cite[(3.3b)]{HullBacon01}. It states that $v = BF$, where $F$ is the horizontal component of the force acting on the dislocation, $v$ is the horizontal velocity of the dislocation, and $B$ is a constant drag coefficient. By the imposed vertical periodicity in Figure \ref{fg:walls}, the velocity of a dislocation wall is given by the velocity of each single dislocation in the wall. Hence, by absorbing $B$ in the time variable, we obtain \eqref{fd:GFn}.

\section{Notation and functional framework}
\label{s:not}

Here we list the symbols and notation which we use in the remainder of this paper: 
\newcommand{\specialcell}[2][c]{\begin{tabular}[#1]{@{}l@{}}#2\end{tabular}}
\begin{longtable}{lll}
$a \wedge b$, $a \vee b$ & $\min \{a, b\}$, $\max \{a, b\}$ & \\
$f_{\operatorname{eff}}$ & $f_{\operatorname{eff}} (x) := \sum_{k=1}^\infty f(kx)$ & \\
$\|f\|_q$ & $L^q$-norm of $f$ on the domain of $f$ & \\
$\widehat f$, $\mathcal{F}(f)$ & \specialcell[t]{Fourier transform of $f$; \\ $\mathcal{F}(f)(\omega) = \widehat f (\omega) := \int_\R f(x) e^{-2\pi ix\omega} \, dx$} & \\
$\mu \otimes \nu$ & product measure; $(\mu \otimes \nu) (A \times B) = \mu(A) \nu (B)$ & \\
$\mu_n \boxtimes \mu_n$ & product measure `without the diagonal' & \eqref{fd:boxt} \\
$\indicatornoacc A$ & $\indicatornoacc A (x)$ equals $1$ if $x \in A$ and $0$ if $x \notin A$ & \\
$\bmu$ & $\bmu := (\mu^+, \mu^-) \in M ([0,1])$ \\
$\mathcal M_+ ([0,1])$ & Space of finite, non-negative Borel measures on $[0,1]$ & \\
$\mathcal M ([0,1]; [0,\infty)^2)$ & $[0,\infty)^2$-valued finite Borel measures on $[0,1]$ & \\
$M ([0,1])$ & Domain of $E$; $M ([0,1]) \subset \mathcal M ([0,1]; [0,\infty)^2)$ & \eqref{fd:M} \\
$\N$ & $\{1,2,3,\ldots\}$ &  \\
$\mathcal P ([0,1])$ & \specialcell[t]{Space of probability measures; \\$\mathcal P ([0,1]) = \{ \mu \in \mathcal M_+ ([0,1]) : \mu ([0,1]) = 1 \}$} & \\
$W(\mu, \nu)$ & $2$-Wasserstein distance between $\mu, \nu \in \mathcal P ([0,1])$ & \cite{AmbrosioGigliSavare08} \\
$\bW(\bmu, \bnu)$ & Modified Wasserstein distance between $\bmu, \bnu \in M([0,1])$ & \eqref{fd:bW:genl} \\ 
\end{longtable} 

For $\mu_n^\pm$ as in \eqref{for:defn:munpm}, we set 
\begin{equation} \label{fd:boxt}
  \mu_n^\pm \boxtimes \mu_n^\pm
  := \frac1{n^2} \sum_{i=1}^{n^\pm} \sum_{ \substack{ j=1 \\ j \neq i } }^{n^\pm} \delta_{(x_i^\pm, x_j^\pm)}
  \in \mathcal M([0,1]^2)
\end{equation}
as the product measure `without the diagonal'. We recall from \cite[Lem.~1]{GeersPeerlingsPeletierScardia13} and \cite[\S 3.4]{Billingsley68} that
$\mu_n^\pm \weakto \mu^\pm$ implies
\begin{align} \notag
  \mu_n^\pm \boxtimes \mu_n^\pm 
  \weakto \mu^\pm \otimes \mu^\pm \quad
  &\text{as } n \to \infty, \text{ and} \\\label{f:ot:conv}
  \mu_n^+ \otimes \mu_n^- 
  \weakto \mu^+ \otimes \mu^- \quad
  &\text{as } n \to \infty.
\end{align}

The domain of the limit energy $E$ defined in Theorem \ref{thm:main} is given by
\begin{equation} \label{fd:M}
  M ([0,1]) 
  := \{ \bmu := (\mu^+, \mu^-) \in \mathcal M_+ ([0,1]) \times \mathcal M_+ ([0,1]) : \mu^+ + \mu^- \in \mathcal P ([0,1]) \},
\end{equation} 
In the special case when $\mu^\pm$ are \textit{absolutely continuous}, i.e.,
\begin{equation} \label{for:defn:abs:cont}
  \text{there exist } \rho^\pm \in L^1_+ (0,1) 
  \text{ such that } d\mu^\pm (x) = \rho^\pm (x) \, dx,
\end{equation}
we denote by $\brho := (\rho^+, \rho^-)$ their density. 

We prove Theorem \ref{thm:main} separately for each of the three scaling regimes of $\alpha_n$ as outlined in Table \ref{tab:Eita}. We establish the corresponding $\Gamma$-convergence result by proving the following two inequalities for all $\bmu \in M ([0,1])$:
\begin{subequations}
\label{for:Gconv}
\begin{alignat}2
\label{for:Gconv:liminf}
&\text{for all } \bmu_n \weakto \bmu, & \liminf_{n\to\infty} E_n (\bmu_n) &\geq E (\bmu), \\
&\text{there exists } \bmu_n \weakto \bmu \text{ such that}\quad & \limsup_{n\to\infty} E_n (\bmu_n) &\leq E (\bmu),
\label{for:Gconv:limsup}
\end{alignat}
\end{subequations}
where $\bmu_n \weakto \bmu$ if and only if $\mu_n^+ \weakto \mu^+$ and $\mu_n^- \weakto \mu^-$. The expression for $E$ depends on the scaling regime of $\alpha_n$. Any sequence $(\bmu_n)$ satisfying \eqref{for:Gconv:limsup} is called a \emph{recovery sequence}. 

A basic property of $\Gamma$-convergence is that it is stable under \emph{continuously converging} perturbations. A sequence of functionals $(G_n)$ converges continuously to $G$ if
\begin{equation} \label{fd:ct:ptb}
  \text{for all } \bmu_n \weakto \bmu, \quad G_n (\bmu_n) \xto{n \to \infty} G (\bmu).
\end{equation}
Then, $\Gamma$-convergence of $E_n$ to $E$ being stable under continuously converging perturbations means that $(E_n + G_n)$ $\Gamma$-converges to $(E + G)$ for all $(G_n)$ converging continuously to $G$.

\section{Proof of Theorem \ref{thm:main}}
\label{s:pf}

We observe from \eqref{for:defn:narrow:conv} that the last two terms of $E_n$ in the right-hand side of \eqref{for:defn:En} converge continuously (see \eqref{fd:ct:ptb}), and thus it suffices to focus on the first three terms describing the interactions. Therefore, in this section, we set $\gamma_n = \gamma = 0$ without loss of generality.

Throughout this section, we use the following symmetry of $E_n$ between positive and negative particles
\begin{equation} \label{f:En:sym}
  E_n (x^{n,+}, x^{n,-}) = E_n (\bone^{n,-} - x^{n,-}, \bone^{n,+} - x^{n,+}),
\end{equation}
where $\bone^{n,\pm} := (1, \ldots, 1)^T \in \R^{n^\pm}$. Hence, any statement on positive particles implies a similar statement on the negative particles.

\subsection{The case $\alpha_n \to \alpha > 0$}
 
We note that $E_n$ can be rewritten as 
\begin{equation} \label{f:En:pm}
  E_n ( x^{n,+}, x^{n,-} )
  = E_n^+( x^{n,+} ) + E_n^-( x^{n,-} ) 
    + \frac1{n^2} \sum_{i=1}^{n^+} \sum_{ j = 1 }^{n^-} \alpha_n W \big( \alpha_n ( x_i^+ - x_j^- ) \big),
\end{equation}
where
\begin{equation*}
  E_n^\pm : \Omega_n^\pm \to [0, \infty),
  \qquad E_n^\pm(x^{n,\pm}) 
  := \frac1{n^2} \sum_{i=1}^{n^\pm} \sum_{ j = 1 }^{i-1} \alpha_n V \big( \alpha_n ( x_i^+ - x_j^+ ) \big),
\end{equation*}
are equivalent to the energy considered in \cite{VanMeursMunteanPeletier14} (except for the argument of $E_n^\pm$ being a vector of size $n^\pm$ instead of $n$).


\begin{ass}[Properties of $V$ and $W$ in case $\alpha_n \to \alpha > 0$] \label{ass:VW:L2}
$V$ and $W$ satisfy
\begin{enumerate}[(i)]
  \item $V = \Vsing + \Vreg$, where $\Vreg \in C(\R)$ is even, and $\Vsing \in L_{\operatorname{loc}}^1 (\R)$ is even, non-negative and decreasing on $(0, \infty)$;
  \item  $W \in C(\R)$ is even.
\end{enumerate}
\end{ass}

\begin{lem}[$\Gamma$-convergence of $E_n^\pm$ \cite{GeersPeerlingsPeletierScardia13}] \label{lem:Gconv:single}
Let $\gamma_n \to \gamma \in [0,\infty)$, $\alpha_n \to \alpha \in (0, \infty)$, and let $V$ satisfy Assumption \ref{ass:VW:L2}. Then $E_n^\pm$ $\Gamma$-converges to 
  \begin{equation*}
    \mu^\pm \mapsto \frac12 \int_0^1 \int_0^1 \alpha V (\alpha (x-y)) \, d (\mu^\pm \otimes \mu^\pm) (x, y). 
  \end{equation*}
Moreover, a recovery sequence in \eqref{for:Gconv:limsup} can be constructed for any prescribed $n^\pm$ which satisfies $\tfrac1n n^\pm \to \mu^\pm ([0,1])$ as $n \to \infty$.
\end{lem}

\begin{proof}
\cite[Thm.~5]{GeersPeerlingsPeletierScardia13} states a similar $\Gamma$-convergence result for the setting of the half infinite domain $[0, \infty)$, for a smaller class of potentials $V$, and for $\sigma_n^\pm := \tfrac1n n^\pm = 1$. \cite[Thm.~1.1]{VanMeursMunteanPeletier14} extends this result to finite domains, and in \cite[\S 3.6]{VanMeurs15} this result is extended to $V$ satisfying Assumption \ref{ass:VW:L2}.

Next we extend to any preset $\sigma_n^\pm \to \sigma^\pm \in [0, 1]$. If $\sigma^\pm > 0$, then $\sigma_n^\pm > 0$ for all $n$ large enough. Thus, setting $\tilde \mu_n^\pm := \mu_n^\pm / \sigma_n^\pm \in \mathcal P ([0,1])$ we find from $\mu_n^\pm \weakto \mu^\pm$ that $\tilde \mu_n^\pm \weakto \mu^\pm/\sigma^\pm =: \tilde \mu^\pm$ as $n \to \infty$. Then, from the unit mass case we infer that
\begin{multline} \label{fp:En:mut}
  E_n^\pm (\mu_n^\pm) 
  = \Big( \frac{n^\pm}{n} \Big)^2 \frac1{(n^\pm)^2} \sum_{i=1}^{n^\pm} \sum_{ j = 1 }^{i-1} \alpha_n V \big( \alpha_n ( x_i^\pm - x_j^\pm ) \big) \\
  = \big( \sigma_n^\pm \big)^2 \frac12 \int_0^1 \int_0^1 \alpha_n V (\alpha_n (x-y)) \, d (\tilde \mu_n^\pm \boxtimes \tilde \mu_n^\pm) (x, y)
  = \big( \sigma_n^\pm \big)^2 E_{n^\pm}^\pm (\tilde \mu_n^\pm)
\end{multline}
$\Gamma$-converges to 
\begin{equation*}
  \sigma^2 \frac12 \int_0^1 \int_0^1 \alpha V (\alpha (x-y)) \, d (\tilde \mu^\pm \otimes \tilde \mu^\pm) (x, y)
  = \frac12 \int_0^1 \int_0^1 \alpha V (\alpha (x-y)) \, d (\mu^\pm \otimes \mu^\pm) (x, y).
\end{equation*}

If $\sigma^\pm = 0$, then the $\Gamma$-limit equals $0$. Indeed, since $V \geq -C$ on $[-\alpha - 1, \alpha + 1]$, the liminf inequality \eqref{for:Gconv:liminf} follows from
\begin{equation*}
  E_n^\pm (\mu_n^\pm) 
  = \frac12 \int_0^1 \int_0^1 \alpha V (\alpha (x-y)) \, d (\mu_n^\pm \otimes \mu_n^\pm) (x, y)
  \geq - \frac C2 \alpha \big( \sigma_n^\pm \big)^2 
  \xto{n \to \infty} 0.
\end{equation*}
For the limsup inequality, we choose $\mu_n^\pm$ such that $E_{n^\pm}^\pm (\tilde \mu_n^\pm)$ is bounded. Then, by \eqref{fp:En:mut}, we obtain $E_n^\pm (\mu_n^\pm) = ( \sigma_n^\pm )^2 E_{n^\pm}^\pm (\tilde \mu_n^\pm) \to 0$ as $n \to 0$.
\end{proof}

\begin{thm}[$\Gamma$-convergence of $E_n$ in case $\alpha_n \to \alpha > 0$] \label{thm:Gconv:L2}
Let $\alpha_n \to \alpha > 0$, and let $V$ and $W$ satisfy Assumption \ref{ass:VW:L2}. Then $E_n$ $\Gamma$-converges to 
  \begin{multline*}
    E (\mu^+, \mu^-) =
    \frac12 \iint_{[0,1]^2} \alpha V (\alpha (x-y)) \, d\big( \mu^+ \otimes \mu^+ + \mu^- \otimes \mu^- \big)(x,y) \\
       + \iint_{[0,1]^2} \alpha W \big( \alpha ( x - y ) \big) \, d( \mu^+ \otimes \mu^- )(x,y). 
  \end{multline*}
Moreover, the recovery sequence in \eqref{for:Gconv:limsup} can be constructed for any prescribed $n^\pm$ which satisfies $\tfrac1n n^\pm \to \mu^\pm ([0,1])$ as $n \to \infty$.
\end{thm}

\begin{proof}
In terms of the measures $\mu_n^\pm$, \eqref{f:En:pm} reads
\begin{equation} \label{f:En:pm:mun}
  E_n ( \mu_n^+, \mu_n^- )
  = E_n^+( \mu_n^+ ) + E_n^-( \mu_n^- ) 
    + \iint_{[0,1]^2} \alpha_n W \big( \alpha_n ( x - y ) \big) \, d( \mu_n^+ \otimes \mu_n^- )(x,y).
\end{equation}
Firstly, since $W \in C(\R)$ and $|x-y| \leq 1$, the sequence of maps $(x,y) \mapsto \alpha_n W ( \alpha_n ( x - y ))$ converges uniformly on $[0,1]^2$ to $\alpha W ( \alpha ( x - y ))$. Secondly, for any $\mu_n^\pm \weakto \mu^\pm$, we have by \eqref{f:ot:conv} that $\mu_n^+ \otimes \mu_n^- \weakto \mu^+ \otimes \mu^-$. Together, these properties imply that the third term in the right-hand side of \eqref{f:En:pm:mun} converges to 
\begin{equation*}
  \iint_{[0,1]^2} \alpha W \big( \alpha ( x - y ) \big) \, d( \mu^+ \otimes \mu^- )(x,y),
\end{equation*}
and thus it is a continuous perturbation \eqref{fd:ct:ptb} to the other two terms in \eqref{f:En:pm:mun}. $\Gamma$-convergence of these two terms follows from Lemma \ref{lem:Gconv:single} and the observation that they decouple the dependence of $E_n$ on $\mu_n^+$ and $\mu_n^-$.
\end{proof}

\subsection{The case $1 \ll \alpha_n \ll n$}

\begin{ass}[Properties of $V$ and $W$ in case $1 \ll \alpha_n \ll n$] \label{ass:VW:L3}
$V$ and $W$ satisfy
\begin{enumerate}[(i)]
  \item $V \in L^1(\R)$ is even, and non-increasing on $(0, \infty)$; \label{ass:VW:L3:V}
  \item $W \in L^1(\R) \cap C(\R)$ is even, satisfies $\mathcal F W \geq 0$, and is non-increasing on $(0, \infty)$; \label{ass:VW:L3:W}
  \item $V - W \not\equiv 0$ can be approximated by $U^\delta \nearrow (V - W)$ pointwise a.e.~on $\R$ as $\delta \to 0$, where $\mathcal F U^\delta \geq 0$ and $U^\delta (0) < \infty$ for all $\delta > 0$. \label{ass:VW:L3:VW}
\end{enumerate}
\end{ass}

A typical example of a couple $(V, W)$ which satisfies Assumption \ref{ass:VW:L3} is given by $W$ as in \eqref{ass:VW:L3:W}, and $V \in L^1 (\R)$ even on $\R$ with $V'' \geq W'' \vee 0 $ on $(0, \infty)$. Then, a possible choice for $U^\delta$ is the convex envelope of 
\begin{equation*}
  x \mapsto \left\{ \begin{aligned}
    &V (x) - W (x)
    &&x > 0, \\
    &\delta^{-1}
    &&x = 0
  \end{aligned} \right.
\end{equation*}
with even extension from $x \in [0, \infty)$ to $x \in \R$. Proposition \ref{prop:props:Wa} implies that $V$ and $W_a$ as in \eqref{for:defn:V} and \eqref{for:defn:W} satisfy Assumption \ref{ass:VW:L3} for all $a \in [0,1]$.

We assume non-negativity of the Fourier transform in Assumption \ref{ass:VW:L3} to rule out the formation of microstructures in $x^{n, \pm}$ which could lower the energy. We sketch the argument on how non-negativity of the Fourier transform prevents such low-energy microstructures, and refer for the details to \cite{GeersPeerlingsPeletierScardia13} and \cite[\S 3.6]{VanMeurs15}. We first consider the single particle case $n^- = 0$, in which we set $W = 0$. The approximation from below by $U^\delta$ allows us to include the self-interactions by
\begin{equation*}
  \frac12 \iint \alpha_n V ( \alpha_n (x - y) ) \, d( \mu_n^+ \boxtimes \mu_n^+ ) (x, y) 
  \geq \frac12 \iint \alpha_n U^\delta ( \alpha_n (x - y) ) \, d( \mu_n^+ \otimes \mu_n^+ ) (x, y) - \frac{\alpha_n}{2 n} U^\delta (0).
\end{equation*}
The non-negativity of $\mathcal F U^\delta$ allows us to split the operation `convolution with $U^\delta$' as applying twice the convolution with $u^\delta$, i.e., $U^\delta = u^\delta * u^\delta$. Setting $U_n^\delta := \alpha_n U^\delta ( \alpha_n \, \cdot \, )$, we obtain
\begin{equation*}
  \frac12 \iint U_n^\delta (x - y) \, d( \mu_n^+ \otimes \mu_n^+ ) (x, y)
  = \frac12 \int_\R \big( u_n^\delta * \mu_n^+ \big)^2 (x) \, dx.
\end{equation*}
This approximation of $E_n$ from below by the square of the $L^2$-norm of $u_n^\delta * \mu_n^+$ is the key for deriving the following $\Gamma$-liminf estimate, and the author is unaware of any other technique which leads to the same lower bound. 

\begin{lem}[$\Gamma$-liminf inequality of $E_n^\pm$ {\cite[Thm.~7]{GeersPeerlingsPeletierScardia13}}] \label{lem:Gconv:single:L3}
Let $1 \ll \alpha_n \ll n$ and let $\mathsf V$ satisfy Assumption \ref{ass:VW:L3}.\eqref{ass:VW:L3:VW} (with $\mathsf V = V-W$). Then, for all $\mu_n^\pm \weakto \mu^\pm$ it holds that
  \begin{equation*}
    \liminf_{n\to \infty} \frac12 \int_0^1 \int_0^1 \alpha_n \mathsf V (\alpha_n (x-y)) \, d (\mu_n^\pm \boxtimes \mu_n^\pm) (x, y) 
    \geq \bigg( \int_0^\infty \mathsf V \bigg) \int_0^1 \rho^\pm (x)^2 \, dx,
  \end{equation*}
  where the right-hand side is defined as $\infty$ if $\mu^\pm$ is not absolutely continuous (cf.~\ref{for:defn:abs:cont}).
\end{lem}
In the proof of Theorem \ref{thm:main}, we apply Lemma \ref{lem:Gconv:single:L3} twice; once with $\mathsf V = V-W$ and once with $\mathsf V = W$. 

In the case of mixed particles, a similar strategy for obtaining a sufficient lower bound results in an additional term given by
\begin{equation*}
  \frac12 \int_\R \big( w_n * \mu_n^+ \big) (x) \big( w_n * \mu_n^- \big) (x) \, dx,
\end{equation*}
where $w_n * w_n = W_n := \alpha_n W ( \alpha_n \, \cdot \, )$. This term is an $L^2$-inner-product rather than the square of an $L^2$-norm. We bound it from below by using the Cauchy-Schwartz inequality, which leaves us to bound $\| w_n * \mu_n^\pm \|_2^2$ by part of the energy $E_n^\pm (\mu_n^\pm)$. Assumption \ref{ass:VW:L3}.\eqref{ass:VW:L3:VW} is chosen to make this estimate work.

For the construction of a recovery sequence \eqref{for:Gconv:limsup}, we do not rely on the technique in \cite{GeersPeerlingsPeletierScardia13}. The main reason is that this technique relies on describing the particle positions $x_i^+$ in terms of the displacement $u_n$ \eqref{fd:un}, which is not suited in the case of multiple species.  Instead, we use the description in terms of $\bmu_n$, and construct the recovery sequence similarly as in \cite{MoraPeletierScardia14ArXiv}. We use the assumption that $V$ and $W$ are non-increasing on $(0, \infty)$ to have the monotonicity result that the energy $E_n$ does not decrease whenever we replace the argument of $V$ or $W$ by a number with smaller absolute value.


\begin{thm}[$\Gamma$-convergence of $E_n$ in case $1 \ll \alpha_n \ll n$] \label{thm:Gconv:L3}
Let $1 \ll \alpha_n \ll n$, and let $V$ and $W$ satisfy Assumption \ref{ass:VW:L3}. Then $E_n$ $\Gamma$-converges to 
  \begin{align} \label{for:defn:E3}
    E (\mu^+, \mu^-) =
    \bigg( \int_0^\infty V \bigg) \int_0^1 \big( \rho^+ (x)^2 + \rho^- (x)^2 \big) \, dx 
    + \bigg( \int_0^\infty W \bigg) \int_0^1 2 \rho^+ (x) \rho^- (x) \, dx,
  \end{align}
  where the right-hand side is defined as $\infty$ if $\mu^\pm$ is not absolutely continuous (cf.~\ref{for:defn:abs:cont}).
\end{thm}

\begin{proof}
Setting $V_n := \alpha_n V ( \alpha_n \, \cdot \, )$ and $W_n := \alpha_n W ( \alpha_n \, \cdot \, )$, we prove the liminf-inequality \eqref{for:Gconv:liminf} by splitting the interaction energies of particles of the same type as
\begin{align*}
  E_n^\pm (\mu_n^\pm) 
  &= \frac12 \iint \big[ (V_n - W_n) + W_{n} \big] (x - y) \, d( \mu_n^\pm \boxtimes \mu_n^\pm ) (x, y) \\
  &= \frac12 \iint [ V_n - W_n ] (x - y) \, d( \mu_n^\pm \boxtimes \mu_n^\pm ) (x, y) \\
  &\qquad + \frac12 \iint W_n (x - y) \, d (\mu_n^\pm \otimes \mu_n^\pm) (x, y) - \frac{n^\pm}{2 n^2} W_n (0).
\end{align*}
Then, we rewrite
\begin{align*}
  E_n (\mu_n^+, \mu_n^-) 
  &= \frac12 \iint [ V_n - W_n ] (x - y) \, d( \mu_n^+ \boxtimes \mu_n^+ ) (x, y) \\
  &\quad + \frac12 \iint [ V_n - W_n ] (x - y) \, d( \mu_n^- \boxtimes \mu_n^- ) (x, y) \\
  &\quad + \frac12 \iint W_n (x - y) \, d \big( (\mu_n^+ + \mu_n^-) \otimes (\mu_n^+ + \mu_n^-) \big) (x, y)
  - \frac{\alpha_n}{2 n} W (0).
\end{align*}
Next we take $\liminf_{n \to \infty}$ on all four terms in the right-hand side separately. The $\liminf_{n \to \infty}$ of the first three terms are given by Lemma \ref{lem:Gconv:single:L3}, and since $\alpha_n \ll n$, the fourth term converges to $0$. \smallskip

We establish the limsup-inequality \eqref{for:Gconv:limsup} by constructing a recovery sequence for $\bmu$ in a dense subset of $M([0,1])$, which is similar to one used in \cite{MoraPeletierScardia14ArXiv}. To construct this subset, we divide the domain of the dislocation walls in closed intervals $I_k$ with $k = 1, \ldots, K$ as in Figure \ref{fig:partn:domain}, with size $\varepsilon > 0$ such that the intervals fit `nicely', i.e., $K \varepsilon (1 + \varepsilon) = 1$. 

\begin{figure}[h]
\centering
\begin{tikzpicture}[scale=1.8] 
    \def \xi {1}
    \def \xii {0.5}
    \def \xiii {1}
    \def \xiv {\xii/2} 
    \def \xv {\xi/2}
    \def \xvi {\xv}
    \def \xvii {\xv/3}
    \def \xviii {2*\xii/3}
    
	\foreach \x in {0, {\xi + \xii}, {3*\xi + 3*\xii +\xiii}}{
	  \begin{scope}[shift={(\x, 0)}, scale=1]
        \draw (0, 0) -- (\xi, 0); 
        \draw[<->] (0, \xvii) -- (\xi, \xvii) node[midway, above]{$\varepsilon$};
        \draw[gray] (-\xii/2, -\xviii) -- (\xi + \xii/2, -\xviii);
      \end{scope} 
	}  
	\foreach \x in {\xi, 2*\xi + \xii, 3*\xi + 2*\xii +\xiii}{
	  \begin{scope}[shift={(\x, 0)}, scale=1]
        \draw (0, 0) -- (\xii, 0); 
        \draw[dashed] (0, 0) -- (0, \xv);
        \draw[dashed] (\xii, 0) -- (\xii, \xv);
        \draw (0, 0) -- (0, -\xviii/2);
        \draw (\xii, 0) -- (\xii, -\xviii/2);
        \draw[<->] (0, \xvii) -- (\xii, \xvii) node[midway, above]{$\varepsilon^2$};
        \draw[dotted,gray] (\xii/2,0) -- (\xii/2, -\xviii);
        \draw[gray] (\xii/2, -\xviii) -- (\xii/2, -\xv);
      \end{scope} 
	}   
	\foreach \x in {-\xii/2, 4*\xi + 3*\xii +\xiii}{
	  \begin{scope}[shift={(\x, 0)}, scale=1]
        \draw (0, 0) -- (\xii/2, 0); 
        \draw[<->] (0, \xvii) -- (\xii/2, \xvii) node[midway, above]{$\frac{\varepsilon^2}2$};
      \end{scope} 
	} 
	\draw[dashed] (0, 0) -- (0, \xv);
    \draw[dashed] (4*\xi + 3*\xii +\xiii, 0) -- (4*\xi + 3*\xii +\xiii, \xv); 
    \draw (0, 0) -- (0, -\xviii/2);
    \draw (4*\xi + 3*\xii +\xiii, 0) -- (4*\xi + 3*\xii +\xiii, -\xviii/2);
	\draw (2*\xi + 2*\xii, 0) -- (2.5*\xi + 2*\xii, 0);
	\draw[dotted] (2.5*\xi + 2*\xii, 0) -- (2.5*\xi + 2*\xii + \xiii, 0);
	\draw (2.5*\xi + 2*\xii + \xiii, 0) -- (3*\xi + 2*\xii + \xiii, 0);
	\draw[gray] (2*\xi + 1.5*\xii, -\xviii) -- (2.5*\xi + 2*\xii, -\xviii);
	\draw[dotted, gray] (2.5*\xi + 2*\xii, -\xviii) -- (2.5*\xi + 2*\xii + \xiii, -\xviii);
	\draw[gray] (2.5*\xi + 2*\xii + \xiii, -\xviii) -- (3*\xi + 2.5*\xii + \xiii, -\xviii);
	
	\begin{scope}[shift={(-\xii/2, 0)}, scale=1]
	  \draw[very thick] (0,-\xv) -- (0,\xv);
	  \fill[pattern = north east lines] (-\xvi, -\xv) rectangle (0, \xv);
    \end{scope}
	\begin{scope}[shift={(4*\xi + 3.5*\xii + \xiii, 0)}, scale=1, rotate=180]
      \draw[very thick] (0,-\xv) -- (0,\xv);
	  \fill[pattern = north east lines] (-\xvi, -\xv) rectangle (0, \xv);
    \end{scope} 
    
    \draw (0.5*\xi, 0) node[below] {$I_1$};
    \draw (1.5*\xi + \xii, 0) node[below] {$I_2$};
    \draw (3.5*\xi + 3*\xii + \xiii, 0) node[below] {$I_K$};
    \draw (0.5*\xi, -\xviii) node[below] {$J_1$};
    \draw (1.5*\xi + \xii, -\xviii) node[below] {$J_2$};
    \draw (3.5*\xi + 3*\xii + \xiii, -\xviii) node[below] {$J_K$};
    \draw (-\xii/2,-\xv) node[below] {$0$};
    \draw (4*\xi + 3.5*\xii + \xiii,-\xv) node[below] {$1$};
    
\end{tikzpicture} \hspace{10mm}
\caption{Location of the closed intervals $I_k$ of length $\varepsilon$ and intervals $J_k \supset I_k$ of length $\varepsilon (1 + \varepsilon)$.}
\label{fig:partn:domain}
\end{figure}
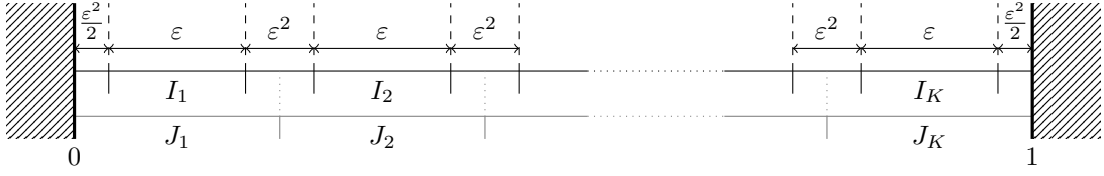

The dense subset consists of all densities $\rho^\pm$ which are piecewise constant on the intervals $I_k$ and $0$ elsewhere, viz.
\begin{equation} \label{for:defn:rhopm:limsupClass}
  \rho^\pm := \sum_{k=1}^K \sigma_k^\pm \indicatornoacc{I_k},
\end{equation}
where the constants $\sigma_k^\pm \geq 0$ satisfy $\varepsilon \sum_{k=1}^K (\sigma^+_k + \sigma^-_k) = 1$, and $K \in \N$. Since $E$ as in \eqref{for:defn:E3} is continuous in $L^2(0,1; \R^2)$, it is enough to show that this subset is dense in $L^2 (0,1; \R^2) \cap M([0,1])$ with respect to the $L^2$-norm. This is straightforward; it is clearly $L^2$-dense in $C ([0,1]; \R^2) \cap M([0,1])$, whose closure in the $L^2$-norm equals $L^2 (0,1; \R^2) \cap M([0,1])$.

It remains to construct $x^n$ for any $\rho^\pm$ as in \eqref{for:defn:rhopm:limsupClass}. For $n \in \N$, we set $\bsigma := \int_0^1 \brho$, note that $\sigma^+ + \sigma^- = 1$, and choose $n^\pm \in \N \cup \{0\}$ such that $| n^\pm - n \sigma^\pm | \leq \tfrac12$ and $n^+ + n^- = n$. We build the recovery sequence in the locally equidistant way: 
$$ \int_0^{x_i^\pm} \rho^\pm (x) \, dx := \frac {i- \tfrac12}n,
\quad \text{for } i = 1,\ldots,n^\pm. $$
We note that $x_i \in \cup_{k=1}^K I_k$ for all $i$, define
\begin{equation*}
  n_k^\pm 
  := \# \{i : x_i^\pm \in I_k\},
\end{equation*}
and relabel the particle positions in $I_k$ as $x_i^{k,\pm}$ for $i = 1,\ldots,n_k^\pm$. We note that 
\begin{equation*}
  \varepsilon \sigma_k^\pm
  \geq \int_{x_1^{k,\pm}}^{x_{n_k^\pm}^{k,\pm}} \rho^\pm (x) \, dx = \frac {n_k^\pm -1}n,
\end{equation*}
and thus $n_k^\pm \leq \varepsilon n \sigma_k^\pm + 1$.

Next we estimate the interaction energy. We start with the interactions between particles of the same type:
\begin{equation} \label{fp:En:exp:L3limp}
  E_n^\pm (x^{n,\pm}) 
  = \frac1{n^2} \sum_{k = 1}^K \sum_{i = 1}^{n_k^\pm}
      \Bigg[ \sum_{ j = 1 }^{i-1} V_n \big( x_i^{k,\pm} - x_j^{k,\pm} \big) + \sum_{j = 1}^{N_k} V_n \big( x_i^{k,\pm} - x_j^{\pm} \big) \Bigg],
\end{equation}
where $N_k := \sum_{\ell = 1}^{k-1} n_\ell^\pm$. For the first term in the right-hand side, we use that $x_i^{k,\pm} - x_j^{k,\pm} = (i-j) / (n \sigma_k^\pm)$, and estimate
\begin{multline*} 
  \frac1{n^2} \sum_{k = 1}^K \sum_{i = 1}^{n_k^\pm} \sum_{ j = 1 }^{i-1} V_n \big( x_i^{k,\pm} - x_j^{k,\pm} \big) 
  \leq \frac1{n^2} \sum_{k = 1}^K \sum_{i = 1}^{n_k^\pm} \sum_{ j = 1 }^\infty V_n \bigg( \frac j{n \sigma_k^\pm} \Big) \\
  = \frac1{n^2} \sum_{k = 1}^K n_k^\pm n \sigma_k^\pm \sum_{ j = 1 }^\infty \frac{\alpha_n}{n \sigma_k^\pm} V \bigg( \frac{ \alpha_n j}{n \sigma_k^\pm} \Big)
  \leq \frac1n \sum_{k = 1}^K (\varepsilon n \sigma_k^\pm + 1) \sigma_k^\pm \bigg( \int_0^\infty V \bigg).
\end{multline*}
Expanding the parenthesis, we obtain from $\sum_{k = 1}^K \sigma_k^\pm = \sigma^\pm / \varepsilon$ that the term related to `$+1$' is of the order of $\frac1n$, which vanishes in the limit $n \to \infty$. For the other term, we observe that
\begin{equation*}
  \bigg( \int_0^\infty V \bigg) \varepsilon \sum_{k = 1}^K (\sigma_k^\pm)^2
  = \bigg( \int_0^\infty V \bigg) \int_0^1 (\rho^\pm)^2,
\end{equation*}
which is independent of $n$. In conclusion, we obtain for the first term in \eqref{fp:En:exp:L3limp} that
\begin{equation*}
  \limsup_{n \to \infty} \frac1{n^2} \sum_{k = 1}^K \sum_{i = 1}^{n_k^\pm} \sum_{ j = 1 }^{i-1} V_n \big( x_i^{k,\pm} - x_j^{k,\pm} \big) 
  \leq \bigg( \int_0^\infty V \bigg) \int_0^1 (\rho^\pm)^2.
\end{equation*}
For the second term in the right-hand side of \eqref{fp:En:exp:L3limp}, we estimate
\begin{equation*}
  x_i^{k,\pm} - x_j^{\pm}
  = \big( x_i^{k,\pm} - x_{N_k}^\pm \big) + \big( x_{N_k}^\pm - x_j^{\pm} \big)
  \geq \varepsilon^2 + \frac{ N_k - j }{ n \| \rho^\pm \|_\infty }.
\end{equation*}
Since $V$ is non-increasing on $(0, \infty)$, we estimate
\begin{align*}
  &\frac1{n^2} \sum_{k = 1}^K \sum_{i = 1}^{n_k^\pm} \sum_{j = 1}^{N_k} V_n \big( x_i^{k,\pm} - x_j^{\pm} \big)
  \leq \frac1{n^2} \sum_{k = 1}^K \sum_{i = 1}^{n_k^\pm} \sum_{j = 0}^{N_k - 1} V_n \bigg( \varepsilon^2 + \frac{ j }{ n \| \rho^\pm \|_\infty } \bigg) \\
  &\leq \frac1{n^2} n^\pm \sum_{j = 0}^\infty n \| \rho^\pm \|_\infty \frac{ \alpha_n }{ n \| \rho^\pm \|_\infty } V \bigg( \alpha_n\varepsilon^2 + \frac{ \alpha_n j }{ n \| \rho^\pm \|_\infty } \bigg)
  \leq \frac{ \alpha_n }n V(\alpha_n \varepsilon^2) + \| \rho^\pm \|_\infty \bigg( \int_{\alpha_n\varepsilon^2}^\infty V \bigg).
\end{align*}
From $1 \ll \alpha_n \ll n$ we observe that the right-hand side converges to $0$ as $n \to \infty$. Reflecting back on \eqref{fp:En:exp:L3limp}, we obtain
\begin{equation*}
  \limsup_{n \to \infty } E_n^\pm (x^{n,\pm}) 
  \leq \bigg( \int_0^\infty V \bigg) \int_0^1 (\rho^\pm)^2.
\end{equation*}

It remains to estimate the interactions between particles of opposite type:
\begin{equation} \label{fp:Enpm:exp:L3limp}
  \frac1{n^2} \sum_{k = 1}^K \sum_{i = 1}^{n_k^+}
      \Bigg[ \sum_{ j = 1 }^{n_k^-} W_n \big( x_i^{k,+} - x_j^{k,-} \big) + \sum_{\substack{ \ell = 1 \\ \ell \neq k }}^K \sum_{ j = 1 }^{n_\ell^-} W_n \big( x_i^{k,+} - x_j^{\ell,-} \big) \Bigg].
\end{equation}
The second term accounts for all interactions between particles that are contained in different intervals $I_k$. Analogously to the case of particles of the same type, we can show that this term vanishes in the limit $n \to \infty$. We skip the details.

Regarding the first term in the right-hand side of \eqref{fp:Enpm:exp:L3limp}, we first estimate $| x_i^{k,+} - x_j^{k,-} |$ from below. For fixed $1 \leq k \leq K$ and $1 \leq i \leq n_k^+$, we set 
\begin{equation*}
  J := \max \{ j : x_j^{k,-} \leq x_i^{k,+} \} \vee 0.
\end{equation*}
Together with $x_{j + \ell}^{k,-} - x_j^{k,-} = \ell/(n \sigma_k^-)$, we obtain
\begin{equation*}
  | x_i^{k,+} - x_j^{k,-} |
  \leq \frac1{n \sigma_k^-} \left\{ \begin{aligned}
    & J - j
    &&\text{if } j \leq J, \\
    & j - (J + 1)
    &&\text{if } j \geq J + 1.
  \end{aligned} \right.
\end{equation*}
Then, since $W$ in non-increasing on $(0, \infty)$, we obtain, similarly to the case of particles of the same type,
\begin{align*}
  &\frac1{n^2} \sum_{k = 1}^K \sum_{i = 1}^{n_k^+} \sum_{ j = 1 }^{n_k^-} W_n \big( x_i^{k,+} - x_j^{k,-} \big)
  \leq \frac1{n^2} \sum_{k = 1}^K 2 n_k^+ \sum_{ j = 0 }^{\infty} W_n \bigg( \frac j{n \sigma_k^-} \bigg) \\
  &= \frac2{n^2} \sum_{k = 1}^K n_k^+ \bigg[ W_n(0) + \sum_{ j = 1 }^{\infty} n\sigma_k^- \frac{\alpha_n}{n \sigma_k^-} W \bigg( \frac{\alpha_n j}{n \sigma_k^-} \bigg) \bigg]
  \leq \frac2n \sum_{k = 1}^K (\varepsilon n \sigma_k^+ + 1 ) \bigg( \frac{\alpha_n}n W(0) + \sigma_k^- \int_0^\infty W \bigg) \\
  &= 2 \varepsilon \sum_{k = 1}^K \sigma_k^+ \sigma_k^- \bigg( \int_0^\infty W \bigg) + \mathcal O \Big( \frac{\alpha_n}n \Big)
  = \bigg( \int_0^\infty W \bigg) \int_0^1 2 \rho^+ \rho^- + \mathcal O \Big( \frac{\alpha_n}n \Big).
\end{align*}
Hence, the $\limsup_{n\to \infty}$ of \eqref{fp:Enpm:exp:L3limp} is bounded from above by $( \int_0^\infty W ) \int_0^1 2 \rho^+ \rho^-$, which completes the proof of \eqref{for:Gconv:limsup}.
\end{proof}

\subsection{The case $\tfrac1n \alpha_n \to \alpha > 0$}

We use the description of $x^{n,\pm}$ in terms of $x^n$ and $b^n$ as introduced above \eqref{fd:GFn}. Defining
\begin{equation*}
  V_{ij} := \left\{ \begin{aligned}
      &V
      &&\text{if } b_i b_j = 1, \\
      &W
      &&\text{if } b_i b_j = -1,
    \end{aligned} \right. 
\end{equation*}
the expression for the energy $E_n$ in \eqref{for:defn:En} can be written compactly as
  \begin{align*}
    E_n ( x^n ) 
    = \frac1{n^2} \sum_{i=1}^n \sum_{ j = 1  }^{i-1} \alpha_n V_{ij} \big( \alpha_n ( x_i - x_j ) \big).
  \end{align*}
We switch between these different descriptions whenever convenient.

\begin{ass}[Properties of $V$ and $W$ in case $\alpha_n / n \to \alpha > 0$] \label{ass:VW:L4}
$V$ and $W$ satisfy
\begin{enumerate}[(i)]
  \item $V : \R \setminus \{0\} \to [0, \infty)$ is even, lower semi-continuous on $\R$ with $V(0) = \infty$, non-increasing on $(0, \infty)$, and satisfies $\int_1^\infty V < \infty$; \label{ass:VW:L4:V}
  \item $W \in L^\infty (\R) \cap L^1(\R)$ is even, and non-increasing on $(0, \infty)$;
  \item $V \geq W$. \label{ass:VW:L4:VW}
\end{enumerate}
\end{ass}

\begin{rem}[Consequences of Assumption \ref{ass:VW:L4}] The monotonicity and integrability of $W$ implies that $W \geq 0$. We further note that $x \Veff (x)$ and $x \Weff (x)$ are Riemann lower-sums for $\int_0^\infty V$ and $\int_0^\infty W$ respectively. Hence, 
\begin{equation*}
  \Veff, \Weff \in L^\infty (\delta, \infty)
  \quad \text{for any } \delta > 0.
\end{equation*}
Furthermore, by the monotonicity of $V$, $V |_{(0,\infty)}$ has a pseudo-inverse $V^{-1} : (0, \infty) \to (0, \infty)$, which has finite integral on $(0,M)$ for any $M > 0$. We obtain
\begin{equation} \label{for:rem:Veff:L4}
  x \Veff (x)
  \leq x V (x) + \int_x^\infty V
  = \int_0^{V(x)} V^{-1}
  \xto{x \to \infty} 0.
\end{equation}
\end{rem}

In a similar spirit as in \cite{BraidesGelli04}, the $\Gamma$-limit of the interactions is determined implicitly through a cell energy density $\eita : [0,\infty)^2 \to [0, \infty)$. We define $\psi$ as
\begin{subequations} \label{for:defn:eita:and:emita}
\begin{align} \label{for:defn:eita}
   \eita (\sigma^+, \sigma^-) 
   &:= \lim_{m \to \infty} \emita (\sigma^+, \sigma^-), \\\notag
        \emita (\sigma^+, \sigma^-)
   &:= 0 \vee \min \bigg\{ 
        \frac1m \sum_{i=1}^{\tilde n} \ \sum_{ j = 1 }^{i-1}
        \alpha V_{ij} \big( \alpha m [ y_i - y_j ] \big) 
        \ : \ \tilde n := \tilde n^+ + \tilde n^-, \ \tilde n^\pm := \lfloor \sigma^\pm m \rfloor \\\label{for:defn:emita}
   &\qquad \qquad \qquad 0 \leq y_1^+ \leq \ldots \leq y_{\tilde n^+}^+ \leq 1, \ 
      0 \leq y_1^- \leq \ldots \leq y_{\tilde n^-}^- \leq 1 \bigg\}, 
\end{align} 
\end{subequations}
where $m$ is allowed to be any positive real. Lemma \ref{lem:props:eita} guarantees that the limit in \eqref{for:defn:eita} exists, and provides further properties of $\emita$ and $\eita$ that are essential for our proof of Theorem \ref{thm:main}. The proof of Lemma \ref{lem:props:eita} relies on the $\Gamma$-liminf inequality of \cite[Thm.~8]{GeersPeerlingsPeletierScardia13} for particles of the same type with convex interaction potential:

\begin{lem}[Liminf inequality of $E_n^\pm$ {\cite[Thm.~8]{GeersPeerlingsPeletierScardia13}}] \label{lem:Gconv:single:L4}
Let $\tfrac1n \alpha_n \to 1$ and let $\mathsf V : \R \setminus \{0\} \to [0, \infty)$ be even on $\R$, convex on $(0, \infty)$, $\int_1^\infty \mathsf V < \infty$ and $\mathsf V(x) \to \infty$ if $x \to 0$. Then, for all $\mu_n^\pm \weakto \mu^\pm$ it holds that
  \begin{equation*}
    \liminf_{n\to \infty} E_n^\pm (\mu_n^\pm) 
    \geq \int_0^1 \mathsf V_{\operatorname{eff}} \bigg( \frac{1}{\rho^\pm (x)} \bigg) \rho^\pm (x) \, dx,
  \end{equation*} 
  where the right-hand side is defined as $\infty$ if $\mu^\pm$ is not absolutely continuous (cf.~\eqref{for:defn:abs:cont}).
\end{lem}

While \cite{GeersPeerlingsPeletierScardia13} focuses on $\mu^\pm$ with mass $1$, a simple scaling argument as in the proof of Lemma \ref{lem:Gconv:single} implies that Lemma \ref{lem:Gconv:single:L4} also holds for any measure $\mu^\pm \in \mathcal M_+ ([0,1])$ and any approximating sequence $\mu_n^\pm$ with possibly different mass than $\mu^\pm$.

In Lemma \ref{lem:Gconv:single:L4} and Lemma \ref{lem:props:eita} we set $\alpha = 1$, for the simple reason that Assumption \ref{ass:VW:L4} is invariant under the rescaling $V_\alpha = \alpha V (\alpha \, \cdot \,)$ and $W_\alpha = \alpha W (\alpha \, \cdot \,)$.

\begin{lem} [Properties of $\emita$ and $\eita$] \label{lem:props:eita} 
Set $\alpha = 1$. For any $\bsigma = (\sigma^+, \sigma^-) \in [0,\infty)^2$, $\einfita (\bsigma) := \eita (\bsigma)$ defined by \eqref{for:defn:eita} is well-defined. Moreover, for any $m \in [0, \infty]$
\begin{enumerate}[(i)]
  \item \label{lem:props:eita:sym} $\emita (\sigma^+, \sigma^-) = \emita (\sigma^-, \sigma^+)$;
  \item \label{lem:props:eita:incr} $\sigma \mapsto \emita (\sigma, \sigma^-)$ is non-decreasing;
  \item \label{lem:props:eita:UB} $\emita (\bsigma) 
         \leq \emita(\sigma^+ + \sigma^-, 0)$;
  \item \label{lem:props:eita:LB} $\emita (\bsigma) 
         \geq \emita (\sigma^+,0) + \emita (0, \sigma^-)$;  
  \item \label{lem:props:eita:bds} $\sigma \Veff (1/\sigma) \geq \eita (\sigma, 0) \geq c (\sigma^2 - C)$ for some $C, c > 0$ independent of $\sigma$;
  \item \label{lem:props:eita:cty} $\eita \in C([0, \infty)^2)$;
  \item \label{lem:props:eita:cont:conv} $\emita$ converges continuously to $\eita$ as $m \to \infty$, i.e., for all $\sigma^\pm \geq 0$ and all $0 \leq \sigma_m^\pm \to \sigma^\pm$ it holds that $\emita (\sigma_m^+, \sigma_m^-) \to \eita (\sigma^+, \sigma^-)$.
\end{enumerate}

\end{lem}

\begin{proof}
We note that \eqref{lem:props:eita:sym}--\eqref{lem:props:eita:LB} are basic observations, relying only on $0 \leq W \leq V$ being even functions. 

As a preliminary step to proving that the limit in \eqref{for:defn:eita} exists, we show that
\begin{align} \label{fp:UBlim}
  \emita (\sigma, 0) 
  &\leq \sigma \Veff (1/\sigma)
  && \text{for all } m \geq \tfrac1\sigma, \\\label{fp:LBlim}
  \liminf_{m \to \infty} \emita (\sigma, 0) 
  &\geq c (\sigma^2 - C) 
  && \text{for some $C, c > 0$ independent of $\sigma$,}
\end{align}
which together imply \eqref{lem:props:eita:bds}. We establish \eqref{fp:UBlim} by bounding from above the minimisation problem in \eqref{for:defn:emita}, given by $\emita (\sigma, 0)$, by the equidistant configuration $z_i := (i-1)/(\tilde n - 1)$, where $\tilde n := \lfloor \sigma m \rfloor$. We obtain
\begin{equation*}
  \emita (\sigma, 0)
  \leq \frac1m \sum_{k=1}^{\tilde n} \sum_{j=1}^{k-1} V( m (z_{k+j} - z_j) )  
  \leq \frac{\tilde n - 1}m \sum_{k=1}^{\tilde n} V \Big( \frac m{\tilde n - 1} k \Big) 
  \leq \frac{\tilde n - 1}m \Veff \Big( \frac m{\tilde n - 1} \Big).
\end{equation*}  
Since $V$ is non-increasing on $(0,\infty)$, $\Veff$ is also non-increasing on $(0,\infty)$, and thus \eqref{fp:UBlim} follows by using $(\tilde n - 1)/m < \sigma$. To prove \eqref{fp:LBlim}, we set $V^{**}$ as the even extension of the convex envelope of $V$ on $(0,\infty)$. Applying Lemma \ref{lem:Gconv:single:L4} with $V^{**}$, we find
\begin{equation*}
  \liminf_{m \to \infty} \emita (\sigma, 0)
  \geq \inf \bigg\{ \int_0^1 \Veffce \bigg( \frac{1}{\rho (x)} \bigg) \rho(x) \, dx : \rho \in L^1_+ (0,1) \: \text{with} \: \int_0^1 \rho = \sigma \bigg\}.
\end{equation*}
Since $V^{**}$ is convex, $r \mapsto r \Veffce ( \tfrac{1}{r} )$ is convex, and thus it follows from Jensen's Inequality that
\begin{equation*}
  \liminf_{m \to \infty} \emita (\sigma, 0)
  \geq \sigma \Veffce \Big( \frac1\sigma \Big).
\end{equation*}
Using that $V$ is non-increasing, we find $\int_0^\infty V^{**} \geq \frac12 \int_0^\infty V > 0$. Since $\frac1\sigma \Veffce ( \frac1\sigma )$ is a Riemann lower-sum of $V^{**}$, it holds that $\lim_{\sigma \to \infty} \frac1\sigma \Veffce ( \frac1\sigma ) \geq \tfrac12 \int_0^\infty V \in (0, \infty]$, and thus there exists a $c > 0$ such that 
\begin{equation} \label{for:pf:lem:props:eita:1}
  \sigma \Veffce \Big( \frac1\sigma \Big) \geq c\sigma^2 
  \quad \text{for all $\sigma$ large enough.} 
\end{equation}
We conclude \eqref{fp:LBlim}.

The remainder of the proof of Lemma \ref{lem:props:eita} concerns \eqref{lem:props:eita:cont:conv}, which we prove in four steps. Step 1 treats the easiest case where $\bsigma = \bzero$. Steps 2 and 3 establish a continuity estimate on $\emita$, uniform in $m$ (see \eqref{for:pf:lem:props:eita:2} and \eqref{for:pf:lem:props:eita:11}). In Step 4 we prove pointwise convergence of $\emita$ to $\eita$ as $m \to \infty$. Steps 1--4 together imply \eqref{lem:props:eita:cont:conv}, and hence \eqref{lem:props:eita:sym}--\eqref{lem:props:eita:LB} also hold for $m = \infty$. \eqref{lem:props:eita:cty} is a corollory of \eqref{lem:props:eita:cont:conv}, because \eqref{lem:props:eita:cont:conv} implies that both $\emita$ and $-\emita$ $\Gamma$-converge to $\eita$, from which we infer that both $\eita$ and $-\eita$ are lower semi-continuous. 
\smallskip

\textit{Step 1: \eqref{lem:props:eita:cont:conv} for $\bsigma = \bzero$.} 
Let $\bsigma_m \to \bzero$ as $m \to \infty$, and let $\varepsilon > 0$ be arbitrary. We set $\sigma_m = \sigma_m^+ + \sigma_m^-$, and consider $m$ large enough such that $\sigma_m < \varepsilon$. Then \eqref{lem:props:eita:incr}, \eqref{lem:props:eita:UB} and \eqref{fp:UBlim} imply that
\begin{equation*}
  \limsup_{m \to \infty} \emita (\bsigma_m) 
  \leq \limsup_{m \to \infty} \emita (\sigma_m, 0)
  \leq \limsup_{m \to \infty} \emita (\varepsilon, 0)
  \leq \varepsilon \Veff ( \tfrac1{\varepsilon} )
\end{equation*}
for $m$ large enough. We conclude \eqref{lem:props:eita:cont:conv} in case $\bsigma = \bzero$ from the arbitrariness of $\varepsilon > 0$ and \eqref{for:rem:Veff:L4}.
 
\textit{Step 2: continuity estimate for $\emita$ at $\sigma^\pm > 0$.} 
In this step we prove the following estimate:
\begin{equation} \label{for:pf:lem:props:eita:2}
  \forall \, \sigma^\pm > 0 \
  \exists \, C > 0 \
  \forall \, \delta > 0 \
  \exists \, M > 0 \
  \forall \, m > M :
  \emita((1+\delta) \bsigma) - \emita((1-\delta) \bsigma) 
  \leq C \delta.
\end{equation}
We fix some notation by writing out \eqref{for:defn:emita} in detail: 
\begin{align} \label{for:pf:lem:props:eita:3}
   &\begin{multlined}[t][0.7 \textwidth] 
   \emita ( (1-\delta) \bsigma )
   = \min \bigg\{ 
        \frac1m \sum_{i > j}^{\tilde n} V_{ij} \big( m [ y_i - y_j ] \big) 
        \ : \ \tilde n := \tilde n^+ + \tilde n^-, \ \tilde n^\pm := \lfloor \sigma^\pm (1 - \delta) m \rfloor \\
      0 \leq y_1^+ \leq \ldots \leq y_{\tilde n^+}^+ \leq 1, \ 
      0 \leq y_1^- \leq \ldots \leq y_{\tilde n^-}^- \leq 1 \bigg\} \end{multlined} \\\label{for:pf:lem:props:eita:4}
   &\begin{multlined}[t][0.6 \textwidth]   
   \emita ( (1+\delta) \bsigma )
   = \min \bigg\{ 
        \frac1m \sum_{i > j}^{\tilde n_\delta} V_{ij} \big( m [ z_i - z_j ] \big) 
        \ : \ \tilde n_\delta := \tilde n_\delta^+ + \tilde n_\delta^-, \ \tilde n_\delta^\pm := \lfloor \sigma^\pm (1 + \delta) m \rfloor \\
      0 \leq z_1^+ \leq \ldots \leq z_{\tilde n_\delta^+}^+ \leq 1, \ 
      0 \leq z_1^- \leq \ldots \leq z_{\tilde n_\delta^-}^- \leq 1 \bigg\}. \end{multlined}
\end{align}
Since $\sigma^\pm > 0$, it holds that $\tilde n^\pm \to \infty$ as $m \to \infty$, and thus for all $m$ large enough we can assume $\tilde n^\pm$ to be large enough in the argument below. We also assume $\delta > 0$ to be small enough, independent of $m$. We fix such $m$, let $y$ be a minimiser of \eqref{for:pf:lem:props:eita:3} satisfying $y_1 \leq \ldots \leq y_{\tilde n}$, and set 
\begin{equation} \label{fp:dcI}
  d_i^\pm := y_{i+1}^\pm - y_i^\pm,
  \quad c_i^\pm := \tfrac12 \big( y_{i+1}^\pm + y_i^\pm \big),
  \quad I^\pm (i) := [y_i^\pm, y_{i+1}^\pm]
  \quad \text{for all } 1 \leq i \leq \tilde n^\pm - 1.
\end{equation}
We construct an admissible vector $z \in \R^{\tilde n_\delta}$ for the minimisation problem in \eqref{for:pf:lem:props:eita:4} by
\begin{equation} \label{for:pf:lem:props:eita:45}
  z_i 
  := \left\{ \begin{aligned}
    &y_i,
    &&1 \leq i \leq \tilde n, \\
    &c_{\ell_i},
    &&\tilde n + 1 \leq i \leq \tilde n_\delta,
  \end{aligned} \right.
\end{equation}
where $1 \leq \ell_i \leq \tilde n -1$ are carefully chosen indices to find sufficient estimates for the remainder terms $\Sigma_1$ and $\Sigma_2$ in the following estimate
\begin{multline} \label{for:pf:lem:props:eita:47}
  \emita((1+\delta) \bsigma) 
  \leq \frac1m \sum_{i > j}^{\tilde n_\delta} V_{ij} \big( m [ z_i - z_j ] \big) \\
  = \underbrace{ \frac1m \sum_{i > j}^{\tilde n} V_{ij} \big( m [ y_i - y_j ] \big) }_{ \emita((1-\delta) \bsigma) }
       + \underbrace{ \frac1m \sum_{ i = \tilde n + 1  }^{ \tilde n_\delta } \sum_{ j = \tilde n + 1 }^{ i-1 } V_{ij} \big( m [ c_{\ell_i} - c_{\ell_j} ] \big) }_{ \Sigma_1 }
       + \underbrace{ \frac1m \sum_{i = 1}^{\tilde n} \sum_{j = \tilde n + 1}^{ \tilde n_\delta } V_{i j} \big( m [ y_i - c_{\ell_j} ] \big) }_{ \Sigma_2 }.
\end{multline}

Next we construct the indices $\ell_i$ such that $\Sigma_1 + \Sigma_2 < C \delta$. To this aim, we put three conditions on $\ell_i$. For convenience, we introduce the indices $\ell_i^\pm$ by the same change of variables which transforms $y, b$ into $y^+, y^-$. We also introduce the index shift $\kappa_i \geq 1$, which characterises $y_{i + \kappa_i}$ as the next particle with the same sign as $y_i$.

The first condition on $\ell_i^\pm$ ensures $d^\pm_{\ell_i^\pm}$ to be large enough. Let $s$ be a permutation such that the interdistances $d_i^\pm$ satisfy $d_{s(1)}^\pm \leq \ldots \leq d_{s( \tilde n^\pm - 1 )}^\pm$. We set 
\begin{equation} \label{for:pf:lem:props:eita:48}
\tilde n_4^\pm := \big\lfloor \tfrac14 (\tilde n^\pm - 1) \big\rfloor, \quad
d_*^\pm := d_{s(\tilde n_4^\pm)}^\pm,  
\end{equation}
and estimate from below
\begin{multline} \label{for:pf:lem:props:eita:49}
  \emita((1-\delta) \bsigma)
  \geq \frac1m \sum_{i = 1}^{\tilde n_4^+} V \big( m d_{s(i)}^+ \big)
  \geq \frac{ \tilde n^+ - 4 }{ 4m } V \Big( m d_{s(\tilde n_4^+)}^+ \Big)  \\
  = \frac{ \lfloor \sigma^+ (1 - \delta) m \rfloor - 4 }{ 4m } V \big( m d_*^+ \big)
  \geq \frac{ \sigma^+ }{ 5 } V \big( m d_*^+ \big)
\end{multline}
and from above (relying on \eqref{lem:props:eita:incr}, \eqref{lem:props:eita:UB} and \eqref{fp:UBlim})
\begin{equation} \label{for:pf:lem:props:eita:5}
  \emita((1-\delta) \bsigma)
  \leq \emita(\bsigma)
  \leq \emita(\sigma, 0)
  \leq \sigma \Veff ( \tfrac1\sigma ),
\end{equation}
where $\sigma = \sigma^+ + \sigma^-$. We obtain that $V \big( m d_*^+ \big) \leq \tfrac{5\sigma}{\sigma^+} \Veff ( \tfrac1\sigma )$, and thus $d_*^+ \geq \frac cm$ for some constant $c > 0$ which is independent of $\delta$ and $m$. Since $d_{s(i)}^+$ is ordered in $i$, we finally obtain
\begin{equation} \label{for:pf:lem:props:eita:6}
  d_\ell^+ \geq d_*^+ \geq \tfrac cm
  \quad \text{for all } \ell \in J_1^+ := \{ s^{-1} (i) : \tilde n_4^+ \leq i \leq \tilde n^+ - 1 \}.
\end{equation}
An analogous argument for the negative particles yields
\begin{equation} \label{for:pf:lem:props:eita:7}
  d_\ell^- \geq d_*^- \geq \tfrac cm
  \quad \text{for all } \ell \in J_1^- := \{ s^{-1} (i) : \tilde n_4^- \leq i \leq \tilde n^- - 1 \}
\end{equation}
for some (possibly different) permutation $s$ and constant $c > 0$ which is independent of $\delta$ and $m$.

The second condition on the indices $\ell_i^\pm$ is that the following quantity, which is part of $\Sigma_2$, is bounded uniformly in $i$, $\delta$ and $m$:
\begin{equation*}
  \sum_{j = 1}^{\ell_i - 1} V_{\ell_i j} \big( m [ y_{\ell_i} - y_j ] \big)
  + \sum_{j = \ell_i + \kappa_{\ell_i} + 1}^{\tilde n} V_{(\ell_i + \kappa_{\ell_i})j} \big( m [ y_{\ell_i + \kappa_{\ell_i}} - y_j ] \big).
\end{equation*}
We establish the related index sets $J^\pm_2$ by a similar argument to the one leading to $J^\pm_1$. The main difference is the following bound from below, which follows simply by neglecting several interactions between particles:
\begin{equation*}
  \emita((1-\delta) \bsigma)
  \geq \frac1{2m} \sum_{i=1}^{\tilde n} \bigg[ \sum_{j = 1}^{i - 1} V_{i j} \big( m [ y_i - y_j ] \big) + \sum_{j = i + \kappa_i + 1}^{\tilde n} V_{(i + \kappa_i)j} \big( m [ y_{i + \kappa_i} - y_j ] \big) \bigg].
\end{equation*}
Then, by introducing permutations $s_\pm$ we can order the summands from high to low values (for the positive and negative particles separately), and estimate the highest $\frac14$-fraction of them by the constant given by the right-hand side of \eqref{for:pf:lem:props:eita:5} to conclude that
\begin{multline} \label{for:pf:lem:props:eita:8}
  \sum_{j = 1}^{\ell - 1} V_{\ell j} \big( m [ y_{\ell} - y_j ] \big)
  + \sum_{j = \ell + \kappa_\ell + 1}^{\tilde n} V_{(\ell + \kappa_\ell)j} \big( m [ y_{\ell + \kappa_\ell} - y_j ] \big) \leq C \\
  \text{for all } \ell \in J_2^\pm := \{ s_\pm^{-1} (i) : \tilde n_4^\pm \leq i \leq \tilde n^\pm \}
\end{multline}
for some constant $C > 0$ which is independent of $\delta$ and $m$.

The third condition on the indices $\ell_i^\pm$ is that the interval $I^\pm (\ell_i^\pm)$ (see \eqref{fp:dcI}) does not contain too many particles of the opposite sign. Let $N^- (i)$ be the number of negative particles in $I^+ (i)$ for $1 \leq i \leq \tilde n^+ - 1$, and $s_+$ be the permutation for which $N^- (s_+(1)) \geq \ldots \geq N^- (s_+(\tilde n^+ - 1))$. Then,
\begin{equation*}
  2 \tilde n^- 
  \geq \sum_{i=1}^{\tilde n^+ - 1} N^- (i)
  \geq \tilde n_4^+ N^- (s_+ (\tilde n_4^+)),
\end{equation*}
where the factor $2$ covers all negative particles located at any of the endpoint of $I^+(i)$, which are counted twice in the sum above. It follows that $N^- (s_+ (\tilde n_4^+)) \leq K^-$ for some $K^- \in \N$ independent of $\delta$ and $m$. An analogous argument for the positive particles yields $N^+ (s_- (\tilde n_4^-)) \leq K^+$ for some $K^+ \in \N$. We conclude that
\begin{equation} \label{for:pf:lem:props:eita:85}
  N^\mp (\ell) \leq K
  \quad \text{for all } \ell \in J_3^\pm := \{ s_\pm^{-1} (i) : \tilde n_4^\pm \leq i \leq \tilde n^\pm \}
\end{equation}
for some constant $K \in \N$ which is independent of $\delta$ and $m$.

We finally construct the set of indices
\begin{equation*}
  J^\pm := J_1^\pm \cap J_2^\pm \cap J_3^\pm.
\end{equation*}
Since $J_1^\pm$, $J_2^\pm$ and $J_3^\pm$ contain $\lfloor \tfrac34 \tilde n^\pm \rfloor$ or more indices, $J^\pm$ contains at least $\lfloor \tilde n^\pm / 5 \rfloor$ indices, which is enough to choose all the centre points $c_{\ell_i}$ in \eqref{for:pf:lem:props:eita:45} differently from each other. Moreover, we use the freedom in this choice to take $\ell_i$ increasing in $i$. As a consequence of \eqref{for:pf:lem:props:eita:6} and \eqref{for:pf:lem:props:eita:7}, we obtain
\begin{equation} \label{for:pf:lem:props:eita:9}
  \min \{ |c_k^\pm - c_\ell^\pm | : k, \ell \in J^\pm, \ k \neq \ell \} \geq \tfrac cm.
\end{equation}

Together with the related properties \eqref{for:pf:lem:props:eita:6}--\eqref{for:pf:lem:props:eita:9}, we estimate the sums $\Sigma_1$ and $\Sigma_2$ defined in \eqref{for:pf:lem:props:eita:47}. We expand
\begin{multline} \label{for:pf:lem:props:eita:10}
  \Sigma_1
  = \frac1m \sum_{k=1}^{ \tilde n_\delta^+ - \tilde n^+ - 1 } \sum_{ j = \tilde n^+ + 1 }^{ \tilde n_\delta^+ - k } V \Big( m \Big[ c_{\ell_{k+j}^+}^+ - c_{\ell_j^+}^+ \Big] \Big)
    + \frac1m \sum_{k=1}^{ \tilde n_\delta^- - \tilde n^- - 1 } \sum_{ j = \tilde n^- + 1 }^{ \tilde n_\delta^- - k } V \Big( m \Big[ c_{\ell_{k+j}^-}^- - c_{\ell_j^-}^- \Big] \Big) \\
    + \frac1m \sum_{ i = \tilde n^+ + 1 }^{ \tilde n_\delta^+ } \sum_{ j = \tilde n^- + 1 }^{ \tilde n_\delta^- } W \Big( m \Big[ c_{\ell_i^+}^+ - c_{\ell_j^-}^- \Big] \Big).
\end{multline}
Using \eqref{for:pf:lem:props:eita:9} and $V$ being decreasing on $(0,\infty)$, we estimate the first sum in the right-hand side by
\begin{align*}
  \frac1m \sum_{k=1}^{ \tilde n_\delta^+ - \tilde n^+ - 1 } \sum_{ j = \tilde n^+ + 1 }^{ \tilde n_\delta^+ - k } V \Big( m \Big[ c_{\ell_{k+j}^+}^+ - c_{\ell_j^+}^+ \Big] \Big)
  \leq \frac1{m} \sum_{k=1}^\infty \sum_{ j = \tilde n^+ + 1 }^{ \tilde n_\delta^+ } V \big( m [ k \tfrac cm ] \big)
  = \frac{\tilde n_\delta^+ - \tilde n^+}{m} \Veff(c)
  \leq C \delta.
\end{align*}
The same argument for the negative particles yields the same estimate. We estimate the third sum in the right-hand side of \eqref{for:pf:lem:props:eita:10} by
\begin{multline*}
  \frac1m \sum_{ i = \tilde n^+ + 1 }^{ \tilde n_\delta^+ } \sum_{ j = \tilde n^- + 1 }^{ \tilde n_\delta^- } W \Big( m \Big[ c_{\ell_i^+}^+ - c_{\ell_j^-}^- \Big] \Big)
  \leq \frac1m \sum_{ i = \tilde n^+ + 1 }^{ \tilde n_\delta^+ } \sum_{ k = 0 }^{ \infty } 2 W \big( m [ k \tfrac cm ] \big) \\
  = 2 \frac{\tilde n_\delta^+ - \tilde n^+}{m} \big( W(0) + \Weff(c) \big)
  \leq C \delta, 
\end{multline*}
and conclude that $\Sigma_1 < C \delta$ for a $\delta$- and $m$-independent constant $C$.

To estimate $\Sigma_2$, we recall that $z_j^\pm = c_{\ell_j^\pm}^\pm$ is the midpoint of the interval $I^\pm (\ell_j^\pm)$, and split the interactions of $y_i$ with $z_j$ for $y_i \notin I^\pm (\ell_j^\pm)$ and $y_i \in I^\pm (\ell_j^\pm)$. Then, we use \eqref{for:pf:lem:props:eita:8} to estimate the interactions with $y_i \notin I^\pm (\ell_j^\pm)$, and \eqref{for:pf:lem:props:eita:85} for those with $y_i \in I^\pm (\ell_j^\pm)$. This yields 
\begin{align*}
  \Sigma_2 
  &= \frac1m \sum_{j = \tilde n + 1}^{ \tilde n_\delta } \Bigg[ 
      \sum_{i : y_i \notin I^\pm (\ell_j^\pm) } V_{i \ell_j} \big( m [ y_i - c_{\ell_j} ] \big) 
      + \sum_{i : y_i \in I^\pm (\ell_j^\pm) } V_{i \ell_j} \big( m [ y_i - c_{\ell_j} ] \big) 
    \Bigg] \\
  &\leq \begin{multlined}[t][0.87 \textwidth]
  	  \frac1m \sum_{j = \tilde n + 1}^{ \tilde n_\delta } \Bigg[ 
      \sum_{i = 1}^{\ell_j - 1} V_{i \ell_j} \big( m [ y_i - y_{\ell_j} ] \big)
      + \sum_{i = \ell_i + \kappa_{\ell_i} + 1}^{\tilde n} V_{i (\ell_i + \kappa_{\ell_i})} \big( m [ y_i - y_{\ell_i + \kappa_{\ell_i}} ] \big) \\
      + \sum_{i : y_i \in I^\pm (\ell_j^\pm) } V_{i \ell_j} \big( m [ y_i - c_{\ell_j} ] \big) 
    \Bigg] \end{multlined} \\
  &\leq \frac1m \sum_{j = \tilde n + 1}^{ \tilde n_\delta } \big[ 
      C
      + V \big( m [ c_{\ell_j} - y_{\ell_j} ] \big)
      + V \big( m [ c_{\ell_j} - y_{\ell_i + \kappa_{\ell_i}} ] \big)
      + K W(0) \big] \\
  &\leq \frac{ \tilde n_\delta - \tilde n }m \big[ 
      C
      + 2 V \big( m [ \tfrac c{2m} ] \big)
      \big]
  \leq C \delta.
\end{align*}
This concludes the proof for $\Sigma_1 + \Sigma_2 \leq C \delta$, which by \eqref{for:pf:lem:props:eita:47} implies \eqref{for:pf:lem:props:eita:2}. 
\smallskip

\textit{Step 3: continuity estimate for $\emita$ at $\sigma^+ \wedge \sigma^- = 0$.}
We establish a similar estimate as \eqref{for:pf:lem:props:eita:2} in the case when $\sigma^+ \wedge \sigma^- = 0$. By (i) it is enough to prove continuity at the $\sigma^+$-axis, and by Step 1 we can further assume $\sigma := \sigma^+ > 0$. This motivates us to prove
\begin{equation} \label{for:pf:lem:props:eita:11}
  \forall \, \sigma > 0 \
  \exists \, C > 0 \
  \forall \, \delta > 0 \
  \exists \, M > 0 \
  \forall \, m > M :
  \emita((1+\delta) \sigma, \delta \sigma) - \emita((1-\delta) \sigma, 0) 
  \leq C \delta.
\end{equation}

Since \eqref{lem:props:eita:UB} implies that 
\begin{equation*}
  \emita((1+\delta) \sigma, \delta \sigma) - \emita((1-\delta) \sigma, 0)
  \leq \emita((1+2\delta) \sigma, 0) - \emita((1-\delta) \sigma, 0),
\end{equation*}
the argument in Step 2 (simplified to $n^- = 0$) yields \eqref{for:pf:lem:props:eita:11}.
\smallskip

\textit{Step 4: Pointwise convergence of $\emita$ to $\eita$.}
We prove that the point-wise limit of $\emita (\bsigma)$ exists as $m \to \infty$ for all $\bsigma \neq \bzero$. Since $\emita (\bsigma) \geq 0$, it is enough to show that
\begin{equation} \label{for:pf:lem:props:eita:12}
  \forall \, \sigma^\pm > 0 \  
  \exists \, C > 0 \  
  \forall \, \varepsilon > 0 \
  \exists \, L > 0 \
  \forall \, \ell \geq L \
  \exists \, M > 0 \
  \forall \, m \geq M :
  \emita (\bsigma) - \elita (\bsigma) < C \varepsilon.
\end{equation}
Indeed, it is easy to see that \eqref{for:pf:lem:props:eita:12} implies that the sequence $(\emita (\bsigma))_m$ is bounded in $m$ (set $\varepsilon = 1$, and choose $\ell = L$; then $\emita (\bsigma) \leq \elita (\bsigma) + C$), and that $(\emita (\bsigma))_m$ can have at most one accumulation point. Therefore, \eqref{for:pf:lem:props:eita:12} implies that $(\emita (\bsigma))_m$ is convergent.

To prove \eqref{for:pf:lem:props:eita:12}, we fix any $\bsigma \neq \bzero$, and take any $0 < \varepsilon < \tfrac 12$ small enough such that either \eqref{for:pf:lem:props:eita:2} or \eqref{for:pf:lem:props:eita:11} applies with $\delta = \varepsilon$. We choose $L$ such that
\begin{equation} \label{for:pf:lem:props:eita:125}
  \sigma L \geq \frac{12}\varepsilon,
  \quad \max_{\ell \geq L} \ell \Veff (\ell) < \varepsilon,
  \quad \max_{\ell \geq \varepsilon L} \ell V (\ell) < \varepsilon^2,
\end{equation}
where $\sigma = \sigma^+ + \sigma^-$. We take any $\ell \geq L$, set $n_\ell := \lfloor \sigma^+ \ell \rfloor + \lfloor \sigma^- \ell \rfloor$ and $y \in [0,1]^{n_\ell}$ as a minimiser of $\elita (\bsigma)$. We choose $M$ such that $M \geq 3 \ell / \varepsilon$ and such that for any $m \geq M$, it holds that
\begin{equation} \label{for:pf:lem:props:eita:13}
  \emita(\bsigma) - \emita( (1 - \varepsilon) \bsigma) 
  \leq C \varepsilon.
\end{equation}
The existence of such $M$ is guaranteed by \eqref{for:pf:lem:props:eita:2} or \eqref{for:pf:lem:props:eita:11}. We take any $m > M$, and observe from \eqref{for:pf:lem:props:eita:13} that \eqref{for:pf:lem:props:eita:12} holds if
\begin{equation} \label{for:pf:lem:props:eita:135}
   \emita ((1 - \varepsilon)\bsigma) - \elita (\bsigma) < C \varepsilon
\end{equation} 
for some $C$ which only depends on $\bsigma$. 

Next we construct an admissible vector $z$ for the minimisation problem given by $\emita ((1 - \varepsilon)\bsigma)$. Such vector should be orderd, have at least 
$$
  n_m := \lfloor (1 - \varepsilon) \sigma^+ m \rfloor + \lfloor (1 - \varepsilon) \sigma^- m \rfloor
$$
entries, $z_1 \geq 0$ and the last entry should be smaller than or equal to $1$. We construct such $z$ by concatenating $N := \lceil n_m/n_\ell \rceil$ scaled copies of the minimiser $y$ of $\elita (\bsigma)$, including a small gap between any consecutive copies;
\begin{equation*}
  z 
  := \frac \ell m \big( y, \,
     y + (1 + \tfrac\varepsilon3), \,
     y + 2(1 + \tfrac\varepsilon3),
     \ldots, \,
     y + (N-1)(1 + \tfrac\varepsilon3)
     \big).
\end{equation*}
To show that the final entry satisfies $z_{N n_\ell} \leq 1$, we first use $\frac{4}{\sigma \ell} \leq \tfrac\varepsilon3$ to estimate
\begin{equation} \label{for:pf:lem:props:eita:14}
  N - 1
  \leq \frac{n_m}{n_\ell}
  \leq \frac{ (1 - \varepsilon) \sigma m }{ \sigma \ell - 2 }
  \leq \frac m\ell \frac{1 - \varepsilon}{1 - \tfrac2{\sigma \ell} }
  \leq \frac m\ell (1 - \varepsilon) \Big(1 + \frac4{\sigma \ell} \Big)
  \leq \frac m\ell (1 - \varepsilon) (1 + \tfrac\varepsilon3 ).
\end{equation}
Then, we obtain by $\tfrac \ell m \leq \tfrac\varepsilon3$ that
\begin{equation*}
  z_{N n_\ell} 
  \leq \tfrac\ell m \big( 1 + (N-1)(1 + \tfrac\varepsilon3) \big)
  \leq \tfrac\varepsilon3 + (1 - \varepsilon) (1 + \tfrac\varepsilon3 )^2
  < 1.
\end{equation*}

We motivate our choice of $z$ as follows. Each scaled copy of $y$ has interaction energy $\frac1N \elita (\bsigma) (1 + \mathcal O (\varepsilon))$. The gaps between neighbouring  copies of $y$ allow us to estimate the interdistance (and hence the interaction energy) of any two particles within these copies. For any other pair of particles, we use the number of copies in between them to estimate their interaction energy. More precisely, we estimate
\begin{multline} \label{for:pf:lem:props:eita:15}
  \emita ((1 - 4 \varepsilon)\bsigma)
  \leq \frac1m \sum_{i > j}^{n_m} V_{ij} (m (z_i - z_j)) \\
  \leq N \frac{\ell}m \frac1{\ell} \sum_{i > j}^{n_\ell} V_{ij} (m \tfrac \ell m (y_i - y_j)) 
       + \frac1m \sum_{k = 1}^{N-1} \sum_{j=1}^{N-k} n_\ell^2 V \big( m \tfrac \ell m [ \varepsilon + (k-1) ] \big).
\end{multline}
We estimate both terms in the right-hand side of \eqref{for:pf:lem:props:eita:15} separately. Using \eqref{for:pf:lem:props:eita:14} and $\tfrac \ell m \leq \tfrac\varepsilon3$, we have that $N \frac{\ell}m < 1$, and thus
\begin{equation*}
  N \frac{\ell}m \frac1{\ell} \sum_{i > j}^{n_\ell} V_{ij} (\ell (y_i - y_j)) \leq \elita (\bsigma),
\end{equation*}
where we have used that $y$ is a minimiser of $\elita (\bsigma)$. Regarding the second term in the right-hand side of \eqref{for:pf:lem:props:eita:15}, the summand is independent of $j$, and thus we can estimate the sum over $j$ from above by multiplication with $N-1$. Then, estimating the constant in front of the summation over $k$ by
\begin{equation*}
  (N-1) n_\ell^2
  \leq n_m n_\ell
  \leq \big( (1 - \varepsilon) \sigma m \big) ( \sigma \ell ) 
  \leq \sigma^2 m \ell,
\end{equation*}
we estimate the second term in the right-hand side of \eqref{for:pf:lem:props:eita:15} by
\begin{multline*}
  \frac1m \sum_{k = 1}^{N-1} \sum_{j=1}^{N-k} n_\ell^2 V \big( m \tfrac \ell m [ \varepsilon + (k-1) ] \big)
   \leq \frac{N-1}m n_\ell^2 \bigg( V (\ell \varepsilon )
       + \sum_{k = 2}^{N-1} V (\ell (k-1) ) \bigg) \\
  \leq \sigma^2 \big( \tfrac1\varepsilon \ell \varepsilon V (\ell \varepsilon ) + \ell \Veff (\ell) \big),
\end{multline*}
which, by our choice of $L$ in \eqref{for:pf:lem:props:eita:125}, is bounded by $C \varepsilon$. Collecting our estimates on the right-hand side of \eqref{for:pf:lem:props:eita:15}, it follows directly that \eqref{for:pf:lem:props:eita:135} holds, which completes the proof of \eqref{for:pf:lem:props:eita:12}.
\end{proof}

Before proving $\Gamma$-convergence in Theorem \ref{thm:Gconv:L4}, we cite a standard property of Lebesgue points and introduce the dual bounded Lipschitz norm, which, on the interval $[0,1]$, is equivalent to the narrow topology.

\begin{lem} [Rudin, Thm.~7.10] \label{lem:Lebesque:points}
Let $d \geq 1$ and $f \in L^1(0,1; \R^d)$. Then, for any Lebesgue point $x \in (0,1)$ of $f$ and any sequences $(A_i)$ of Lebesgue measurable sets satisfying $A_i \subset [x - \tfrac1i, x + \tfrac1i]$ and $|A_i| \geq \tfrac c i$ for some $c > 0$ independent of $i$, it holds that
\begin{equation*}
  f (x) = \lim_{i \to \infty} \frac1{|A_i|} \int_{A_i} f(y) \, dy. 
\end{equation*}
\end{lem}

We define the bounded Lipschitz norm for functions $\varphi : [0,1] \to \R$ by
\begin{equation*}
  \| \varphi \|_{\text{BL}} 
  := \| \varphi \|_\infty + \sup_{x,y \in [0,1]} \frac{ |\varphi(x) - \varphi(y)| }{|x - y|},
\end{equation*}
and the dual bounded Lipschitz norm on the space of signed measures as
\begin{equation*}
  \| \nu \|_{\text{BL}}^* := \sup_{\| \varphi \|_{\text{BL}} = 1} \int_0^1 \varphi \, d \nu.
\end{equation*} 

\begin{lem} [Special case of {\cite[Thm.~18]{Dudley66}}] \label{lem:equiv:narTop:BLastNorm}
Let $(\mu_n) \subset \mathcal M_+([0,1])$. Then
\begin{equation*}
  \mu_n \weakto \mu  
  \quad \Longleftrightarrow \quad  \| \mu_n - \mu \|_{\text{BL}}^* \to 0.
\end{equation*}
\end{lem}

\begin{thm}[$\Gamma$-convergence of $E_n$ in case $\tfrac1n \alpha_n \to \alpha > 0$] \label{thm:Gconv:L4}
Let $\tfrac1n \alpha_n \to \alpha > 0$, and let $V$ and $W$ satisfy Assumption \ref{ass:VW:L4}. Then $E_n$ $\Gamma$-converges to 
  \begin{align*}
    E (\bmu) =
    \int_0^1 \eita \big( \brho (x) \big) \, dx,
  \end{align*}
  where the right-hand side is defined as $\infty$ if $\bmu$ is not absolutely continuous (cf.~\eqref{for:defn:abs:cont}).
\end{thm}

\begin{proof}
Since Assumption \ref{ass:VW:L4} is invariant under the scaling $\alpha V(\alpha \, \cdot \, )$ and $\alpha W(\alpha \, \cdot \, )$ for any $\alpha > 0$, we set $\alpha = 1$ without loss of generality.

We first proof the liminf-inequality \eqref{for:Gconv:liminf}. For technical reasons, we assume that $(\alpha_n)$ is strictly increasing as a mapping of $\N$ to $(0, \infty)$, and leave the general case to the end of the proof of \eqref{for:Gconv:liminf}. We set $\alpha : (0, \infty) \to (0, \infty)$ as the linear interpolation between the coordinates $(0,0)$ and $(n, \alpha_n)_{n \in \N}$, note that the inverse $\alpha^{-1} : (0, \infty) \to (0, \infty)$ exists, and obtain
\begin{equation*}
  \alpha(x)/x \xto{x \to \infty} 1,
  \qquad \alpha^{-1}(x)/x \xto{x \to \infty} 1.
\end{equation*} 

Let $\bmu \in M([0,1])$ and $\bmu_n \weakto \bmu$ with corresponding particle positions $x^{n, \pm}$ such that $E_n (x^n)$ is bounded uniformly in $n$. First, we prove that this uniform bound on $E_n (x^n)$ implies regularity on $\mu^\pm$. Indeed, starting from
  \begin{equation*}
    C
    \geq \liminf_{n\to\infty} E_n (x^n)
    \geq \liminf_{n\to\infty} \big( E_n^{+}(x^{n,+}) + E_n^{-}(x^{n,-}) \big)
  \end{equation*}
  and using Lemma \ref{lem:Gconv:single:L4} and \eqref{for:pf:lem:props:eita:1} to estimate
  \begin{multline*}
    \liminf_{n\to\infty} E_n^\pm (x^{n,\pm})
    \geq \liminf_{n\to\infty} \frac1{n^2} \sum_{i=1}^{n^+} \sum_{ j = 1 }^{i-1} \alpha_n V^{**} \big( \alpha_n ( x_i^\pm - x_j^\pm ) \big) \\
    \geq \int_0^1 \Veffce \bigg( \frac{1}{\rho^\pm (x)} \bigg) \rho^\pm (x) \, dx
    \geq c ( \| \rho^\pm \|_2^2 - C ),
  \end{multline*}
  we conclude that $\mu^\pm$ is absolutely continuous with density $\rho^\pm \in L^2(0,1)$.
 
Next we estimate $E_n (x^n)$ from below by a sum of $K \in \N$ \emph{independent} cell problems. Given $K$, we consider the equidistant partition of $[0,1]$ as illustrated in Figure \ref{fig:partn:domain}, and interpret each interval $J_k$ as a cell. More precisely, we set $J_1 := [0, \tfrac1K]$ and $J_k := (\tfrac{k - 1}K, \tfrac{k}K]$ for all $2 \leq k \leq K$. For any $1 \leq k \leq K$, we further set
  \begin{equation*}
    n_k^\pm := n \mu_n^\pm (J_k) = \# \{ x_i^\pm \in J_k : 1 \leq i \leq n^\pm \},
    \qquad n_k := n_k^+ + n_k^-.
  \end{equation*}  
  Since $\mu_n^\pm \weakto \mu^\pm$ and $\rho^\pm \in L^1(0,1)$, it holds that
\begin{equation} \label{fp:sigmaknpm:conv}
  \sigma_k^{n,\pm} 
     := K \mu_n^\pm (J_k)
     \xto{n \to \infty} K \mu^\pm (J_k) 
     =: \sigma_k^{\pm},
     \quad \text{for all } 1 \leq k \leq K.
\end{equation}
By removing many long range interactions from the energy and exploiting the translation invariance of the interactions, we estimate 
  \begin{align*}
    E_n ( x^n )
    &\geq \sum_{k=1}^K \frac1{n^2} \sum_{ \substack{ x_i, x_j \in J_k \\ i > j } } \alpha_n V_{ij} \big( \alpha_n [ x_i - x_j ] \big) \\
    &\geq \sum_{k=1}^K \min \bigg\{ \frac{\alpha_n}{n^2} \sum_{i > j}^{n_k} V_{ij} \big( \alpha_n [ \tilde x_i - \tilde x_j ] \big) : \tilde x_i \in \overline{J_k} \ \ \forall \, 1 \leq i \leq n_k \bigg\} \\
    &= \frac{\alpha_n^2}{n^2} \frac1K \sum_{k=1}^K \min \bigg\{ \frac K{\alpha_n} \sum_{i > j}^{n_k} V_{ij} \Big( \frac{\alpha_n}K [ y_i - y_j ] \Big) : y_i \in [0,1] \ \ \forall \, 1 \leq i \leq n_k \bigg\},
  \end{align*}
  where we define the value of the minimisation problem to be $0$ when $n_k = 0$.
  
  Next, we change variables to bound these minimisation problems from below in terms of the cell problem \eqref{for:defn:emita}. We fix $k$, and set
  \begin{equation*}
    m_n := \frac{\alpha_n}K. 
  \end{equation*}
  It remains to define $\tilde \sigma_k^{m, \pm}$ such that 
  \begin{equation} \label{fp:nk:est}
    n_k = m_n \tilde \sigma_k^{m_n, +} + m_n \tilde \sigma_k^{m_n, -}
  \end{equation}
  for all $n \in \N$. Motivated by
  $$
    n_k 
    = \frac nK \sigma_k^n 
    = m_n \frac1{m_n} \frac nK \sigma_k^{n,+} + m_n \frac1{m_n} \frac nK \sigma_k^{n,-}
  $$
  and recalling that $n = \alpha^{-1} (\alpha_n) = \alpha^{-1} (K m_n)$, we set
  \begin{equation*}
    \tilde \sigma_k^{m, \pm} := \frac{ \alpha^{-1} (K m) }{K m} \sigma_k^{\lfloor \alpha^{-1} (K m) \rfloor, \pm},
    \qquad \tilde \sigma_k^m := \tilde \sigma_k^{m, +} + \tilde \sigma_k^{m, -}
    \qquad \text{for all } m > 0.
  \end{equation*}
  By construction, \eqref{fp:nk:est} holds for $m = m_n$ and $\tilde \sigma_k^{m, \pm} \to \sigma_k^\pm$ as $m \to \infty$. We obtain
  \begin{align*} 
    &\liminf_{n \to \infty} \min \bigg\{ \frac K{\alpha_n} \sum_{i > j}^{n_k} V_{ij} \Big( \frac{\alpha_n}K [ y_i - y_j ] \Big) : y_i \in [0,1] \ \ \forall \, 1 \leq i \leq n_k  \bigg\} \\\notag
    &= \liminf_{n \to \infty} \min \bigg\{ \frac1{ m_n } \sum_{i > j}^{ m_n \tilde \sigma_k^{m_n} } V_{ij} \big( m_n [ y_i - y_j ] \big) : y_i \in [0,1] \ \ \forall \, 1 \leq i \leq m_n \tilde \sigma_k^{m_n}  \bigg\} \\
    &\begin{multlined}[b][0.87 \textwidth] \geq \liminf_{m \to \infty} \min \bigg\{ 
          \frac1{ m } \sum_{i > j}^{ \lfloor m \tilde \sigma_k^{m,+} \rfloor + \lfloor m \tilde \sigma_k^{m,-} \rfloor } V_{ij} \big( m [ y_i - y_j ] \big) : \\\notag
          0 \leq y_1^+ \leq \ldots \leq y_{ \lfloor m \tilde \sigma_k^{m,+} \rfloor }^+ \leq 1, \:
          0 \leq y_1^- \leq \ldots \leq y_{ \lfloor m \tilde \sigma_k^{m,-} \rfloor }^- \leq 1 \bigg\} \end{multlined} \\ 
    &= \liminf_{m \to \infty} \emita( \tilde \bsigma_k^m  ).
  \end{align*}
Applying Lemma \ref{lem:props:eita}.\eqref{lem:props:eita:cont:conv}, we obtain
  \begin{equation} \label{for:pf:limf:step}
    \liminf_{n \to \infty} E_n ( \mu_n )
    \geq \frac1K \sum_{k=1}^K \eita ( \bsigma_k )
	\geq \frac1K \sum_{k=1}^K \big( M \wedge \eita ( \bsigma_k ) \big).
  \end{equation}
for any $M > 0$. 

Finally, we derive \eqref{for:Gconv:liminf} from \eqref{for:pf:limf:step} by first passing to the limit $K \to \infty$ and then $M \to \infty$. To this aim, we set
\begin{equation*}
    \rho_K^\pm (x) 
    := \sum_{k=1}^K \bigg( K \int_{J_k} \rho^\pm (y) \, dy \bigg) \indicatornoacc{ J_k } (x)
    = \sum_{k=1}^K \sigma_k^\pm \indicatornoacc{ J_k } (x)
  \end{equation*}  
and observe that 
\begin{equation} \label{fp:KM}
  \frac1K \sum_{k=1}^K \big( M \wedge \eita ( \bsigma_k ) \big)  
  = \int_0^1 \big[ M \wedge \eita \big( \brho_K(x) \big) \big] \, dx.
\end{equation}
First, by Lemma \ref{lem:Lebesque:points}, $\brho_K \to \brho$ pointwise a.e.~on $(0,1)$ as $K \to \infty$. Second, by Lemma \ref{lem:props:eita}.\eqref{lem:props:eita:bds},\eqref{lem:props:eita:cty}, it holds that $M \wedge \eita : [0,\infty)^2 \to \R$ is uniformly continuous. Together, these statements imply
\begin{equation*} 
  M \wedge \eita \big( \brho_K(x) \big)
  \xto{K \to \infty} M \wedge \eita ( \brho(x) )
  \quad \text{for a.e.~} x \in (0,1).
\end{equation*}
Hence, by the Dominated Convergence Theorem, we can pass to the limit $K \to \infty$ in \eqref{fp:KM} to obtain
\begin{equation*} 
  \frac1K \sum_{k=1}^K \big( M \wedge \eita ( \bsigma_k ) \big)
  \xto{ K \to \infty } \int_0^1 \big( M \wedge \eita ( \brho(x) ) \big) \, dx.
\end{equation*}
Then, using the Monotone Convergence Theorem, we pass to the limit $M \to \infty$ to obtain
\begin{equation*}
  \int_0^1 \big( M \wedge \eita ( \brho(x) ) \big) \, dx
  \xto{ M \to \infty } \int_0^1 \eita ( \brho(x) )\, dx
  = E (\bmu ),
\end{equation*}
which completes the proof of the liminf-inequality \eqref{for:Gconv:liminf} under the assumption that $(\alpha_n)$ is increasing.

In the general case where $(\alpha_n)$ is not increasing, we consider any subsequence $(\alpha_{n_k})$, and extract another subsequence which is increasing. Such a subsequence always exists, because $\alpha_n \to \infty$ as $n \to \infty$. The arguments above apply also to this increasing subsequence, and since the limit in \eqref{fp:sigmaknpm:conv} and the lower bound in \eqref{for:pf:limf:step} do not depend on the choice of the subsequence, we conclude that \eqref{for:Gconv:liminf} holds for any $(\alpha_n)$ with $\tfrac1n \alpha_n \to 1$.

The second part of the proof establishes the limsup-inequality \eqref{for:Gconv:limsup}.
Let $\bmu \in M([0,1])$ such that $E( \bmu )$ is finite. Then, by Lemma \ref{lem:props:eita}.\eqref{lem:props:eita:bds} it follows that $\brho \in L^2(0,1)$. Next we show, by the usual density arguments, that it suffices to construct a recovery sequence in \eqref{for:Gconv:limsup} only for $\brho \in C([0,1]) \cap M([0,1])$. To prove this, we take $\brho \in L^2(0,1) \cap M([0,1])$ arbitrarily and construct $(\brho_\varepsilon) \subset C([0,1]) \cap M([0,1])$ such that $\brho_\varepsilon \weakto \brho$ as $\varepsilon \to 0$ and
\begin{equation} \label{for:pf:limp:Eeps}
  \limsup_{\varepsilon \to 0} E (\brho_\varepsilon)
  \leq E (\brho).
\end{equation}

We first assume that $\brho \in L^\infty ([0,1]) \cap M([0,1])$. We take any $(\brho_\varepsilon) \subset C([0,1]) \cap M([0,1])$ such that $\brho_\varepsilon \to \brho$ as $\varepsilon \to 0$ both in $L^2 (0,1)$ and pointwise a.e.~on $(0,1)$. To show that such a choice is possible, take for example $\tilde \rho_\varepsilon^\pm := \eta_\varepsilon * \rho^\pm \in C^\infty (\R)$, where $\eta_\varepsilon$ is the usual mollifier. Note that the non-negativity and unit mass condition are satisfied, but that $\supp \tilde \brho_\varepsilon$ may not be contained in $[0,1]$. This is easily fixed by setting 
\begin{equation*}
  \rho_\varepsilon^\pm
  := \tilde \rho_\varepsilon^\pm \big|_{(0,1)} + \int_{\R \setminus (0,1)} \tilde \rho_\varepsilon^\pm.
\end{equation*}
By construction, $(\brho_\varepsilon) \subset C([0,1]) \cap M([0,1])$ and $\brho_\varepsilon \to \brho$ in $L^2 (0,1)$ as $\varepsilon \to 0$. By extracting a subsequence, we then also have $\brho_\varepsilon \to \brho$ pointwise a.e.~on $(0,1)$. To check that \eqref{for:pf:limp:Eeps} is satisfied, we observe that $\| \brho_\varepsilon \|_\infty \leq \| \brho \|_\infty + \frac12$. Then, by Lemma \ref{lem:props:eita}.\eqref{lem:props:eita:bds},\eqref{lem:props:eita:cty}, $\eita$ is uniformly continuous on the levelset $\{ \eita \leq \| \brho \|_\infty + \frac12 \}$, and thus we obtain \eqref{for:pf:limp:Eeps} by applying the Dominated Convergence Theorem.

To complete the density argument, we take any $\brho \in L^2 ([0,1]) \cap M([0,1])$, and construct $(\brho_\varepsilon) \subset L^\infty ([0,1]) \cap M([0,1])$ which converges narrowly to $\brho$ and satisfies \eqref{for:pf:limp:Eeps}. Let $A^\pm := \{ \rho^\pm \leq 3 \}$ and $A = A^+ \cap A^-$. We note that $|A| \geq \frac13$, and set
\begin{equation*}
  \rho_\varepsilon^\pm
  := ( \rho^\pm \wedge \tfrac1\varepsilon ) \vee ( m_\varepsilon^\pm \indicatornoacc{A} ),
\end{equation*}
where we set $0 \leq m_\varepsilon^\pm \leq 3$ such that $\int_0^1 \rho_\varepsilon^\pm = \int_0^1 \rho^\pm$. We note that $m_\varepsilon^\pm \to 0$ as as $\varepsilon \to 0$, and hence $\brho_\varepsilon \to \brho$ both in $L^2 (0,1)$ and pointwise a.e.~on $(0,1)$. In particular, $\brho_\varepsilon$ is uniformly bounded on $A$, and $\brho_\varepsilon \nearrow \brho$ pointwise a.e.~on $A^c$. Hence, by using both the Dominated and Monotone Convergence Theorems, we obtain
\begin{equation*}
  E (\brho_\varepsilon) 
  = \int_{ A } \eita (\brho_\varepsilon) + \int_{ A^c } \eita (\brho_\varepsilon)
  \xto{ \varepsilon \to 0 } \int_{ A } \eita (\brho) + \int_{ A^c } \eita (\brho)
  = E (\brho).
\end{equation*}

To prove \eqref{for:Gconv:limsup}, it remains to construct a recovery sequence $(\bar x^n)$ for any $\brho \in C([0,1]) \cap M([0,1])$. We do this by a slight modification of the usual density argument. First, we approximate $\brho$ by $\brho^\ell$ similarly as in the proof of Theorem \ref{thm:Gconv:L3}, i.e., we set
\begin{equation} \label{for:pf:defn:Kell}
  K_\ell = 2^\ell,
  \quad \varepsilon_\ell 
  =  \sqrt{ \frac1{K_\ell} + \frac14 } - \frac12
  = \frac12 \big( \sqrt{ 2^{2-\ell} + 1 } - 1 \big)
  \quad \text{for } \ell = 1,2,\ldots,
\end{equation}
and take the intervals $I_k$ and $J_k$ of size $\varepsilon_\ell > 0$ and $\varepsilon_\ell (1 + \varepsilon_\ell) > 0$ respectively, as in Figure \ref{fig:partn:domain}. Then, as in \eqref{for:defn:rhopm:limsupClass} we set 
$$
  \brho^\ell 
  = \sum_{k=1}^{K_\ell} \bsigma_k \indicatornoacc{ I_k },
  \quad \text{where} \quad
  \bsigma_k 
  := \frac1{\varepsilon_\ell} \int_{J_k} \brho
  \quad \text{for } k = 1, \ldots, K_\ell.
$$ 
Since $\brho \in C([0,1])$, it holds that $\brho^\ell \to \brho$ in $L^2(0,1)$ as $\ell \to \infty$ and $\| \brho^\ell \|_\infty \leq (1 + \varepsilon_\ell) \| \brho \|_\infty$. Hence, along a subsequence $\ell_j$, $\brho^{\ell_j} \to \brho$ pointwise a.e.~on $(0,1)$. Then, by Lemma \ref{lem:props:eita}.\eqref{lem:props:eita:bds},\eqref{lem:props:eita:cty}, we have $\eita (\brho^{\ell_j}) \to \eita( \brho )$ pointwise a.e.~on $(0,1)$ as $j \to \infty$, and $\eita (\brho^\ell)$ uniformly bounded in $\ell$. By the Dominated Convergence Theorem, we conclude that
\begin{equation} \label{for:pf:limp:densy:ellj}
  E (\brho^{\ell_j})
  \xto{j \to \infty} E (\brho).
\end{equation}

Given $\ell \in \N$, we construct a `recovery sequence' $(x^n)$ for 
$
  \brho^\ell
$
by concatenating the minimisers of the cell problems \eqref{for:defn:emita} corresponding to each $I_k$. We prove that 
\begin{equation} \label{for:pf:limp:recseq}
  \limsup_{n \to \infty} E_n (x^n)
  \leq E (\brho^\ell),
\end{equation}
but do not require that $\bmu_n \weakto \brho^\ell$ as $n \to \infty$. Hence, $(x^n)$ may not be a recovery sequence for $\brho^\ell$. Instead, we construct $(x^n)$ such that in the joint limit $n \to \infty$ and $\ell \to \infty$, it holds that $(x^n)$ (which also depends on $\ell$) converges to $\brho$. We prove this at the end of the proof by a diagonal argument on \eqref{for:pf:limp:densy:ellj} and \eqref{for:pf:limp:recseq}.

Given $\ell \in \N$, we construct $x^n$ and show that it satisfies \eqref{for:pf:limp:recseq}. Since $\ell$ is fixed, we remove it from the notation whenever convenient. Given $n \in \N$, we set $\bsigma := \int_0^1 \brho^\ell$, note that $\sigma^+ + \sigma^- = 1$, and choose $n^\pm \in \N \cup \{0\}$ such that $| n^\pm - n \sigma^\pm | \leq \tfrac12$ and $n^+ + n^- = n$. Recalling that $\varepsilon > 0$ is such that $K \varepsilon (1 + \varepsilon) = 1$, we divide the positive and negative particles over the intervals $I_k$ by choosing $n_k^\pm \in \N \cup \{0\}$ such that $\sum_{k=1}^K n_k^\pm = n^\pm$ and $| n_k^\pm - \varepsilon n \sigma_k^\pm | \leq 1$. An example of such a choice is given in the proof of Theorem \ref{thm:Gconv:L3}. We further set the average density of the particles at $I_k$ as
\begin{equation*}
  \sigma_k^{n,\pm} := \frac1{\varepsilon \alpha_n} {n_k^\pm},
\end{equation*}
and observe that 
\begin{equation} \label{fp:si:est}
  \Big| \frac{\alpha_n}n \sigma_k^{n,\pm} - \sigma_k^\pm \Big| \leq \frac1{\varepsilon n}
  \quad \text{for } 1 \leq k \leq K.
\end{equation}
We set $y^k \in \R^{n_k}$ as a minimiser of $\eenita (\bsigma_k^n)$, and observe from \eqref{for:defn:emita} that
\begin{equation*}
  \eenita (\bsigma_k^n)
   = \frac1{\varepsilon \alpha_n} \sum_{i > j}^{ \varepsilon \alpha_n \sigma_k^n }
        V_{ij} \big( \varepsilon \alpha_n ( y_i - y_j ) \big)
   = \frac1{\varepsilon \alpha_n} \sum_{i > j}^{ n_k }
        V_{ij} \big( \varepsilon \alpha_n ( y_i - y_j ) \big).
\end{equation*}
Finally, we set the recovery sequence $x^n$ by scaling and translating the particle positions $y^k$ from the cell problem to $I_k$. More precisely, we set
\begin{gather} \label{for:pf:ell:recseq}
  \tilde x_i^k := \varepsilon y_i^k + \varepsilon (1 + \varepsilon) (k-1)
  \quad \text{for } i = 1,\ldots,n_k, 
  \quad \text{and} \quad 
  x^n := (\tilde x^1, \ldots, \tilde x^K).
\end{gather}

To prove \eqref{for:pf:limp:recseq}, we follow a similar argument as the one starting at \eqref{for:pf:lem:props:eita:15}. We expand
\begin{multline} \label{for:pf:Enita:expd}
  E_n (x^n)
  = \frac{\alpha_n}{n^2} \sum_{i > j}^{n} V_{ij} (\alpha_n (x_i - x_j)) \\
  = \sum_{k=1}^K \frac{\alpha_n}{n^2} \sum_{i > j}^{n_k} V_{ij} (\alpha_n (\tilde x_i^k - \tilde x_j^k)) 
       + 2 \sum_{l=1}^{K-1} \sum_{k=1}^{K - l} \frac{\alpha_n}{n^2} \sum_{i = 1}^{n_{k + l}} \sum_{j = 1}^{n_k} V_{ij} (\alpha_n (\tilde x_i^{k+l} - \tilde x_j^k)).
\end{multline}
By construction and Lemma \ref{lem:props:eita}.\eqref{lem:props:eita:cont:conv}, we pass to the limit $n \to \infty$ in the first term in the right-hand side of \eqref{for:pf:Enita:expd} by
\begin{multline*}
  \sum_{k=1}^K \frac{\alpha_n}{n^2} \sum_{i > j}^{n_k} V_{ij} (\alpha_n (\tilde x_i^k - \tilde x_j^k)) 
  = \frac{\alpha_n^2}{n^2} \sum_{k=1}^K \varepsilon \frac1{\varepsilon \alpha_n} \sum_{i > j}^{n_k} V_{ij} (\varepsilon \alpha_n (y_i^k - y_j^k)) 
  = \varepsilon\sum_{k=1}^K \eenita( \bsigma_k^n ) \\
  \xto{n \to \infty} \varepsilon \sum_{k=1}^K \eita( \bsigma_k )
  = \int_0^1 \eita (\brho^\ell)
  = E (\brho^\ell).
\end{multline*}
It remains to show that the second term in the right-hand side of \eqref{for:pf:Enita:expd} converges to $0$ as $n \to \infty$. Using that $\tilde x_i^{k+l} - \tilde x_j^k \geq \varepsilon^2 + (l - 1) \varepsilon$ and $V_{ij} \leq V$, we estimate
\begin{align} \notag
  &\sum_{l=1}^{K-1} \sum_{k=1}^{K - l} \frac{\alpha_n}{n^2} \sum_{i = 1}^{n_{k + l}} \sum_{j = 1}^{n_k} V_{ij} (\alpha_n (\tilde x_i^{k+l} - \tilde x_j^k))
  \leq \sum_{l=1}^{K-1} \sum_{k=1}^{K - l} \frac{n_{k + l} n_k}{n^2} \alpha_n V (\alpha_n \varepsilon^2 + \alpha_n (l - 1) \varepsilon) \\\notag
  &\leq \sum_{l=1}^{K-1} \sum_{k=1}^{K - l} \frac{(\varepsilon n \sigma_{k+l} + 2 )  (\varepsilon n \sigma_k + 2 ) }{n^2} \alpha_n V (\alpha_n \varepsilon^2 + \alpha_n (l - 1) \varepsilon) \\\label{for:pf:rest:Enita:small}
  &\leq K \Big[ \max_{1 \leq k \leq K} \sigma_k \Big]^2 \varepsilon^2 \big( 1 + \mathcal O (n^{-1}) \big) \big( \alpha_n V (\alpha_n \varepsilon^2) + \alpha_n \Veff (\alpha_n \varepsilon) \big).  
\end{align}
Since the constants $\varepsilon$, $K$ and $\max_k \sigma_k$ are fixed by the choice of $\brho^\ell$, it follows from \eqref{for:rem:Veff:L4} that \eqref{for:pf:rest:Enita:small} converges to $0$ as $n \to \infty$.

Finally, we complete the proof of \eqref{for:Gconv:limsup} for $\brho \in C([0,1]) \cap M([0,1])$ by constructing the sequence $\bar x^n$. For any $\ell \geq 1$, let $x_\ell^n \in \R^n$ be the sequence constructed in \eqref{for:pf:ell:recseq} for which \eqref{for:pf:limp:recseq} holds. Then, by a diagonal argument, we find from \eqref{for:pf:limp:densy:ellj} and \eqref{for:pf:limp:recseq} that 
\begin{equation*}
  \bar x^n := x_{\ell_{j_n}}^n 
  \quad \text{satisfies} \quad
  \limsup_{n \to \infty} E_n (\bar x^n)
  \leq E (\brho)
\end{equation*}
for any non-decreasing sequence $j_n \to \infty$ as $n \to \infty$ provided that $j_n$ is small enough with respect to $n$. We choose $j_n$ such that $K^n := K_{\ell_{j_n}}$ (defined in \eqref{for:pf:defn:Kell}, together with $\varepsilon^n := \varepsilon_{\ell_{j_n}}$) satisfies $\tfrac1n K^n \to 0$ as $n \to \infty$. For such choice, we prove that the recovery sequence satisfies $\bar \bmu_n \weakto \brho$. In the estimate below, we simplify notation by writing $I_k$, $J_k$, $\bsigma_k$, $n_k^\pm$ and $\bsigma_k^n$ for the objects corresponding to the construction of $x_\ell^n$ in \eqref{for:pf:ell:recseq} for $\ell = \ell_{j_n}$. By Lemma \ref{lem:equiv:narTop:BLastNorm}, the convergence of the recovery sequence follows from \eqref{fp:si:est} by
\begin{align*}
  \| \bar \mu_n^\pm - \rho^\pm \|_{\text{BL}}^*
  &= \sup_{\| \varphi \|_{\text{BL}} = 1} \int_0^1 \varphi \, d (\bar \mu_n^\pm - \rho^\pm )
  = \sup_{\| \varphi \|_{\text{BL}} = 1} \sum_{k=1}^{K^n} \bigg[ \int_{J_k} \varphi \, d \bar \mu_n^\pm - \int_{J_k} \varphi \, d \rho^\pm \bigg] \\
  &\leq \sup_{\| \varphi \|_{\text{BL}} = 1} \sum_{k=1}^{K^n} \bigg[ \Big( \sup_{J_k} \varphi \Big) \frac{n_k^\pm}{n} - \Big( \inf_{J_k} \varphi \Big) \varepsilon^n \sigma_k^\pm \bigg] \\
  &= \sup_{\| \varphi \|_{\text{BL}} = 1} \sum_{k=1}^{K^n} \bigg[ \Big( \sup_{J_k} \varphi \Big) \varepsilon^n \Big( \frac{ \alpha_n }n \sigma_k^{n,\pm} - \sigma_k^\pm \Big) + \Big( \sup_{J_k} \varphi - \inf_{J_k} \varphi \Big) \varepsilon^n \sigma_k^\pm \bigg] \\
  &\leq K^n \varepsilon^n \frac1{ \varepsilon^n n } + \sum_{k=1}^{K^n} |J_k| \varepsilon^n \sigma_k^\pm
  = \frac{K^n}{ n } + \varepsilon^n (1 + \varepsilon^n) \sigma^\pm
  \xto{n \to \infty} 0. \qedhere
\end{align*}
\end{proof}

\section{$\Gamma$-convergence of $E_n$ in the case $\gamma_n \to \infty$}
\label{s:gninf}

The parameter regime $\gamma_n \to \infty$ of the energy $E_n$ in \eqref{for:defn:En} describes a scenario where the external forcing is strong enough for the positive particles to cluster at the left barrier (and the negative particles to cluster at the right barrier) on a length-scale asymptotically smaller than $1$. In order to obtain a useful limit, it is therefore necessary to rescale $x^{n,\pm}$ to fit to this length-scale. It is shown in \cite{VanMeursMunteanPeletier14} that a sensible rescaling is given by $\hat x^{n,+} := \gamma_n x^{n,+}$. We rescale the negative particles similarly, and for convenience later on, we introduce simultaneously the affine variable transformation 
$$
\hat x^-_i := \gamma_n (1 - x^-_{n^- + 1 - i}),
\quad i = 1, \ldots, n^-.
$$ 
A result of this variable transformation is that $0 \leq \hat x_1^\pm < \ldots < \hat x_{n^\pm}^\pm \leq \gamma_n$. We rescale the energy as
\begin{align*}
    \hat E_n (\hat x^{n,+}, \hat x^{n,-})
    &:= \frac1{\gamma_n} E_n \Big( \frac{\hat x^{n,+}}{\gamma_n}, \, \frac1{\gamma_n} \big( 1 - \hat x^-_{n^- + 1 - i} \big)_{i=1}^{n^-} \Big) \\\notag
	&= \frac1{n^2} \sum_{i=1}^{n^+} \sum_{j = 1}^{i-1} \hat \alpha_n V \big( \hat \alpha_n (\hat x_i^+ - \hat x_j^+) \big)
        + \frac1{n^2} \sum_{i=1}^{n^-} \sum_{j = 1}^{i-1} \hat \alpha_n V \big( \hat \alpha_n (\hat x_i^- - \hat x_j^-) \big) \\
        &\quad+ \frac1{n^2} \sum_{i=1}^{n^+} \sum_{j = 1}^{n^-} \hat \alpha_n W \big( \hat \alpha_n (\gamma_n - \hat x_i^+ - \hat x_j^-) \big)
        + \frac1n \sum_{i = 1}^{n^+} \hat x_i^+
        + \frac1n \sum_{i = 1}^{n^-} \hat x_i^-,
\end{align*}
where
\begin{equation*}
  \hat \alpha_n = \alpha_n / \gamma_n.
\end{equation*}
We observe that $\hat E_n$ consists of the components
\begin{align*}
  \hat E_n^\pm (\hat x^{n,\pm})
  &:= \frac1{n^2} \sum_{i=1}^{n^\pm} \sum_{j = 1}^{i-1} \hat \alpha_n V \big( \hat \alpha_n (\hat x_i^\pm - \hat x_j^\pm) \big) + \frac1n \sum_{i = 1}^{n^\pm} \hat x_i^\pm, \\
  \Enmix (\hat x^{n,+}, \hat x^{n,-})
  &:= \frac1{n^2} \sum_{i=1}^{n^+} \sum_{j = 1}^{n^-} \hat \alpha_n W \big( \hat \alpha_n (\gamma_n - \hat x_i^+ - \hat x_j^-) \big).
\end{align*} 
We note that $\hat E_n^+$ and $\hat E_n^-$ are the same energies when $n^+ = n^-$. Moreover, $\Gamma$-convergence of $\hat E_n^\pm$ was first proven in \cite{GeersPeerlingsPeletierScardia13} for all scaling regimes of $\hat \alpha_n$, and in \S \ref{s:pf} we extend this result by relaxing the assumptions on $V$; see Table \ref{tab:V:gninf}. The $\Gamma$-limit is given by 
\begin{equation} \label{fd:Epm:gninf}
  \hat E^\pm (\hat \mu^\pm) = \Egita (\hat \mu^\pm) + \int_0^\infty x \, d\hat \mu^\pm (x),
\end{equation}
where the expression of $\Egita$ depends on the asymptotic behaviour of $\hat \alpha_n$ (see Table \ref{tab:Eita:gninf}). Since boundedness of $\hat E_n^\pm$ forces the bulk of the particles to remain in a bounded interval, the second component of $\hat E_n$ (given by $\Enmix$) vanishes in the limit $n \to \infty$ given that $W$ satisfies Assumption \ref{ass:W:gninf}. Theorem \ref{thm:Gconv:gninf} makes this statement precise.

\begin{table}[t]
\centering 
\begin{tabular}{cl}
  \toprule
  Regime & properties of $V$ \\
  \midrule
  $\hat \alpha_n \rightarrow \hat \alpha > 0$ 
  & \specialcell[c]{$V = \Vsing + \Vreg \in L^1(\R)$, where $\Vreg \in C_b (\R)$ is even, and \\ $\Vsing \in L^1 (\R)$ is even, non-negative and non-increasing on $(0, \infty)$;}
  \\ \midrule
  $1 \ll \hat \alpha_n \ll n$ 
  & $V$ satisfies Assumption \ref{ass:VW:L3} with $W = 0$;
  \\ \midrule
  $\dfrac{\hat \alpha_n}n \rightarrow \hat \alpha$ 
  & \specialcell[c]{$V : \R \setminus \{0\} \to [0, \infty)$ is even, and convex on $(0,\infty)$. \\ Moreover, $V(0) := \lim_{x \to 0} V(x) \in [0, \infty]$ and $\int_1^\infty V < \infty$.}
  \\
  \bottomrule
\end{tabular}
\caption{Properties of $V$ for which $\hat E_n^\pm$ is $\Gamma$-convergent.}
\label{tab:V:gninf}
\end{table}

\begin{table}[b]
\centering 
\begin{tabular}{cl}
  \toprule
  Regime & $\Egita (\hat \mu)$ \\
  \midrule
  $\hat \alpha_n \rightarrow \hat \alpha$ 
  & $\displaystyle \frac12 \iint_{[0,\infty)^2} \hat \alpha V (\hat \alpha (x-y)) \, d( \hat \mu \otimes \hat \mu )(x,y)$ 
  \\ \midrule
  $1 \ll \hat \alpha_n \ll n$ 
  & $\displaystyle \bigg( \int_0^\infty V \bigg) \int_0^\infty \hat \rho (x)^2 \, dx$ 
  \\ \midrule
  $\dfrac{\hat \alpha_n}n \rightarrow \hat \alpha$ 
  & $\displaystyle \int_0^\infty \hat \alpha \Veff \bigg( \frac{\hat \alpha}{\hat \rho (x)} \bigg) \hat \rho (x) \, dx$
  \\
  \bottomrule
\end{tabular}
\caption{ Expressions for $\Egita$, the interaction part of the limit energy $\hat E^\pm$ defined in \eqref{fd:Epm:gninf}. In the regimes where $\hat \alpha_n \gg 1$, the expressions are valid when $\hat \mu$ is absolutely continuous (see \eqref{for:defn:abs:cont}) with density $\hat \rho \in L^1(0,1)$; otherwise $\Egita (\hat \mu) = \infty$. Note the resemblance with Table \ref{tab:Eita}.}
\label{tab:Eita:gninf}
\end{table}

\begin{ass}[Properties of $W$ in case $\gamma_n \to \infty$] \label{ass:W:gninf}
$W : \R \to [0, \infty]$ is even, and satisfies $W(x) \leq \tfrac Cx$ for all $x \geq 1$ and $C > 0$ independent of $x$.
\end{ass}

We adopt the same notation to rewrite $E_n$ in terms of the measures $(\hat \mu_n^+, \hat \mu_n^-) =: \hat \bmu_n$ given by \eqref{for:defn:munpm}.

\begin{thm}[$\Gamma$-convergence of $\hat E_n$ in case $\gamma_n \to \infty$] \label{thm:Gconv:gninf}
Let $\gamma_n \to \infty$, and let $W$ satisfy Assumption \ref{ass:W:gninf}. Let $(\hat \alpha_n)$ be as in any of the three scaling regimes outlined in Table \ref{tab:V:gninf}, and let $V$ satisfy the corresponding assumption. Then, any sequence $(\hat \bmu_n) \subset M([0,\infty))$ satisfying $\hat E_n (\hat \bmu_n) \leq C$ for some $n$-independent $C > 0$ is compact in the narrow topology. Moreover, $\hat E_n$ $\Gamma$-converges to $\hat E(\hat \bmu) = \hat E^+ (\hat \mu^+) + \hat E^- (\hat \mu^-)$, where $\hat E^\pm$ is given by \eqref{fd:Epm:gninf} and Table \ref{tab:Eita:gninf}.
\end{thm}

\begin{proof}
Let $\hat E_n (\hat \bmu_n) \leq C$. Since $W \geq 0$ and $V$ is bounded from below, it holds that
\begin{equation} \label{fp:En:LB}
  \hat E_n (\hat \bmu_n) \geq \hat \alpha_n \frac{ (n^+)^2 + (n^-)^2 - n }{2 n^2} (\inf V) + \int_0^\infty x \, d(\hat \mu_n^+ + \hat \mu_n^-)(x).
\end{equation}
We note from Table \ref{tab:V:gninf} that the first term in the right-hand side of \eqref{fp:En:LB} is bounded from below.
Hence, \eqref{fp:En:LB} implies that the first moments of $\hat \mu_n^+$ and $\hat \mu_n^-$ are uniformly bounded; we conclude compactness of $\hat \bmu_n$.

Since $W \geq 0$, we obtain the liminf-inequality \eqref{for:Gconv:liminf} from the $\Gamma$-convergence of $\hat E_n^\pm$ by
\begin{equation*}
  \liminf_{n \to \infty} \hat E_n (\hat \bmu_n)
  \geq \liminf_{n \to \infty} \hat E_n^+ (\hat \mu_n^+) + \liminf_{n \to \infty} \hat E_n^- (\hat \mu_n^-)
  \geq \hat E^+ (\hat \mu^+) + \hat E^- (\hat \mu^-) 
  = \hat E (\hat \bmu). 
\end{equation*}
We prove the limsup-inequality \eqref{for:Gconv:limsup} by an analogous argument, relying on the claim that
\begin{equation} \label{fp:Enmix:to:0}
  \Enmix (\hat \bmu_n) \xto{n \to \infty} 0,
\end{equation}
where $\hat \bmu_n$ consists of any recovery sequence $\hat \mu_n^\pm$ related to the $\Gamma$-convergence of $\hat E_n^\pm$. These recovery sequences are constructed explicitly only for $\hat \mu^\pm$ smooth enough, including $\hat \mu^\pm$ having bounded support. The case for general $\hat \mu^\pm$ is treated by a diagonal argument, relying on upper semi-continuity of $\hat E^\pm$. Hence, for any $\hat \bmu \in M([0,\infty))$ with $\hat E(\hat \bmu)$ bounded, we choose the recovery sequences $(\hat x^{n,+})$ and $(\hat x^{n,-})$ such that
\begin{equation*}
  \hat x_{n^+}^+ + \hat x_{n^-}^- \leq \tfrac12 \gamma_n.
\end{equation*}
Then, the claim \eqref{fp:Enmix:to:0} follows from
\begin{multline*}
  \Enmix (\hat \bmu_n)
  = \frac1{n^2} \sum_{i=1}^{n^+} \sum_{j = 1}^{n^-} \hat \alpha_n W \big( \hat \alpha_n (\gamma_n - \hat x_i^+ - \hat x_j^-) \big) \\
  \leq \frac 1{n^2} \sum_{i=1}^{n^+} \sum_{j = 1}^{n^-} \hat \alpha_n \frac C{ \hat \alpha_n (\gamma_n - (\hat x_i^+ + \hat x_j^-)) }
  \leq \frac 1{n^2} \sum_{i=1}^{n^+} \sum_{j = 1}^{n^-} \frac{2 C}{ \gamma_n }
  \leq \frac C{ \gamma_n } 
  \xto{n \to \infty} 0. \qedhere
\end{multline*}
\end{proof}

\section{Evolutionary convergence of the gradient flow of $E_n$ in the case $\alpha_n \to \alpha > 0$}
\label{s:dyn}

The starting point in this section is the gradient flow of $E_n$ given by \eqref{fd:GFn} in the scaling regime $\alpha_n \to \alpha > 0$. The main result (Theorem \ref{t:dyncs}) of this section is an evolutionary convergence result of the gradient flows of $E_n$ to the gradient flow of the $\Gamma$-limit $E$ as $n \to \infty$. The proof strategy is to apply the setting of gradient flows with $\lambda$-convex energies in \cite[Chap.~4]{AmbrosioGigliSavare08} and \cite{DaneriSavareLN10} (\S \ref{ss:AGS:DS}) to the gradient flow of $E_n$ (\S \ref{ss:evo:conv}). 

\subsection{Preliminaries on gradient flows of $\lambda$-convex energies}
\label{ss:AGS:DS}

We summarise a simplified version of the results in \cite[Chap.~4]{AmbrosioGigliSavare08} and \cite{DaneriSavareLN10}. Let $(X,d)$ be a complete, separable, non-positively curved (see \eqref{fd:npcurved}), sequentially compact metric space. We call a curve $x : [0,1] \to X$ a (constant speed) \emph{geodesic} if
\begin{equation} \label{fd:geods}
  d(x(s), x(t)) = |t-s| d(x(0), x(1))
  \quad \text{for all } 0 \leq s \leq t \leq 1.
\end{equation}

We consider any $\phi : X \to \R \cup \{\infty\}$ with non-empty domain
\begin{equation*}
  D (\phi) := \{ x \in X : \phi(x) < \infty \},
\end{equation*}
and assume that $\phi$ is $\lambda$\emph{-convex} for some $\lambda \in \R$, i.e., every couple of points $x_0, x_1 \in D (\phi)$ can be connected by a geodesic $x_t$ along which
\begin{equation*}
  \phi(x_t) \leq (1-t) \phi(x_0) + t \phi(x_1) - \tfrac12 \lambda t (1-t) d(x_0, x_1)^2
  \quad \text{for all } 0 \leq t \leq 1.
\end{equation*}
In particular, $(X,d)$ being \emph{non-positively curved} means that 
\begin{equation} \label{fd:npcurved}
  x \mapsto \tfrac12 d(x,y)^2 
  \quad \text{is $1$-convex for any } y \in X.
\end{equation}

We say that $x:(0,\infty) \to X$ is an \emph{absolutely continuous curve} if there exists an $f \in L^1(0, \infty)$ such that
\begin{equation*}
  d(x(s), x(t)) \leq \int_s^t f(\tau) \, d\tau
  \quad \text{for all } 0 < s \leq t < \infty.
\end{equation*}
We denote by $AC(0,\infty; X)$ the space of absolutely continuous curves.

Given $x_\circ \in X$, we say that a curve $x \in AC_{\text{loc}} (0, \infty; X)$ is a \emph{solution} to the \emph{evolution variational inequality} if it satisfies
\begin{equation} \label{fd:EVI:X}
  \frac12 \frac d{dt} d (x(t), y)^2 + \frac\lambda2 d (x(t), y)^2 + \phi (x(t))
  \leq \phi (y),
  \quad \text{for a.e.~} t > 0 \text{, and all } y \in X,
\end{equation}
and $x(t) \to x_\circ$ as $t \to 0$.

\begin{thm}[Gradient flows {\cite[Thm.~4.0.4]{AmbrosioGigliSavare08}}] \label{t:AGS404}
Let $(X,d)$ be a complete, separable, non-positively curved, sequentially compact metric space. Let $\phi : X \to \R$ with $D(\phi) \neq \emptyset$ be $\lambda$-convex for some $\lambda \in \R$. Then for any $x_\circ \in \overline{ D(\phi) }$, the evolution variational inequality \eqref{fd:EVI:X} has a unique solution.
\end{thm}

Part of the complete statement of \cite[Thm.~4.0.4]{AmbrosioGigliSavare08} characterises the solution to the evolution variational inequality as the limit of the solutions to the corresponding time-discretised minimising movement scheme as the time step converges to $0$. This is the motivation to call the solution to the evolution variational inequality a gradient flow.

Note that in the setting of the following theorem, $(X,d)$ need not be non-positively curved.

\begin{thm}[Stability of gradient flows {\cite[Thm.~2.17]{DaneriSavareLN10}}] \label{t:DS10}
Let $\lambda \in \R$ and $(X,d)$ be a complete, separable, sequentially compact metric space. Let $\phi_n : X \to \R \cup \{\infty\}$ be a sequence of $\lambda$-convex functionals, which $\Gamma$-converges (with respect to the metric $d$) to $\phi : X \to \R \cup \{\infty\}$, where $D(\phi) \neq \emptyset$. Let $x^n_\circ \in \overline{ D \phi_n }$ converge to $x_\circ \in \overline{ D \phi }$ as $n \to \infty$. If there exists a solution $x^n (t)$ to the evolution variational inequality \eqref{fd:EVI:X} with respect to $\phi_n$ with initial condition $x^n_\circ$, then there also exists a solution $x(t)$ to the evolution variational inequality with respect to $\phi$ with initial condition $x_\circ$. Moreover,
\begin{equation*}
  x^n (t) \xto{n \to \infty} x(t),
  \quad \text{and} \quad
  \phi_n( x^n (t) ) \xto{n \to \infty} \phi( x(t) )
  \quad \text{for all } t > 0,
\end{equation*}
locally uniformly on $(0, \infty)$.
\end{thm}

\subsection{Application to dislocation walls}
\label{ss:evo:conv}

With Theorem \ref{thm:Gconv:L2} established, the main task for applying Theorem \ref{t:DS10} for proving evolutionary convergence is to construct a suitable metric space $(X,d)$, and to find minimal properties for $V$ and $W$ for which $E_n$ is $\lambda$-convex for some $n$-independent $\lambda \in \R$. 

We start by making the state space of \eqref{fd:GFn} precise. We consider any
\begin{equation} \label{fd:Omn}
  x^n
  \in \Omega_n 
  := \{ 0 \leq x_1 \leq x_2 \leq \ldots \leq x_n \leq 1 \}
  \quad \text{and} \quad
  b^n \in \{-1,1\}^n,
\end{equation}
and switch to the equivalent description in terms of $x^{n, \pm} \in \Omega_n^\pm$ (defined in \eqref{for:defn:Omeganpm}) or the empirical measures $\mu_n^\pm$ or $\bmu_n$ whenever convenient. Note that $\Omega_n$ allows for particles of the same type to be at the same position, while \eqref{fd:GFn} is ill-defined at such states. However, we consider instead the evolution variational inequality of $E_n$, which allows for such $x^n$. Moreover, we prove in \eqref{fd:EVIn:x} that any solution to \eqref{fd:GFn} satisfies the evolution variational inequality of $E_n$.

A technical difficulty for choosing $(X,d)$ is that $X$ and $d$ are not allowed to depend on $n$ in Theorem \ref{t:DS10}. While the gradient flow \eqref{fd:GFn} conserves the mass of positive particles $n^+/n$, this value can vary for different values of $n$. We account for both effects by allowing the mass of positive particles to vary in $X$, and to include the confinement to fixed mass in the energy. 

With these considerations, we choose the space $X = M([0,1])$. Setting for $\bmu, \bnu \in M([0,1])$ the masses $\sigma^\pm := \mu^\pm ([0,1])$ and $\iota^\pm := \nu^\pm ([0,1])$, we equip $M([0,1])$ with the following adjusted Wasserstein distance:
\begin{equation} \label{fd:bW:genl}
  \bW^2 (\bmu, \bnu)
  := (\sigma^+ \wedge \iota^+) W^2 \Big( \frac{\mu^+}{\sigma^+}, \frac{\nu^+}{\iota^+} \Big) + |\sigma^+ - \iota^+|
     + (\sigma^- \wedge \iota^-) W^2 \Big( \frac{\mu^-}{\sigma^-}, \frac{\nu^-}{\iota^-} \Big) + |\sigma^- - \iota^-|,
\end{equation}
where $W$ denotes the $2$-Wasserstein distance in $\mathcal P([0,1])$, and $\sigma^+, \iota^+ \in (0,1)$. We motivate the prefactor of $|\sigma^\pm - \iota^\pm|$ by
\begin{equation*} 
  1 = \max_{\mu, \nu \in \mathcal P ([0,1])} W^2 (\mu, \nu).
\end{equation*}
Since $W$ is bounded, the case $\sigma^+ = 0$ (i.e.~$\mu^+ = 0$ and $\mu^- \in \mathcal P ([0,1])$) is easily dealt with by setting
\begin{equation} \label{fd:bW:genl:si0}
  \bW^2 (\bmu, \bnu)
  := \iota^+
     + \iota^- W^2 \Big( \mu^-, \frac{\nu^-}{\iota^-} \Big) + |1 - \iota^-|.
\end{equation}
By the symmetry in the expression of $\bW$, we treat the cases $\sigma^- = 0$, $\iota^+ = 0$ or $\iota^- = 0$ similarly.

We note that on the closed subspace
\begin{equation*}
  M_{\sigma^+} ([0,1]) 
  := \big\{ \bmu \in M ([0,1]) : \mu^+ ([0,1]) = \sigma^+ \big\},
  \quad \text{for some } 0 < \sigma^+ < 1,
\end{equation*}
the expression for $\bW$ simplifies to
\begin{equation} \label{fd:bW:si}
  \bW^2 (\bmu, \bnu)
  := \sigma^+ W^2 \Big( \frac{\mu^+}{\sigma^+}, \frac{\nu^+}{\sigma^+} \Big) + \sigma^- W^2 \Big( \frac{\mu^-}{\sigma^-}, \frac{\nu^-}{\sigma^-} \Big).
\end{equation}
For $\sigma^+ \in \{0,1\}$, we identify $M_{\sigma^+} ([0,1])$ as $\mathcal P ([0,1])$ equipped with $W$.

Lemma \ref{l:MbW} lists the required properties of the space $(M ([0,1]), \bW)$. A proof is given in Appendix \ref{app:pf:of:l}.

\begin{lem}[properties of {$(M ([0,1]), \bW)$}] \label{l:MbW}
$(M ([0,1]), \bW)$ is a complete, separable, sequentially compact metric space. The closed subspace $M_{\sigma^+} ([0,1])$ is, in addition, non-positively curved for any $0 \leq \sigma^+ \leq 1$. Moreover, 
\begin{enumerate}[(i)]
  \item \label{l:MbW:sym} $\bW \big( (\mu^+, \mu^-), (\nu^+, \nu^-) \big) = \bW \big( (\mu^-, \mu^+), (\nu^-, \nu^+) \big)$;
  \item \label{l:MbW:NT} for $(\bmu_k) \subset M ([0,1])$, it holds that $\bmu_k \weakto \bmu$ if an only if $\bW(\bmu_k, \bmu) \to 0$;
  \item \label{l:MbW:Rn} let $x^n, y^n \in \Omega_n$ (defined in \eqref{fd:Omn}), and $\mu_n, \nu_n \in M_{n^+/n} ([0,1])$ be the corresponding empirical measures. Then $\bW^2 (\bmu_n, \bnu_n) = \frac1n |x^n - y^n|^2$;
  \item \label{l:MbW:geods} for any endpoints $\bmu_0, \bmu_1 \in M_{\sigma^+} ([0,1])$, all $M ([0,1])$-geodesics remains in $M_{\sigma^+} ([0,1])$.
\end{enumerate}
\end{lem}

\medskip

We continue with $\lambda$-convexity of $E_n$ on $\Omega_n$. We note that on $\R^d$, $f : \R^d \to \R$ is $\lambda$-convex if
\begin{equation*}
  x \mapsto f(x) - \tfrac\lambda2 |x|^2 
  \quad \text{is convex on } \R^d.
\end{equation*}
The following result shows how $\lambda$-convexity of $E_n$ follows from $\tilde \lambda$-convexity of $V$ and $W$:

\begin{prop}[$\lambda$-convexity of $E_n$ on $\Omega_n$] \label{p:En:lcv}
Let $V$ be $\tilde \lambda$-convex on $(0,\infty)$ and $W$ be $\tilde \lambda$-convex on $\R$ with $\tilde \lambda \leq 0$. Then $E_n$ is $\tilde \lambda_n$-convex on $\Omega_n$, with $\tilde \lambda_n := -2 \alpha_n^3 \tfrac1n \tilde\lambda$.
\end{prop}

\begin{proof}
By convexity of $V(x) - \tfrac12 \tilde\lambda x^2$ and $W(x) - \tfrac12 \tilde\lambda x^2$, it follows that
\begin{equation*}
  x^n \mapsto E_n(x^n) 
  + \frac{\tilde \lambda}2 \bigg( \sum_{p = \pm} \frac1{n^2} \sum_{i=1}^{n^p} \sum_{j = 1}^{i-1} \alpha_n \big( \alpha_n (x_i^p - x_j^p) \big)^2
  + \frac1{n^2} \sum_{i=1}^{n^+} \sum_{j = 1}^{n^-} \alpha_n \big( \alpha_n (x_i^+ - x_j^-) \big)^2 \bigg)
\end{equation*}
is convex on $\Omega_n$. It remains to compute the eigenvalues of the Hessian of the term in parentheses. Observing that 
\begin{equation*}
  \sum_{p = \pm} \frac1{n^2} \sum_{i=1}^{n^p} \sum_{j = 1}^{i-1} \alpha_n \big( \alpha_n (x_i^p - x_j^p) \big)^2
  + \frac1{n^2} \sum_{i=1}^{n^+} \sum_{j = 1}^{n^-} \alpha_n \big( \alpha_n (x_i^+ - x_j^-) \big)^2
  = \frac{\alpha_n^3}{2 n^2} \sum_{i=1}^n \sum_{j = 1}^n (x_i - x_j)^2,
\end{equation*}
we obtain that the Hessian is given by
\begin{equation*}
  \frac{2 \alpha_n^3}{n^2} (n I - \bone \otimes \bone),
\end{equation*}
where $\bone = (1, \ldots, 1)^T \in \R^n$. Hence, the eigenvalues of the Hessian are $2 \alpha_n^3 \tfrac1n$ and $0$. We conclude that $x^n \mapsto E_n(x^n) + \alpha_n^3 \tfrac1n \tilde \lambda |x^n|^2$ is convex on $\Omega_n$, and thus $E_n$ is $-2 \alpha_n^3 \tfrac1n \tilde \lambda$-convex on $\Omega_n$.
\end{proof}

We note that if $V$ is $(\tilde \lambda + b)$-convex with $b > 0$, then $\tilde \lambda_n$ in Proposition \ref{p:En:lcv} need not increase, as the interaction term in $E_n$ corresponding to the positive particles is invariant under translation of the positive particles. A similar invariance holds for the negative particles.

Adding $\tilde \lambda$-convexity to Assumption \ref{ass:VW:L2}, we obtain

\begin{ass}[Properties of $V$ and $W$ for dynamics] \label{ass:VW:L2:dyn}
There exists a $\tilde \lambda \in \R$ such that $V$ and $W$ satisfy
\begin{enumerate}[(i)]
  \item $V \in L_{\operatorname{loc}}^1 (\R)$ is even, and $\tilde \lambda$-convex on $(0, \infty)$;
  \item $W : \R \to \R$ is even and $\tilde \lambda$-convex on $\R$.
\end{enumerate}
\end{ass}

Next we show that for solutions $x^n(t)$ of \eqref{fd:GFn} (if they exist), the corresponding curve $\bmu_n(t)$ satisfies an evolution variational inequality. By $\tilde \lambda_n$-convexity on $\Omega_n$,
\begin{equation*}
  E_n(y^n) 
  \geq E_n (x^n) + (y^n - x^n) \cdot \nabla E_n(x^n) + \tfrac12 \tilde \lambda_n |x^n - y^n|^2
  \quad \text{for all } x^n, y^n \in \Omega_n,
\end{equation*}
and thus, for any solution $x^n(t)$ of \eqref{fd:GFn} and any $y^n \in \Omega_n$, we find
\begin{multline} \label{fd:EVIn:x}
  \frac1{2n} \frac d{dt} |x^n - y^n|^2
  = \frac1n (x^n - y^n) \cdot \frac {d x^n}{dt}
  = (y^n - x^n) \cdot \nabla E_n(x^n) \\
  \leq E_n(y^n) - E_n (x^n) - \tfrac12 \tilde \lambda_n |x^n - y^n|^2,
\end{multline}
which is of the form \eqref{fd:EVI:X}. Using Lemma \ref{l:MbW}.\eqref{l:MbW:Rn}, we write \eqref{fd:EVIn:x} in terms of the corresponding empirical measures $\mu_n (t), \nu_n \in M_{n^+/n} ([0,1])$. Observing that $E_n (\bnu)$ equals $\infty$ if $\bnu$ is not an empirical measure as in \eqref{f:bW:ito:mun}, we find from \eqref{fd:EVIn:x} that
\begin{multline} \label{fd:EVIn:mun}
  \frac12 \frac d{dt} \bW^2 (\bmu_n (t), \bnu) + \alpha_n^3 \tilde \lambda \bW^2 (\bmu_n (t), \bnu) + E_n (\bmu_n (t))
  \leq E_n(\bnu), \\
  \quad \text{for a.e.~} t > 0 \text{ and all } \bnu \in M_{n^+/n} ([0,1]).
\end{multline}
Below, in Theorem \ref{t:dyncs}, we prove that \eqref{fd:EVIn:mun} has a unique solution for any $\bmu_n^\circ \in M_{n^+/n} ([0,1])$, while existence and uniqueness of solutions to \eqref{fd:GFn} is not clear for all such initial data. Hence, we prefer to work with \eqref{fd:EVIn:mun} instead of \eqref{fd:GFn}. 

Next we show how $\tilde \lambda_n$-convexity of $E_n$ on $\Omega_n$ implies $\lambda_n$-convexity of $E_n$ on $M_{n^+/n} ([0,1])$. 

\begin{prop} [$\lambda$-convexity of $E_n$ and $E$ on {$M_{\sigma^+} ([0,1])$}] \label{p:EnE:lcv}
Let $V, W$ satisfy Assumption \ref{ass:VW:L2:dyn} and $\alpha_n \to \alpha > 0$. Then, setting  
\begin{equation*}
  \lambda_n := -2 \alpha_n^3 \tilde \lambda
  \quad \text{and} \quad
  \lambda := -2 \alpha^3 \tilde \lambda,
\end{equation*}
$E_n$ is $\lambda_n$-convex on $M_{n^+/n} ([0,1])$ for all $n$ large enough and all $n^+ \in \{0, \ldots, n\}$, and $E$ is $\lambda$-convex on $M_{\sigma^+} ([0,1])$ for all $0 \leq \sigma^+ \leq 1$.
\end{prop}

\begin{proof}
Proposition \ref{p:En:lcv} implies that $x \mapsto E_n(x) + \alpha_n^3 \tilde \lambda \tfrac1n |x - y|^2$ is convex for any $y \in \Omega_n$. Hence,
\begin{equation*}
  \bmu_n \mapsto E_n(\bmu_n) + \alpha_n^3 \tilde \lambda \bW^2 (\bmu_n, \bnu_n)
\end{equation*}
is convex in $M_{n^+/n} ([0,1])$ along geodesics, where $\bmu_n, \bnu_n$ are empirical measures corresponding to elements of $\Omega_n$. Since $\Gamma$-convergence conserves convexity, we obtain from Lemma \ref{l:MbW}.\eqref{l:MbW:NT} that
\begin{equation*}
  \bmu \mapsto E(\bmu) + \alpha^3 \tilde \lambda \bW^2 (\bmu, \bnu)
\end{equation*}
is convex in $M_{\sigma^+} ([0,1])$ along geodesics for any $0 \leq \sigma^+ \leq 1$ and any $\bnu \in M_{\sigma^+} ([0,1])$. Hence, $E$ is $\lambda$-convex in $M_{\sigma^+} ([0,1])$.
\end{proof}

\begin{thm}[Evolutionary convergence in the case $\alpha_n \to \alpha > 0$] \label{t:dyncs}
Let $V, W$ satisfy Assumption \ref{ass:VW:L2:dyn}. Then for any sequence $x_\circ^n \in \Omega_n$ for which the corresponding sequence of empirical measure $\bmu_n^\circ \in M_{n^+/n} ([0,1])$ converges narrowly to some $\bmu^\circ$, it holds that \eqref{fd:EVIn:mun} attains a unique solution $\bmu_n (t)$ with initial condition $\bmu_n^\circ$ for all $n \in \N$. Moreover,
\begin{equation*}
  \bmu_n (t) \xto{n \to \infty} \bmu (t),
  \quad \text{and} \quad
  E_n ( \bmu_n (t) ) \xto{n \to \infty} E( \bmu (t) )
  \quad \text{for all } t > 0,
\end{equation*}
locally uniformly on $(0, \infty)$, where $\bmu(t)$ is the unique solution to 
\begin{multline} \label{fd:EVIn:mu}
  \frac12 \frac d{dt} \bW^2 (\bmu (t), \bnu) + \alpha^3 \tilde \lambda \bW^2 (\bmu (t), \bnu) + E (\bmu (t))
  \leq E(\bnu), \\
  \text{for a.e.~} t > 0 \text{ and all } \bnu \in M_{\sigma^+} ([0,1]),
\end{multline}
with initial condition $\bmu^\circ$, and $\sigma^+ := \mu^{\circ, +} ([0,1])$.
\end{thm}

\begin{proof}
We first prove existence and uniqueness of the solution to \eqref{fd:EVIn:mun} with initial condition $\bmu_n^\circ$ by showing that Theorem \ref{t:AGS404} applies. Lemma \ref{l:MbW} implies that the space $(M_{n^+/n} ([0,1]), \bW)$ satisfies the conditions of Theorem \ref{t:AGS404}, and Proposition \ref{p:EnE:lcv} guarantees the required $\lambda_n$-convexity of $E_n$. Since $D(E_n)$ is finite whenever all particles are at different positions, it holds that $\bmu_n^\circ \in \overline{ D(E_n) }$. Hence, Theorem \ref{t:AGS404} guarantees that \eqref{fd:EVIn:mun} attains a unique solution $\bmu_n(t)$ with initial condition $\bmu_n^\circ$.

Similarly, we prove existence and uniqueness of the solution to \eqref{fd:EVIn:mu} with initial condition $\bmu^\circ$. Again, Lemma \ref{l:MbW} implies that the space $(M_{\sigma^+} ([0,1]), \bW)$ satisfies the conditions of Theorem \ref{t:AGS404}, and Proposition \ref{p:EnE:lcv} guarantees the required $\lambda$-convexity of $E$. Since
\begin{equation*}
  D(E) \supset \{ \bmu \in M_{\sigma^+} ([0,1]) : \mu^\pm \in L^\infty(0,1) \}
\end{equation*}
it holds that $\overline{ D(E) } = M_{\sigma^+} ([0,1]) \ni \bmu^\circ$. Hence, Theorem \ref{t:AGS404} guarantees that \eqref{fd:EVIn:mu} attains a unique solution $\bmu(t)$ with initial condition $\bmu^\circ$.

Next we prepare for applying Theorem \ref{t:DS10}. First, we rewrite the evolution variational inequalities \eqref{fd:EVIn:mun} and \eqref{fd:EVIn:mu} in terms of the $n$-independent space $(M ([0,1]), \bW)$, which, by Lemma \ref{l:MbW}, satisfies the condition of Theorem \ref{t:DS10}. To this aim, we set
\begin{equation*}
  \phi_n (\bmu) := E_n(\bmu) + \chi_{\{ \mu^+ = n^+/n \}}
  \quad \text{and} \quad
  \phi (\bmu) := E(\bmu) + \chi_{\{ \mu^+ = \sigma^+ \}},
\end{equation*}
where the characteristic function is given by
\begin{equation*}
  \chi_A := \left\{ \begin{aligned}
    &0
    &&\text{if } A \text{ holds,} \\
    &\infty
    &&\text{otherwise}. 
  \end{aligned} \right.
\end{equation*}
It is obvious that $\bmu_n (t)$ satisfies
\begin{multline} \label{fp:EVIn:mun}
  \frac12 \frac d{dt} \bW^2 (\bmu_n (t), \bnu) + \alpha_n^3 \tilde \lambda \bW^2 (\bmu_n (t), \bnu) + \phi_n (\bmu_n (t))
  \leq \phi_n(\bnu), \\
  \text{for a.e.~} t > 0 \text{ and all } \bnu \in M ([0,1]).
\end{multline}
However, $(M([0,1]), \bW)$ may not satisfy the conditions of Theorem \ref{t:AGS404}, and thus we use a different argument to show that \eqref{fp:EVIn:mun} has a unique solution.
Let $\tilde \bmu_n \in AC_{\text{loc}} (0, \infty; M ([0,1]))$ satisfy \eqref{fp:EVIn:mun} with initial condition $\bmu_n^\circ$. Then $\phi_n (\tilde \bmu_n (t)) < \infty$ for any $t > 0$, and thus $\tilde \bmu_n (t) \in  M_{n^+/n} ([0,1])$ for a.e.~$t > 0$. Hence, $\tilde \bmu_n$ is a solution to \eqref{fd:EVIn:mun}. Since \eqref{fd:EVIn:mun} has a unique solution, $\tilde \bmu_n = \bmu_n$. An analogous argument show that $\bmu (t)$ satisfies
\begin{equation} \label{fp:EVI:mu}
  \frac12 \frac d{dt} \bW^2 (\bmu, \bnu) + \alpha^3 \tilde \lambda \bW^2 (\bmu, \bnu) + \phi (\bmu)
  \leq \phi(\bnu),
  \quad \text{for all } t > 0, \: \bnu \in M ([0,1]),
\end{equation}
and that \eqref{fp:EVI:mu} has no other solution in $AC_{\text{loc}} (0, \infty; M ([0,1]))$ with initial condition $\bmu^\circ$.

Second, we choose $\lambda -1$ as the convexity constant. Then, for all $n$ large enough, $\lambda_n \geq \lambda - 1$. Since the existence and uniqueness of solutions to the evolution variational inequality are invariant under lowering the value of $\lambda$, \eqref{fp:EVIn:mun} and \eqref{fp:EVI:mu} still have $\bmu_n$ and $\bmu$ respectively as their unique solutions when we replace $\lambda_n$ and $\lambda$ by $\lambda - 1$.

Third, by Lemma \ref{l:MbW}.\eqref{l:MbW:geods} and 
\begin{equation*}
  \overline{ D (\phi_n) } = M_{n^+/n} ([0,1])
  \quad \text{and} \quad
  \overline{ D (\phi) } = M_{\sigma^+} ([0,1]),
\end{equation*}
$(\lambda-1)$-convexity of $\phi_n$ and $\phi$ is implied by the $\lambda_n$- and $\lambda$-convexity of $E_n$ and $E$.

Fourth, we prove $\Gamma$-convergence of $\phi_n$ to $\phi$ in the narrow topology. To establish the liminf-inequality \eqref{for:Gconv:liminf}, it is enough to consider sequences $\bnu_n$ for which $\phi_n (\bnu_n)$ is uniformly bounded. Then, $\nu_n^+ ([0,1]) = \tfrac1n n^+ \to \sigma$ as $n \to \infty$, and thus
\begin{equation*}
  \liminf_{n \to \infty} \phi_n (\bnu_n)
  \geq \liminf_{n \to \infty} E_n (\bnu_n)
  \geq E (\bnu)
  = \phi (\bnu).
\end{equation*}
The limsup-inequality \eqref{for:Gconv:limsup} follows from Theorem \ref{thm:Gconv:L2} by taking a recovery sequence for $E_n$ which satisfies $\mu_n^+ ([0,1]) = \tfrac1n n^+$. Then
\begin{equation*}
  \limsup_{n \to \infty} \phi_n (\bnu_n)
  = \limsup_{n \to \infty} E_n (\bnu_n)
  \leq E (\bnu)
  \leq \phi (\bnu).
\end{equation*}

Taking all four conditions into account, Theorem \ref{t:DS10} applies to the solutions of \eqref{fp:EVIn:mun} and \eqref{fp:EVI:mu} with $\lambda_n$ and $\lambda$ replaced by $\lambda - 1$. Since these solutions are unique and given by $\bmu_n$ and $\bmu$, the prove of the convergence statements in Theorem \ref{t:dyncs} is complete.
\end{proof}

While Theorem \ref{t:dyncs} gives a unique characterisation of the limiting curve 
\[ \bmu \in AC_{\text{loc}} \big(0, \infty; M_{\sigma^+} ([0,1]) \big),\]
it does not provide us with an explicit PDE which $\bmu$ satisfies. Next, we characterise this PDE informally. Nonetheless, the derivation is rigorous for limited choices of $V$ and $W$, which include the setting of dislocation walls in \S \ref{s:appl}.

Let us set $\alpha_n = \alpha = 1$ for convenience. We rewrite \eqref{fd:GFn} as
\begin{equation*} 
  \frac d{dt} x_i^\pm 
  = -(V' * \mu_n^\pm)(x_i^\pm) - (W' * \mu_n^\mp)(x_i^\pm) \mp \gamma_n^2,
  \quad i = 1, \ldots, n^\pm,
\end{equation*}
where we define $V'(0) := 0$. Given $\varphi^\pm \in C_c^\infty((0,\infty) \times (0,1))$, we compute from
\begin{equation*}
  0 = \frac1n \sum_{i=1}^{n^\pm} \int_0^T \frac d{dt} \varphi^\pm( t, x_i^\pm(t) ) \, dt
\end{equation*}
with Schochet's symmetrisation argument \cite{Schochet96} that $\bmu_n$ satisfies
\begin{multline} \label{f:wGFn}
  0 = \sum_{p = \pm} \bigg[ \int_0^\infty \int_0^1 \frac{ \partial \varphi^p }{ \partial t } \, d \mu_n^p dt - \int_0^\infty \bigg( 
        \iint_{[0,1]^2} V' (x-y) \frac{ (\varphi^p)'(x) - (\varphi^p)'(y) }2 \, d (\mu_n^p \otimes \mu_n^p )(x,y) \\
        + \iint_{[0,1]^2} W' (x-y) (\varphi^p)'(x) \, d (\mu_n^p \otimes \mu_n^{-p} )(x,y)
        + p \gamma_n^2 \int_0^1 (\varphi^p)'(x) \, d \mu_n^p (x) \bigg) \, dt \bigg],
\end{multline}
where $(\varphi^p)'$ denotes the spatial derivative. Assuming that $x V'(x)$ is bounded on $[-1,1]$ and continuous on $[-1,1] \setminus \{0\}$, and $W' \in C_b([-1,1])$, we can pass to the limit $n \to \infty$ in \eqref{f:wGFn} to obtain
\begin{multline} \label{f:wGF}
  0 = \sum_{p = \pm} \bigg[ \int_0^\infty \int_0^1 \frac{ \partial \varphi^p }{ \partial t } \, d \mu^p dt - \int_0^\infty \bigg( 
        \iint_{[0,1]^2} V' (x-y) \frac{ (\varphi^p)'(x) - (\varphi^p)'(y) }2 \, d (\mu^p \otimes \mu^p )(x,y) \\
        + \iint_{[0,1]^2} W' (x-y) (\varphi^p)'(x) \, d (\mu^p \otimes \mu^{-p} )(x,y)
        + p \gamma^2 \int_0^1 (\varphi^p)'(x) \, d \mu^p (x) \bigg) \, dt \bigg],
\end{multline}
which is commonly abbreviated by \eqref{f:GF}.

\section{Example of non-convergence in the case $\tfrac 1n \alpha_n \to \alpha$}
\label{s:num}

While $\Gamma$-convergence implies convergence of global minima of $E_n$ to a global minimum of $E$, it does not imply convergence of local minima of $E_n$ to a local minimum of $E$. In this section, we show that the setting of dislocation walls exhibits such an example where local minima do not converge to an extremal point of $E$. Moreover, this example is physically meaningful \cite{DoggePeerlingsGeers15b}, and adds to other known examples which show that dislocation networks cannot be fully characterised in terms of the dislocation density alone.

We start from the numerical case studies in \cite[Fig.~5,6,7]{DoggePeerlingsGeers15b}. It considers the gradient flow of $E_n$ given by \eqref{fd:GFn} with $V$ as in \eqref{for:defn:V} and $W_1$ as in \eqref{for:defn:W}. The parameters are $n^+ = n^-$, $\gamma_n = 0$ and $\alpha_n = C \sqrt n$ for some fixed $C > 0$. The initial state is fully separated \eqref{f:full:sep}.
The question in this case study is whether the long-time behaviour exhibits mixing. The conclusion from the numerical computations is that for small values of $n$, full separation is conserved in time, while for large values of $n$, mixing occurs (i.e., $\mathcal O (n^2)$ couples $(x_i^+, x_j^-)$ swap position). Mixing is also observed in \cite{DoggePeerlingsGeers15b} in their postulated $\bW$-gradient flow of $E$ given by 
\begin{equation} \label{f:GF:L3}
  \partial_t \rho^\pm 
  = \bigg[ \int_\R V \bigg] \big( \rho^\pm (\rho^\pm)' \big)' + \bigg[ \int_\R W \bigg] \big( \rho^\pm (\rho^\mp)' \big)'.
\end{equation}
We call \eqref{f:GF:L3} ``the $\bW$-gradient flow of $E$" because it is given by the formal formula (see \cite[(11.1.6)]{AmbrosioGigliSavare08}) given by
\begin{equation*}
  \partial_t \rho^\pm
  = \div \Big( \rho^\pm \nabla \frac{\delta E(\brho)}{\delta \rho^\pm} \Big),
\end{equation*}
where $\delta/\delta \rho^\pm$ denotes the $L^2$-gradient of $E$.

\subsection{Numerical observations}
\label{ss:num}

We extend the aforementioned case study in \cite{DoggePeerlingsGeers15b} by varying $\alpha_n$. We set $n^+ = n^-$, $\gamma_n = 0$, and take the equispaced initial condition
\begin{equation} \label{fd:ICnum}
  x_{\circ,i}^+ = \frac{i-1}n,
  \quad x_{\circ,i}^- = \frac12 + \frac in,
  \quad i = 1, \ldots, n^+,
\end{equation}
which is fully separated \eqref{f:full:sep}. Figure \ref{fg:al} shows the gradient flow trajectories for $n = 2^6$. These trajectories are computed with the `ode15s' solver \cite{ShampineReichelt97} in MATLAB, which is designed for stiff systems and has variable time steps. The variable time steps allow to compute the long-time behaviour of $x^n (t)$ (we take $T = 10^{10}$ as the end time) without significantly increasing the computation time.

For large values of $t$, we observe from Figure \ref{fg:al} that the case $\alpha_n = 2 n$ exhibits full separation. Moreover, the particles seem to spread out evenly, which corresponds to the continuum state $\brhosep = \big( \mathcal L_{(0, \, 1/2)} , \mathcal L_{(1/2, \, 1)} \big)$ as defined in \eqref{fd:brhosep}.

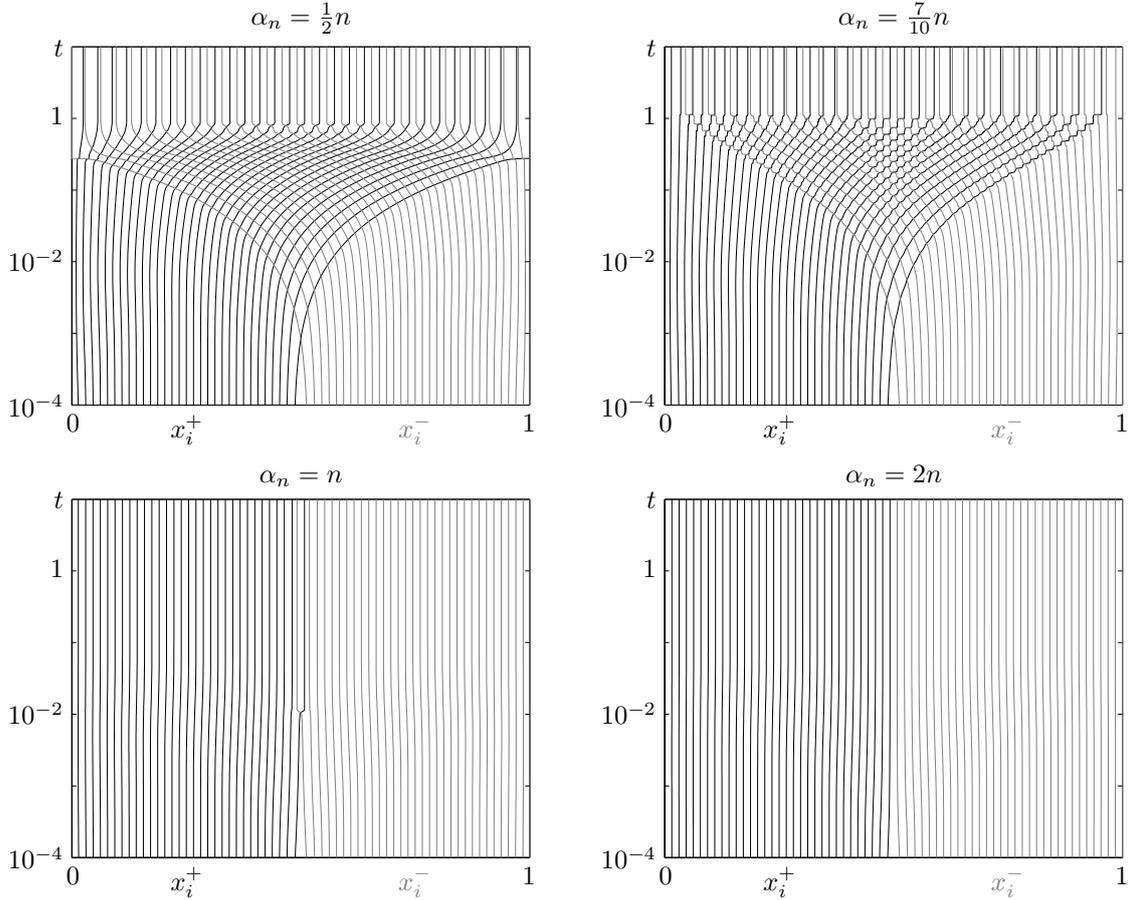
\begin{figure}[h]
\centering
\begin{tikzpicture}[scale=1.2]
\node (label) at (0,0){\includegraphics[height=4.8cm]{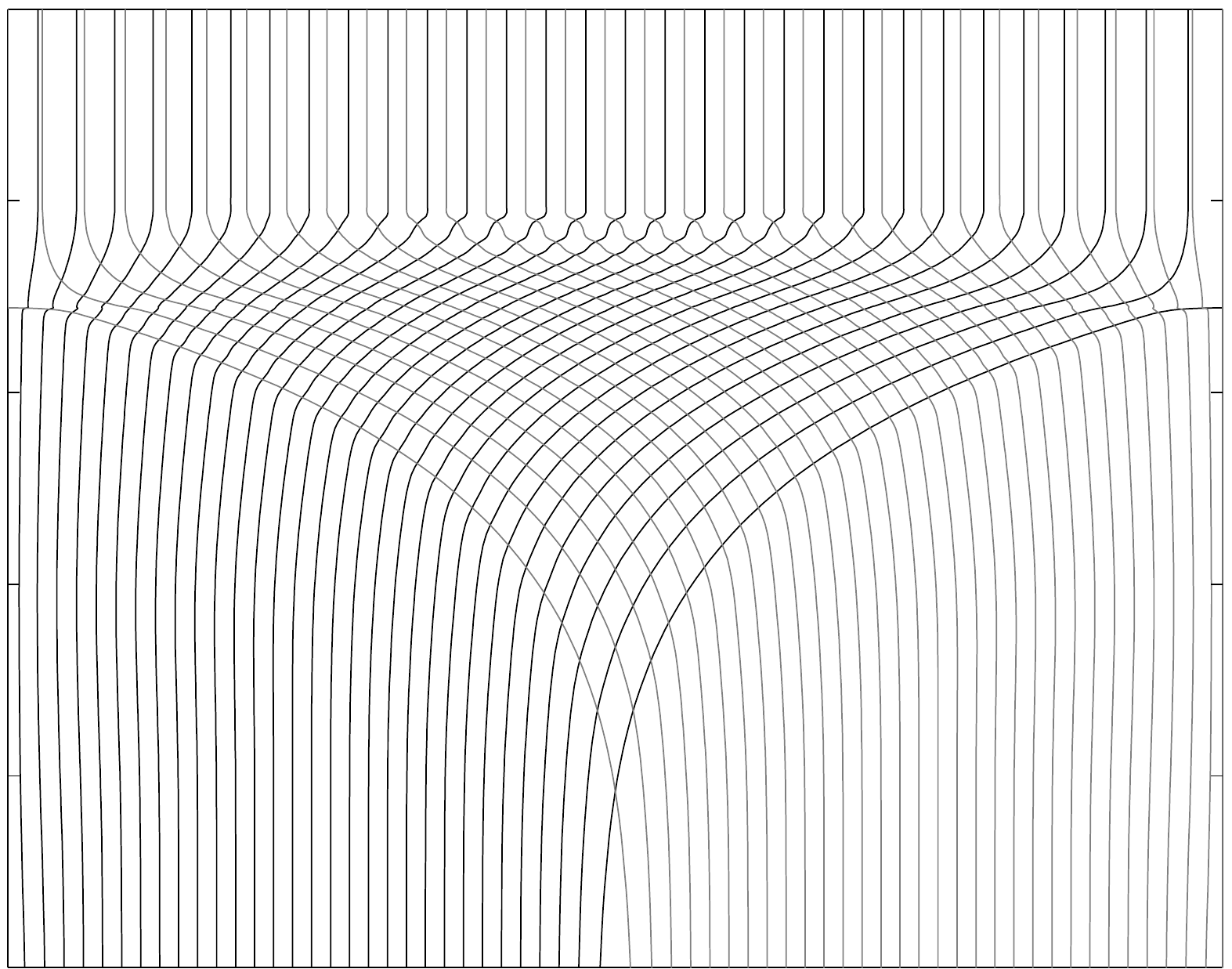}};
\draw (-2.5,-2) node[anchor = north] {$0$};
\draw (2.5,-2) node[anchor = north] {$1$};
\draw (-1.25,-2) node[anchor = north] {$x_i^+$};
\draw[gray] (1.25,-2) node[anchor = north] {$x_i^-$};
\draw (-2.5,-2) node[anchor = east] {$10^{-4}$};
\draw (-2.5,-0.4) node[anchor = east] {$10^{-2}$};
\draw (-2.5, 1.2) node[anchor = east] {$1$};
\draw (-2.5, 1.95) node[anchor = east] {$t$};
\draw (0, 2) node[anchor = south] {$\alpha_n = \tfrac12 n$};
\begin{scope}[shift={(6.5,0)},scale=1]
  \node (label) at (0,0){\includegraphics[height=4.8cm]{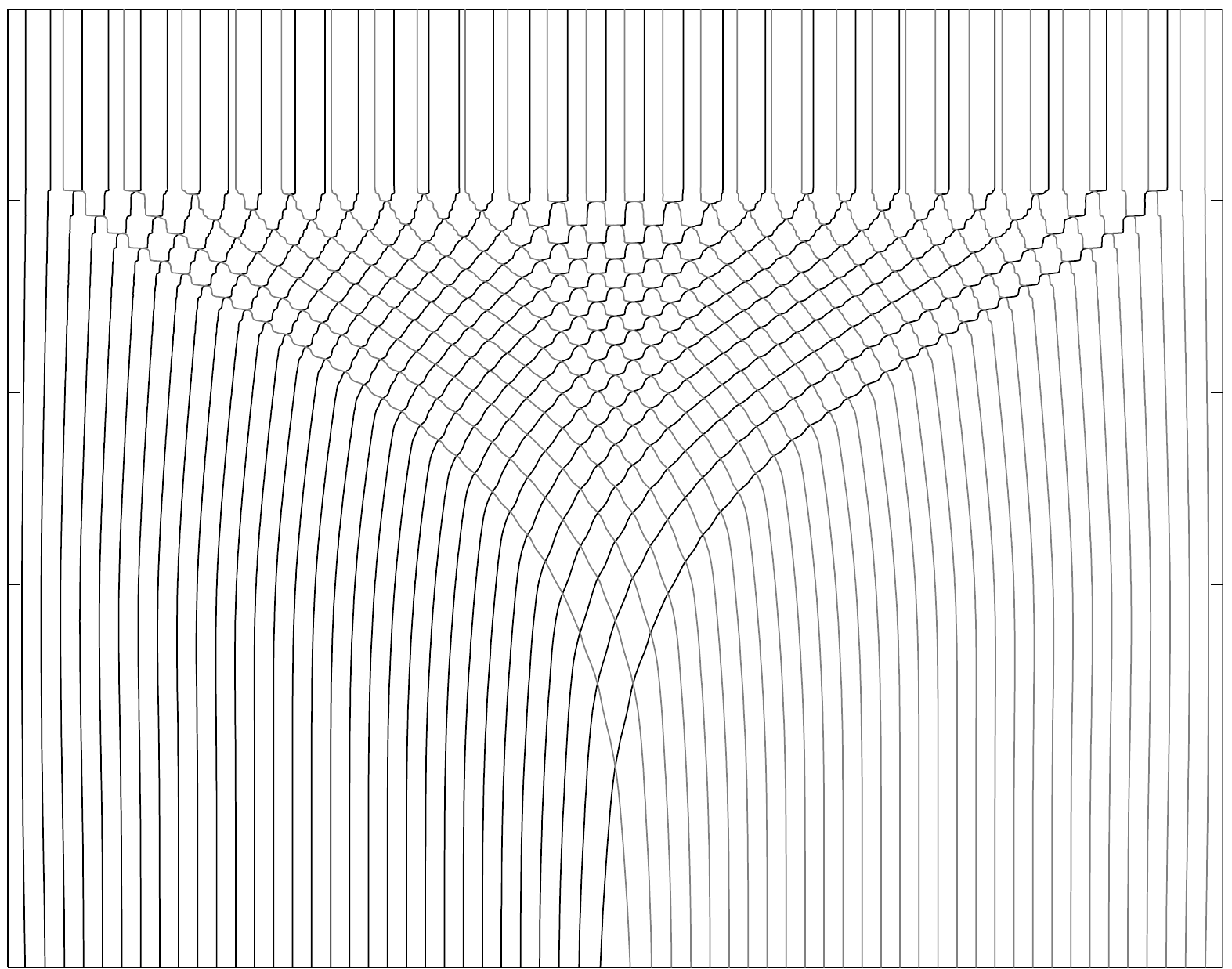}};
\draw (-2.5,-2) node[anchor = north] {$0$};
\draw (2.5,-2) node[anchor = north] {$1$};
\draw (-1.25,-2) node[anchor = north] {$x_i^+$};
\draw[gray] (1.25,-2) node[anchor = north] {$x_i^-$};
\draw (-2.5,-2) node[anchor = east] {$10^{-4}$};
\draw (-2.5,-0.4) node[anchor = east] {$10^{-2}$};
\draw (-2.5, 1.2) node[anchor = east] {$1$};
\draw (-2.5, 1.95) node[anchor = east] {$t$};
\draw (0, 2) node[anchor = south] {$\alpha_n = \tfrac7{10} n$};
\end{scope}
\begin{scope}[shift={(0,-5)},scale=1]
  \node (label) at (0,0){\includegraphics[height=4.8cm]{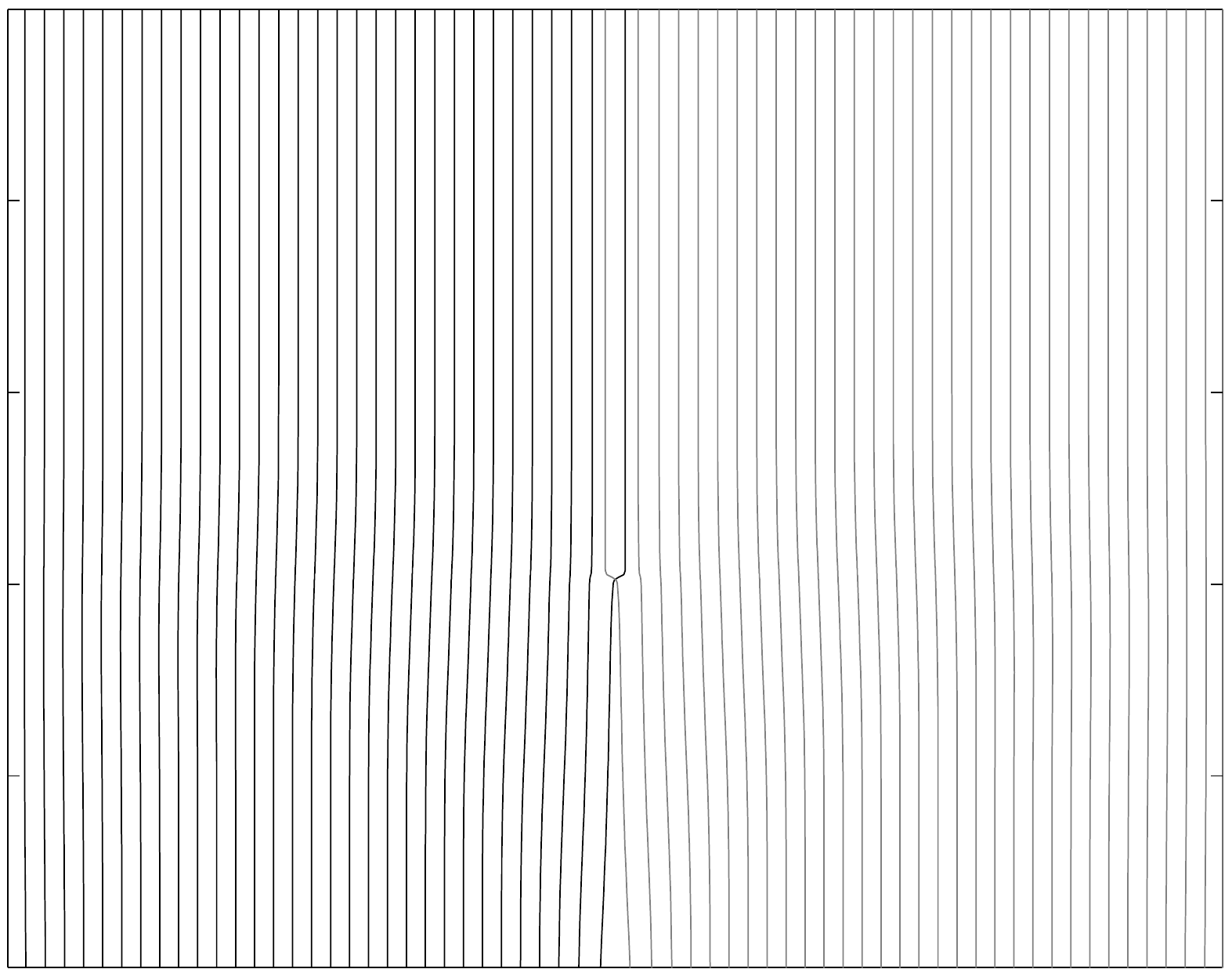}};
\draw (-2.5,-2) node[anchor = north] {$0$};
\draw (2.5,-2) node[anchor = north] {$1$};
\draw (-1.25,-2) node[anchor = north] {$x_i^+$};
\draw[gray] (1.25,-2) node[anchor = north] {$x_i^-$};
\draw (-2.5,-2) node[anchor = east] {$10^{-4}$};
\draw (-2.5,-0.4) node[anchor = east] {$10^{-2}$};
\draw (-2.5, 1.2) node[anchor = east] {$1$};
\draw (-2.5, 1.95) node[anchor = east] {$t$};
\draw (0, 2) node[anchor = south] {$\alpha_n = n$};
\end{scope}
\begin{scope}[shift={(6.5,-5)},scale=1]
  \node (label) at (0,0){\includegraphics[height=4.8cm]{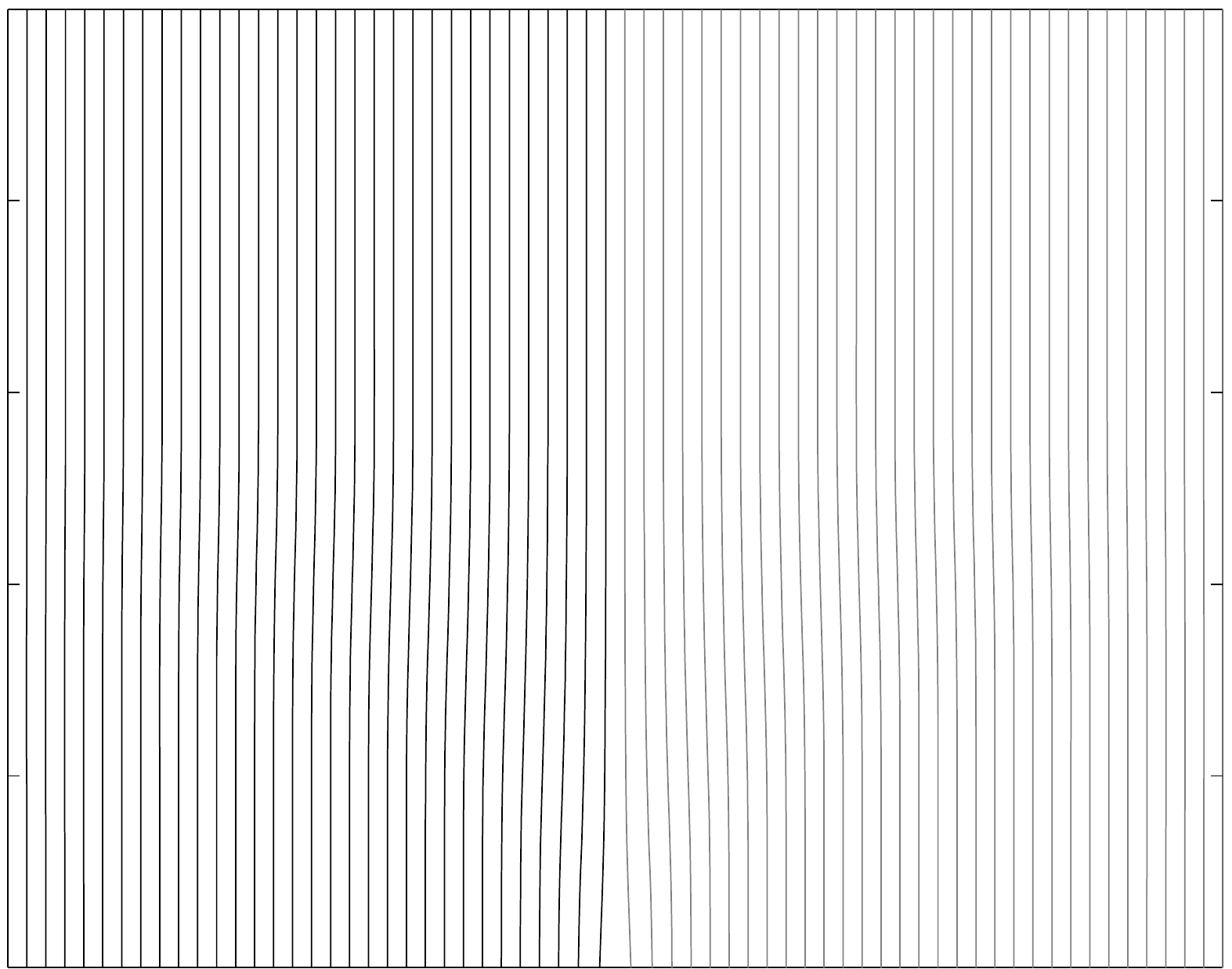}};
\draw (-2.5,-2) node[anchor = north] {$0$};
\draw (2.5,-2) node[anchor = north] {$1$};
\draw (-1.25,-2) node[anchor = north] {$x_i^+$};
\draw[gray] (1.25,-2) node[anchor = north] {$x_i^-$};
\draw (-2.5,-2) node[anchor = east] {$10^{-4}$};
\draw (-2.5,-0.4) node[anchor = east] {$10^{-2}$};
\draw (-2.5, 1.2) node[anchor = east] {$1$};
\draw (-2.5, 1.95) node[anchor = east] {$t$};
\draw (0, 2) node[anchor = south] {$\alpha_n = 2 n$};
\end{scope}
\end{tikzpicture}
\caption{Trajectories of the solutions $x_i^\pm (t)$ to the gradient flow \eqref{fd:GFn} with $n = 2^6$ and initial condition given by \eqref{fd:ICnum}. For increasing values of $\alpha_n$, less particles swap position as time passes. For $t > 10$ there is no visible time dependence on the trajectories.}
\label{fg:al}
\end{figure}

On the other end of the spectrum, the case $\alpha_n = \tfrac12 n$ exhibits complete mixing at $t = T$, i.e.,
\begin{equation} \label{f:compl:mix}
  0 = x_1^+ \leq x_1^- < x_2^+ < x_2^-  < x_3^+ < \ldots < x_{n^- - 1}^- < x_{n^+}^+ \leq x_{n^-}^- = 1.
\end{equation}
Moreover, the particles seem to spread out evenly, which corresponds to the continuum state
\begin{equation*} 
  \brhomix = \big( \tfrac12 \mathcal L_{(0, 1)} , \tfrac12 \mathcal L_{(0, 1)} \big).
\end{equation*}

In the intermediate case $\alpha_n = n$, only the middle two particles swap position. The profile at $t = T$ still corresponds to the continuum state $\brhosep$. When $\alpha_n = \tfrac7{10} n$, many particles swap positions, but not all, and thus \eqref{f:compl:mix} is not satisfied at $t = T$. Indeed, the average number of positive particles near the left barrier exceeds the average number of negative particles, which corresponds to a different continuum state than $\brhomix$.

\medskip

Next we focus on the case $\alpha_n = 2n$ at $t = T$, and check numerically whether full separation \eqref{f:full:sep} occurs for several values of $n$. To this aim, we compute 
\begin{equation} \label{f:dnpm}
   d_n^{+-} := \alpha_n \big( x_{1}^- - x_{n^+}^+ \big)
   \quad \text{and} \quad
   \bW (\bmu_n, \brhosep).
\end{equation} 
Table \ref{tb:al2} shows that $d_n^{+-} > 0$. Since the logarithmic singularity of $V$ keeps particles of the same type ordered, we conclude full separation for all values of $n$ in Table \ref{tb:al2}. Moreover, to test whether $\bW (\bmu_n, \brhosep) \to 0$ as $n \to \infty$, we compute the decay rate $q$ in $\bW (\bmu_n, \brhosep) \sim C n^{-q}$ by
\begin{equation} \label{f:qn}
  q_n = \frac{ \log \big( \bW (\bmu_n, \brhosep) \big) - \log \big( \bW (\bmu_{2n}, \brhosep) \big) }{\log 2}.
\end{equation}
Table \ref{tb:al2} shows that $q_n \approx 1$. Consequently, we expect that $\bmu_n \weakto \brhosep$ as $n \to \infty$. We discuss the meaning of the values of $\overline{\mathsf d}_n^\pm$ after proving in \S \ref{ss:al:big} that for $\alpha_n = \alpha n$ with $\alpha$ large enough, $E_n$ attains a fully separated local minimiser for all $n$.

\begin{table}[h]            
\centering                   
\begin{tabular}{cccc}   
\toprule         
$n$ & $d_n^{+-}$ & $q_n$ & $\overline{\mathsf d}_n^\pm$ \\
\midrule              
$2^4$ & $2.102$ & $1.022$ & $1.069$ \\                        
$2^5$ & $2.026$ & $1.010$ & $1.033$ \\                        
$2^6$ & $1.989$ & $1.005$ & $1.017$ \\                        
$2^7$ & $1.971$ & $1.002$ & $1.008$ \\                        
$2^8$ & $1.962$ & $1.001$ & $1.004$ \\                        
$2^9$ & $1.959$ & $1.000$ & $1.002$ \\                        
$2^{10}$ & $1.955$ & $1.001$ & $1.001$ \\                       
$2^{11}$ & $1.954$ & $1.001$ & $1.001$ \\                       
$2^{12}$ & $1.953$ & $-$ & $1.000$ \\
\bottomrule                       
\end{tabular}                
\caption{Values computed from \eqref{f:dnpm}, \eqref{f:qn} and \eqref{fd:dns} for $x^n(T)$, which is the solution at $t = T$ of the gradient flow \eqref{fd:GFn} with $\alpha_n = 2 n$ and initial condition \eqref{fd:ICnum}. $d_n^{+-} > 1.9$ means that $x^n (T)$ is fully separated \eqref{f:full:sep}, $q_n \approx 1$ suggests that $\bW( \bmu_n (T), \brhosep ) \leq C \tfrac1n$, and $\overline{\mathsf d}_n^\pm \geq 1$ means that neighbouring particles of the same type are separated by a distance of at least $\tfrac1n$ if they are located in the interval $(\tfrac14, \tfrac34)$.}     
\label{tb:al2}   
\end{table}

We repeat similar simulations for $t = T$, $\alpha_n = \tfrac12 n$ and $n = 2^4, 2^5, \ldots, 2^9$ to inspect whether complete mixing \eqref{f:compl:mix} depends on $n$. The reason for the relatively small values of $n$ is that during a swapping event of two particles, the force acting on both particles is of the order of $n$, which requires a small time step to resolve. Moreover, from Figure \ref{fg:al} we expect $\mathcal O (n^2)$ such swapping events to occur. 

For all experiments in Table \ref{tb:al05} (including $n = 2^9$) we have verified that $x^n$ is completely mixed \eqref{f:compl:mix}. Moreover, similar to \eqref{f:qn}, we compute
\begin{equation} \label{f:qnt}
  \tilde q_n = \frac{ \log \big( \bW (\bmu_n, \brhomix) \big) - \log \big( \bW (\bmu_{2n}, \brhomix) \big) }{\log 2},
\end{equation}
and speculate from Table \ref{tb:al05} that $\bmu_n \to \brhomix$ as $n \to \infty$. However, the values for $n$ remain relatively small, and we have not found a theoretical lower bound on $\alpha$ such that complete mixing occurs for $\alpha_n = \alpha n$ for all $n$ large enough.

\begin{table}[h]         
\centering                
\begin{tabular}{cc}    
\toprule
$n$ & $\tilde q_n$ \\
\midrule                    
$2^4$ & $0.940$ \\                                  
$2^5$ & $0.944$ \\                                  
$2^6$ & $0.965$ \\                                  
$2^7$ & $0.981$ \\                                  
$2^8$ & $0.990$ \\              
\bottomrule                    
\end{tabular}             
\caption{Similar to Table \ref{tb:al2}, but here with $\alpha_n = \tfrac12 n$ and $\tilde q_n$ as in \eqref{f:qnt}. $\tilde q_n \approx 1$ suggests that $\bmu_n (T)$ converges to $\brhomix$ as $n \to \infty$.}  
\label{tb:al05}
\end{table} 

To get insight in the macroscopic dynamics leading to the completely mixed state, we illustrate in Figure \ref{fg:n28} a few time slices of the piecewise-constant discrete density $\rho_n^\pm(t, x)$ given by
\begin{equation*}
  \rho_n^\pm (t, x) := \frac1{n \big( x_{i+1}^\pm(t) - x_i^\pm (t) \big) },
  \quad \text{with $i$ such that } x_i^\pm(t) < x \leq x_{i+1}^\pm(t).
\end{equation*}
The plots in Figure \ref{fg:n28} are the linear interpolations of $\rho_n^\pm$ evaluated at the midpoints, i.e.,
\begin{equation} \label{f:rhon:LI}
  \big( m_i^\pm (t), 
  \, \rho_n^\pm (t, m_i^\pm (t) ) \big)_{i=1}^{n^\pm - 1},
  \quad \text{where }
  m_i^\pm (t) := \tfrac12 \big( x_{i+1}^\pm(t) + x_i^\pm (t) \big).
\end{equation}
We observe that $(\rho_n^+ + \rho_n^-)(t)$ is not constant as a function of $x$ during the evolution, while $(\rho_n^+ + \rho_n^-)(0)$ and $(\rho_n^+ + \rho_n^-)(T)$ appear to be constant in $x$. This is in line with a locally mixed state having lower energy than a locally separated state (because of $V > W$), which allows for a denser packing of particles (as $V$ and $W$ are decreasing on $(0, \infty)$). Moreover, the spatial change from local separation to mixing is characterised by a spatial jump-discontinuity in $\rho_n^\pm (t)$. Figure \ref{fg:al} suggests that the location of this 'shock' propagates in time to the boundary, which it meets at some $t \in (\tfrac1{10}, \tfrac13)$. We expect the wiggles in the profiles of $\rho_n^\pm (t)$ close to the shock to be caused by frustration due to the difference in the local density of the positive and negative particles. Finally, we observe small boundary-layer effects close to the barriers at $x \in \{0, 1\}$. We do not study these effects here, and refer to \cite{HallHudsonVanMeurs16ArXiv} for analysis and numerics of such boundary layers at equilibrium.

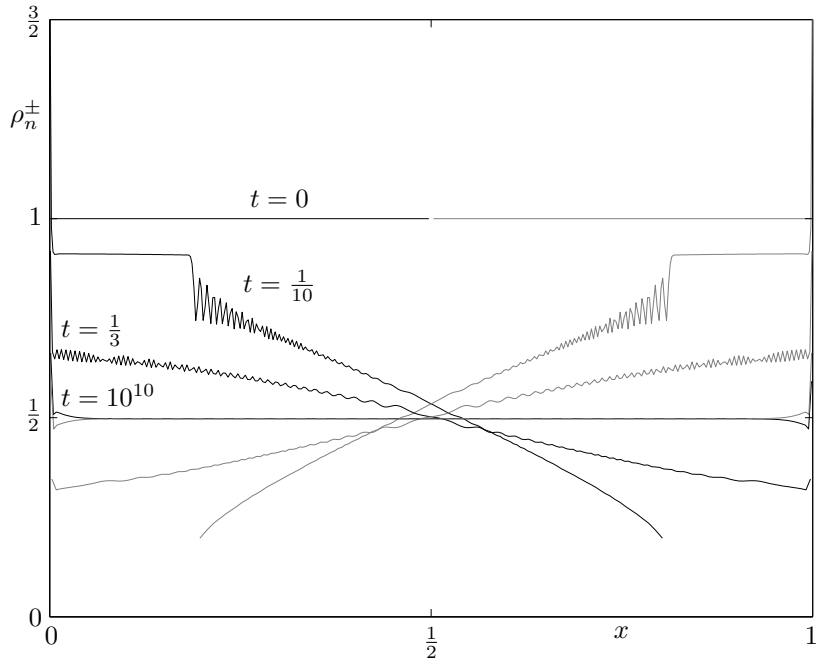
\begin{figure}[h]
\centering
\begin{tikzpicture}[scale=2]
\node (label) at (0,0){\includegraphics[height=8cm]{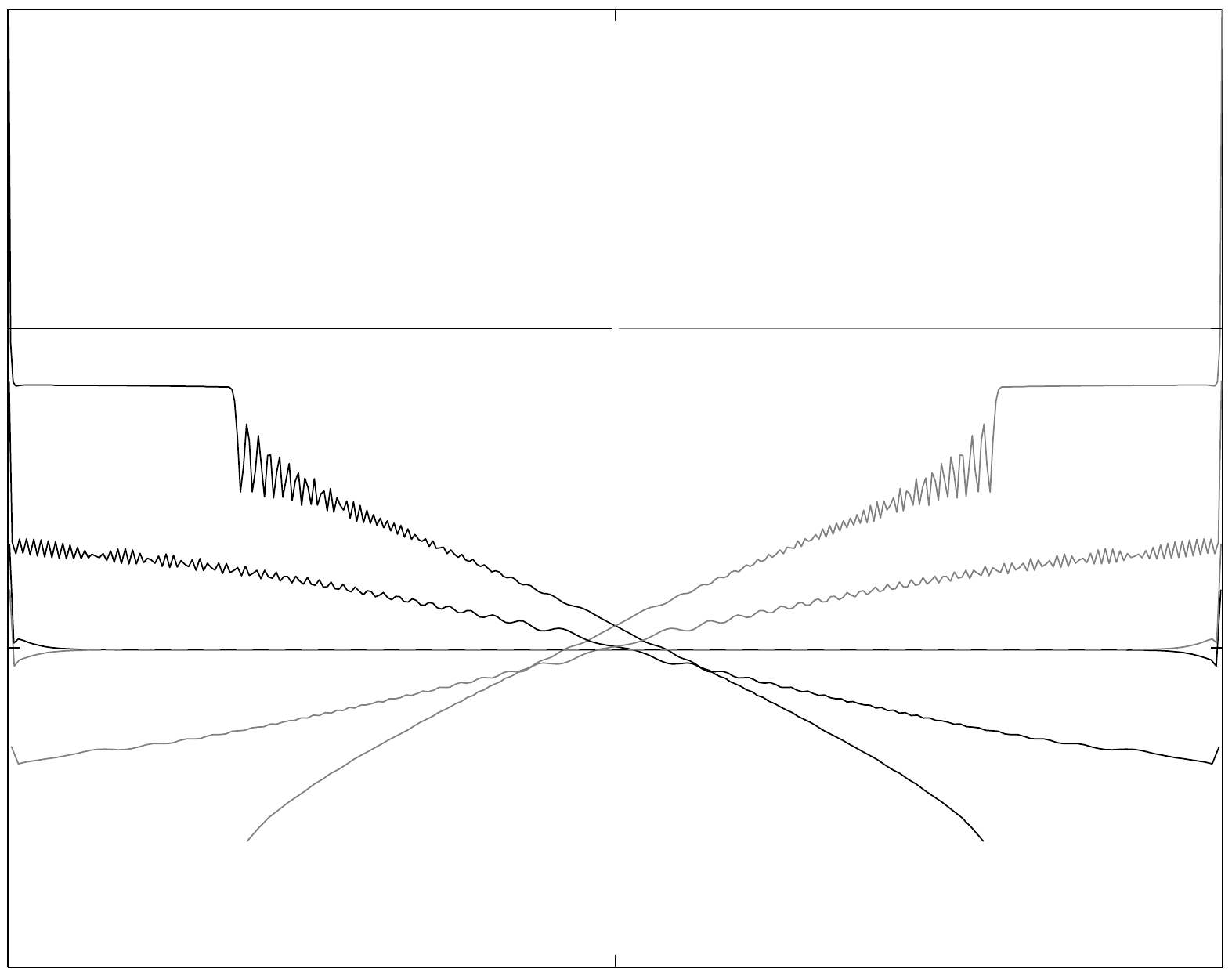}};
\draw (-2.5,-2) node[anchor = north] {$0$};
\draw (0,-2) node[anchor = north] {$\tfrac12$};
\draw (1.25,-2) node[anchor = north] {$x$};
\draw (2.5,-2) node[anchor = north] {$1$};
\draw (-2.5,-2) node[anchor = east] {$0$};
\draw (-2.5,-0.66) node[anchor = east] {$\tfrac12$};
\draw (-2.5, 0.66) node[anchor = east] {$1$};
\draw (-2.5, 1.33) node[anchor = east] {$\rho_n^\pm$};
\draw (-2.5, 1.95) node[anchor = east] {$\tfrac32$};
\draw (-1, 0.66) node[anchor = south] {$t = 0$};
\draw (-1, 0.2) node {$t = \tfrac1{10}$};
\draw (-2.5, -0.1) node[anchor = west] {$t = \tfrac13$};
\draw (-2.5, -0.66) node[anchor = south west] {$t = 10^{10}$};
\end{tikzpicture}
\caption{Several time slices of the discrete density $\rho_n^+$ (black) and $\rho_n^-$ (gray) (see \eqref{f:rhon:LI}) to the gradient flow \eqref{fd:GFn} with $n = 2^9$, $\alpha_n = \tfrac12 n$, and initial condition given by \eqref{fd:ICnum}. The profiles of $\rho_n^\pm(t)$ show typical stages of the evolution from $\brhosep$ to $\brhomix$ on the discrete level.}
\label{fg:n28}
\end{figure}

\subsection{Local minima fro $\alpha_n = \alpha n$ with $\alpha$ large}
\label{ss:al:big}

Proposition \ref{pp:sep} gives a quantitative upper bound on the asymptotic behaviour of $\alpha_n$ for which $E_n$ has a local minimum which is fully separated. For the sake of its proof and later use, we set 
\begin{equation} \label{fd:rast}
  r^* := \argmax_{(0,\infty)} | W'(r)|.
\end{equation}
We also rely on the following property of $V$, which we prove in Appendix \ref{app:VW}:
\begin{equation} \label{f:Vprop}
  \forall \, C > 0 \
  \exists \, \alpha_0 > 0 :
  \sup_{\alpha > \alpha_0} \big[ C \Veff (\alpha / C) - V(\alpha x) \big] > 0 
  \: \Longrightarrow \: x \geq \frac1{4C}.
\end{equation}

\begin{proof}[Proof of Proposition \ref{pp:sep}] 
For any $n^\pm \geq 2$ and $\alpha > r^*/n$, we define the open, convex set
\begin{equation*}
   \Omega_n^* := \big\{ (x^{n,+}, x^{n,-}) \in \Omega_n^+ \times \Omega_n^- :
   x_1^+ = 0, \: x_{n^-}^- = 1, \: d_n^{+-} > r^* \big\},
\end{equation*}
where $\Omega_n^\pm$, $d_n^{+-}$ and $r^*$ are defined in \eqref{for:defn:Omeganpm}, \eqref{f:dnpm} and \eqref{fd:rast} respectively. Since $\alpha > r^*/n$, $\Omega_n^*$ is not empty. We further note that all elements of $\Omega_n^*$ satisfy the full separation condition \eqref{f:full:sep}. We also observe that in the expression for $E_n(x^n)$ as in \eqref{for:defn:En}, the argument of $W$ remains smaller than $-r^*$ for all $x^n \in \Omega_n^*$. Since $W$ is strictly convex on $(- \infty, -r^*)$, and $V$ is strictly convex on $(0, \infty)$, we conclude that $E_n$ is strictly convex on $\Omega_n^*$. Hence, $E_n$ has a unique minimiser in $\overline \Omega_n^*$, which we set as $\overline x^n$. It remains to prove that $\overline x^n \in \Omega_n^*$, which by the strict convexity of $E_n$ on $\Omega_n^*$ implies that $\overline x^n$ is a local minimiser of $E_n$ on the full state space $\Omega_n^+ \times \Omega_n^-$.

We need several \textit{a priori} estimates on $\overline x^n$ to prove that $\overline x^n \in \Omega_n^*$. We start by showing that we can add an extra negative particle `inside' $\overline x^{n,-}$ without increasing the total energy by too much, i.e.,
\begin{multline} \label{fp:z}
  \exists \, C > 0, \, \alpha_0 > \tfrac12 r^* \
  \forall \, n^\pm \geq 2, \, \alpha > \alpha_0 \
  \exists \, z \in ( \overline x_1^-, \overline x_{n^-}^- ) : \\
  \frac \alpha n \bigg[ \sum_{ j = 1 }^{n^+} W \big( \alpha n (z - \overline x_j^+) \big) + \sum_{ j = 1 }^{n^-} V \big( \alpha n (z - \overline x_j^-) \big) \bigg] \leq \frac Cn \alpha \Veff \Big( \frac \alpha C \Big).
\end{multline}
The restriction to negative particles is not restrictive because of the symmetry of $E_n$ (see \eqref{f:En:sym}). To prove \eqref{fp:z}, we take $n^\pm \geq 2$ and $\alpha > \alpha_0$ arbitrary. We choose $z$ as one of the midpoints of $\overline x^{n, -}$, which we select by a similar argument as in Step 2 of the proof of Lemma \ref{lem:props:eita}.\eqref{lem:props:eita:cont:conv}. On the one hand, the argument is easier since we only add one particle, but on the other hand, we need uniformity of $C$ with respect to $\alpha$. The sum in \eqref{fp:z} corresponds to $\Sigma_2$ in \eqref{for:pf:lem:props:eita:47}. We construct the index sets $J_1^-$ and $J_2^-$ in a similar fashion. To establish estimates corresponding to those in \eqref{for:pf:lem:props:eita:49} and \eqref{for:pf:lem:props:eita:5}, we set $d_*^-$ as in \eqref{for:pf:lem:props:eita:48}. Since $\alpha_0 > \tfrac12 r^*$, we obtain $x_\circ^n \in \Omega_n^*$ (see \eqref{fd:ICnum}). We estimate 
\begin{equation} \label{fp:Enal:est}
  \frac\alpha5 V (\alpha n d_*^-)
  < E_n ( \overline x^n )
  \leq E_n ( x_\circ^n )
  < \frac \alpha n \sum_{i=1}^n \sum_{j=1}^{i-1} V \big( \alpha n ( x_{\circ, i} - x_{\circ, j} ) \big)
  < \alpha \sum_{i=1}^n V \big( \alpha i \big)
  < \alpha \Veff (\alpha).
\end{equation}
Applying \eqref{f:Vprop} with $C = 5$ and taking $\alpha_0$ accordingly, we conclude that
\begin{equation} \label{fp:and}
  n d_*^- \geq \tfrac1{20}, 
\end{equation}
and construct the index set $J_1^-$ as in \eqref{for:pf:lem:props:eita:7} with respect to this estimate. Regarding the index set $J_2^-$, we obtain from \eqref{fp:Enal:est} that the corresponding property of \eqref{for:pf:lem:props:eita:8} reads
\begin{multline*} 
  \sum_{ j = 1 }^{n^+} W \big( \alpha n (\overline x_\ell^- - \overline x_j^+) \big) +
  \sum_{j = 1}^{\ell - 1} V \big( \alpha n (\overline x_\ell^- - \overline x_j^- ) \big)
  + \sum_{j = \ell + 2}^{\tilde n} V \big( \alpha n (\overline x_j^- - \overline x_{\ell+1}^- ) \big) \leq 5 \Veff (\alpha) \\
  \text{for all } \ell \in J_2^-.
\end{multline*}
Then, choosing any $i \in J_1^- \cap J_2^-$ and setting the midpoint $z := \tfrac12 (\overline x_{i+1}^- + \overline x_i^-)$, we estimate
\begin{align*}
  &\frac \alpha n \bigg[ \sum_{ j = 1 }^{n^+} W \big( \alpha n (z - \overline x_j^+) \big) + \sum_{ j = 1 }^{n^-} V \big( \alpha n (z - \overline x_j^-) \big) \bigg] \\
  &\leq \frac \alpha n \bigg[ \sum_{ j = 1 }^{n^+} W \big( \alpha n (\overline x_i^- - \overline x_j^+) \big) + \sum_{ j = 1 }^{i-1} V \big( \alpha n (\overline x_i^- - \overline x_j^-) \big) \\
  &\qquad \qquad + 2 V \big( \alpha n \tfrac12(\overline x_{i+1}^- - \overline x_i^-) \big) + \sum_{ j = i+2 }^{n^-} V \big( \alpha n (\overline x_j^- - \overline x_{i+1}^-) \big) \bigg] \\
  &\leq \frac \alpha n \big( 5 \Veff (\alpha) + 2 V ( \alpha n \tfrac12 d_*^- ) \big).
\end{align*}
Using \eqref{fp:and}, we continue the estimate by
\begin{equation*}
  \frac \alpha n \big( 5 \Veff (\alpha) + 2 V ( \alpha n \tfrac12 d_*^- ) \big)
  \leq \frac \alpha n \big( 5 \Veff (\alpha) + 2 V ( \tfrac1{40} \alpha ) \big)
  \leq \frac7n \alpha \Veff (\tfrac1{40} \alpha).
\end{equation*}
We conclude \eqref{fp:z}.

Using \eqref{fp:z}, we prove the following lower bound on the distance between neighbouring negative particles:
\begin{equation} \label{fp:dn}
  \exists \, c, \alpha_0 > 0 \
  \forall \, \alpha \geq \alpha_0, \, n^\pm \geq 2 :
  \overline d_n^- := \min_{1 \leq i \leq n^- -1} n ( \overline x_{i+1}^- - \overline x_i^-) \geq c.
\end{equation}
Given $\alpha \geq \alpha_0$ with $\alpha_0$ as in \eqref{fp:z} and $n^\pm \geq 2$, we derive this estimate by moving the particle $\overline x_i^-$ with any index $i$ such that  
\begin{equation*}
  \overline d_n^- = n ( \overline x_{i+1}^- - \overline x_{i}^-),
\end{equation*}
to the position $z$ provided by \eqref{fp:z}. This yields (with abuse of notation)
\begin{multline*}
  0
  \leq E_n \big( \overline x^{n, +}, ( \overline x_1^-, \ldots, \overline x_{i-1}^-, z, \overline x_{i + 1}^-, \ldots, \overline x_{n^-}^- )^T \big) - E_n (\overline x^n) \\
  = \frac \alpha n \sum_{ \substack{ j = 1 \\ j \neq i } }^n V_{i j} \big( \alpha n (z - \overline x_j) \big) - \frac \alpha n \sum_{ \substack{ j = 1 \\ j \neq i } }^n V_{i j} \big( \alpha n (\overline x_{i}^- - \overline x_j) \big)
  < \frac Cn \alpha \Veff \Big( \frac \alpha C \Big) - \frac\alpha n V ( \alpha \overline d_n^- ).
\end{multline*}
Then, choosing $\alpha_0$ large enough such that \eqref{f:Vprop} applies, we obtain $\overline d_n^- \geq 1/(4C)$. We conclude that \eqref{fp:dn} holds.

Finally, we show that $\overline x^n \in \Omega_n^*$. By the singularity of $V$, it is enough to show that $\overline d_n^{+-} > r^*$. We reason by contradiction, and suppose that $\overline d_n^{+-} = r^*$. Treating $y = \overline x_1^-$ as a variable, we compute
\begin{equation} \label{p:Fy}
  \frac d{dy} \Big|_{y = \overline x_1^-} E_n (\overline x^n) 
  = \frac{\alpha_n^2}{n^2} \sum_{i=2}^{n^-} -V' \big( \alpha_n (\overline x_i^- - \overline x_1^-) \big) 
    - \frac{\alpha_n^2}{n^2} \sum_{i=1}^{n^+} -W' \big( \alpha_n (\overline x_1^- - \overline x_i^+) \big).
\end{equation}
We show that the right-hand side of \eqref{p:Fy} is negative, which contradicts the minimality of $\overline x^n \in \overline \Omega_n^*$. Noting that $(-V'), (-W') > 0$ on $(0, \infty)$ and $(-V')$ is decreasing on $(0, \infty)$, we use \eqref{fp:dn} to estimate
\begin{multline} \label{p:Fy:est}
  \sum_{i=2}^{n^-} -V' \big( \alpha_n (\overline x_i^- - \overline x_1^-) \big) 
    - \sum_{i=1}^{n^+} -W' \big( \alpha_n (\overline x_1^- - \overline x_i^+) \big)
    < \bigg[ \sum_{k=1}^\infty -V' \big( \alpha \overline d_n^- k \big) \bigg]
    -  (-W') ( \overline d_n^{+-} ) \\
    < \bigg[ \sum_{k=1}^\infty -V' \big( \alpha c k \big) \bigg]
    - | W' ( r^* ) |.
\end{multline}
Since $-V' \in L^1(1, \infty)$ is decreasing, the right-hand side of \eqref{p:Fy:est} is negative for all $\alpha$ large enough. For any such $\alpha$, the right-hand side of \eqref{p:Fy} is negative too. The contradiction is reached, and we conclude that $\overline x^n \in \Omega_n^*$.
\end{proof}

We reflect back on the numerical results in Table \ref{tb:al2} with $\alpha = 2$. For the proof of Proposition \ref{pp:Esep} to apply to $\alpha = 2$, it is sufficient to derive a similar estimate as \eqref{p:Fy:est}. Table \ref{tb:al2} suggests that 
\begin{equation} \label{fd:dns}
  \overline{\mathsf d}_n^- := \min_{1 \leq i \leq n^-/2} n ( \overline x_{i+1}^- - \overline x_i^-) > 0.99
  \quad \text{for all } n^- \geq 8.
\end{equation}
We choose a slightly different definition of $\overline{\mathsf d}_n^-$ then in \eqref{fp:dn} to avoid boundary-layer effects at the barrier at $1$. Then, we redo the estimate in \eqref{p:Fy:est} to find
\begin{equation*} 
  \sum_{k = 1}^\infty | V'(\alpha \overline{\mathsf d}_n^- k) | 
  \approx \sum_{k = 1}^\infty | V'(1.98 k) | 
  \approx 0.1576 
  < 0.4477 
  \approx \big| W' (r^*) \big|.
\end{equation*}
Hence, given that the estimate in \eqref{fd:dns} holds, the right-hand side in \eqref{p:Fy:est} is negative, and thus $E_n$ has a local minimiser which is fully separated.

Supplementary to Proposition \ref{pp:sep}, we expect from our simulations in \S \ref{ss:num} with $\alpha_n = \tfrac12 n$ that complete mixing \eqref{f:compl:mix} occurs when $\alpha > 0$ is small enough. However, due to dynamical effects and the lack of convexity, this statement is hard to prove.

\subsection{No separation in continuum energy}

Proposition \ref{pp:Esep} shows that the continuum analogue of Proposition \ref{pp:sep} does \emph{not} hold for \emph{any} $\alpha > 0$.

\begin{defn}[Metric slope] Let $(X,d)$ be a complete metric space, and $\phi : X \to (-\infty, \infty]$ with domain $D(\phi) := \{x \in X : \phi(x) < \infty\} \neq \emptyset$. The metric slope of $\phi$ (with respect to the metric $d$)  is
\begin{equation*}
  |\partial \phi|_d : D(\phi) \to [0, \infty],
  \quad |\partial \phi|_d (x) := \limsup_{y \to x} \frac{ [E(x) - E(y)]_+ }{ d(x,y) }.
\end{equation*}
\end{defn}

\begin{prop} [$E$ not critical at $\brhosep$] \label{pp:Esep}
In the scaling regime $\tfrac1n \alpha_n \to \alpha > 0$ it holds that $|\partial E|_\bW (\brhosep) = \infty$. Moreover, there is a finite speed curve in $(M_{1/2}([0,1]), \, \bW)$ connecting the separated state $\brhosep$ to the mixed state $\brhomix := (\tfrac12 \mathcal L_{(0, 1)}, \tfrac12 \mathcal L_{(0, 1)})$ along which $E$ is strictly decreasing.
\end{prop}

\begin{proof}
We set
\begin{equation*}
  \rho_t := \mathcal L_{(0, \, \frac12 (1 - t))} + \tfrac12 \mathcal L_{(\frac12 (1 - t), \, \frac12 (1 + t) )},
  \quad 0 \leq t \leq 1,
\end{equation*}
and note that the curve
\begin{equation*}
  \brho_t := (\rho_t, \, 1 - \rho_t) \in M_{1/2}([0,1]),
  \quad \text{for all } 0 \leq t \leq 1
\end{equation*}
satisfies the end conditions $\brho_0 = \brhosep$ and $\brho_1 = \brhomix$. The energy along the curve $\brho_t$ is given by
\begin{equation*}
  E (\brho_t) 
  = (1-t) \eita (1,0) + t \eita (\tfrac12, \tfrac12).
\end{equation*}
Since $V > W$, it follows easily that the inequality in Lemma \ref{lem:props:eita}.\eqref{lem:props:eita:UB} is strict, and thus
\begin{equation*}
  \frac d{dt} E (\brho_t) 
  = \eita (\tfrac12, \tfrac12) - \eita (1,0)
  < 0.
\end{equation*}
We claim that 
\begin{equation} \label{p:bW:brhot}
  \bW (\brho_t, \brho_s)
  = \frac12 \sqrt{ \frac{t + 2s}3 } (t-s),
  \quad 0 \leq s \leq t \leq 1.
\end{equation} 
This claim implies Proposition \ref{pp:Esep}, because it shows that $\brho_t$ is a finite speed curve in \\ $(M_{1/2}([0,1]), \, \bW)$ connecting $\brhosep$ to $\brhomix$, and
\begin{equation*}
  |\partial E|_\bW (\brhosep) 
  = \limsup_{\bmu \to \brhosep} \frac{ [E(\brhosep) - E(\bmu)]_+ }{ \bW( \brhosep, \bmu ) }
  \geq \limsup_{t \downarrow 0} \frac{ [E(\brho_0) - E(\brho_t)]_+ }{ \bW( \brho_0, \brho_t ) } = \infty.
\end{equation*}

It remains to prove the claim \eqref{p:bW:brhot}. By the symmetry in $\rho_t$ and $1 - \rho_t$, we obtain
\begin{equation*}
  \bW^2 (\brho_t, \brho_s)
  = \tfrac12 W^2 (2\rho_t,  2\rho_s) + \tfrac12 W^2 \big( 2(1-\rho_t),  2(1-\rho_s) \big) 
  = W^2 (2\rho_t,  2\rho_s).
\end{equation*}
We observe that
\begin{equation*}
  X_t (\xi) := \left\{ \begin{aligned}
    &\tfrac12 \xi
    &&0 < \xi < 1 - t \\
    &\xi - \tfrac12 (1-t)
    && 1-t < \xi < 1 
  \end{aligned} \right.
\end{equation*}
is the inverse of the cumulative distribution of $2\rho_t$, i.e., $(X_t^{-1})' = 2 \rho_t$. We use \cite[Thm.~2.18]{Villani03} to rewrite
\begin{equation*}
  W^2 (2\rho_t, \, 2\rho_s)
  = \int_0^1 |X_t (\xi) - X_s (\xi)|^2 \, d \xi.
\end{equation*}
We continue to compute $W^2 (2\rho_t, \, 2\rho_s)$ by
\begin{multline*}
  \int_0^1 |X_t (\xi) - X_s (\xi)|^2 \, d \xi
  = \int_{1-t}^{1-s} \frac14 \big( \xi - (1-t) \big)^2 \, d \xi + \int_{1-s}^1 \frac14  (t-s)^2 \, d \xi \\
  = \frac14\int_0^{t-s} \eta^2 \, d \eta + \frac s4  (t-s)^2
  = \frac{t + 2s}{12} (t-s)^2. \qedhere
\end{multline*}
\end{proof}

\section{Acknowledgements}

PvM is supported by the International Research Fellowship of the Japanese Society for the Promotion of Science, together with the JSPS KAKENHI grant 15F15019. PvM wishes to thank M.~A.~Peletier and R.~H.~J.~Peerlings for the modelling process which has led to $E_n$, and L.~Scardia for fruitful discussions and feedback on earlier versions of this paper.

\appendix

\section{Properties of $V$ \eqref{for:defn:V} and $W$ \eqref{for:defn:W}}
\label{app:VW}

In this section we prove Proposition \ref{prop:props:Wa} and \eqref{f:Vprop}.

\begin{proof}[Proof of Proposition \ref{prop:props:Wa}] We compute
\begin{gather*}
  W_a' (r) = -2r \frac{ 1 + a \cosh (2r) }{ ( \cosh (2r) + a )^2 },
  \quad W_a'' (r) = 4 r \sinh (2r) \frac{ a \cosh (2r) + 2 - a^2 }{ ( \cosh (2r) + a )^3 } - 2 \frac{ 1 + a \cosh (2r) }{ ( \cosh (2r) + a )^2 }.
\end{gather*}
In particular, for $V$ and $a \in \{0, 1\}$, we use the doubling formula's to simplify
\begin{align} \label{for:der:V:W0:W1}
  V'(r) &= - \frac r{\sinh^2 r},
  \qquad W_1' (r) = - \frac r{\cosh^2 r} = W_0' \Big( \frac r2 \Big), \\\notag
  V''(r) &=  \frac{2 r \coth r - 1}{\sinh^2 r}. 
\end{align}

Next we prove (i). Using that $\cosh (2r) + a \geq 1 + a > 0$, we can write $W_a$ as a composition of smooth functions. Hence, $W_a$ is smooth. Exponential decay of the tails of $W_a$ is shown in [thesis, Prop.~A.3.1]. The proof relies on basic Taylor expansions.

We continue with proving (ii). The ordering of the potentials $W_a$ in $a$ follows from 
  $$
    \frac d{da} W_a (r) 
    = \frac{ \tfrac12 }{ \cosh (2r) + a } + \frac{ r \sinh (2r) }{ ( \cosh (2r) + a )^2 }
    > 0.
  $$ 
  Since $W_a(r) \to 0$ as $r \to \infty$, we conclude $0 < W_0$ from \eqref{for:der:V:W0:W1} and $W_1 < V$ from $W_1' > V'$ \eqref{for:der:V:W0:W1}.

Next we prove (iii). Writing $\cosh(2r) = 2 \sinh^2 r + 1$, a straight-forward computation yields
\begin{align} \notag
  &\big( V''(r) - W_a''(r) \big) \big( \cosh(2r) + a \big)^3 \sinh^2 r \\\notag
  &= \big( \cosh(2 r) + a \big) \big( 4 (1-a) \sinh^4 r + (1+a)^2 \big) (2 r \coth r - 1) \\ \label{fp:VWapp}
  &\quad + 4 (1+a) ( 4 a r \coth r - 1 ) \sinh^4 r 
    + 2 (1+a)^2 ( 4 r \coth r - 1 ) \sinh^2 r.
\end{align} 
Except for $( 4 a r \coth r - 1 )$, it follows from $\inf_{r > 0} r \coth r = 1$ that all three terms in the right-hand side of \eqref{fp:VWapp} are non-negative for all $r \in \R$ and all $0 \leq a \leq 1$. We transfer a part of the first term to the second term such that both terms are non-negative. To this aim, we estimate the first term from below by
\begin{align*}
  &\big( \cosh(2 r) + a \big) \big( 4 (1-a) \sinh^4 r + (1+a)^2 \big) (2 r \coth r - 1) \\
  &\geq \big( \cosh(2 r) + a \big) \big( 4 (1-a) \sinh^4 r + (1+a)^2 \big) (r \coth r - 1)
       + 4 (1+a) ( 1-a ) r \coth r \sinh^4 r.
\end{align*}
Using this estimate in \eqref{fp:VWapp}, we obtain
\begin{align*}
  &\big( V''(r) - W_a''(r) \big) \big( \cosh(2r) + a \big)^3 \sinh^2 r \\\notag
  &\geq \big( \cosh(2 r) + a \big) \big( 4 (1-a) \sinh^4 r + (1+a)^2 \big) (r \coth r - 1) \\ 
   &\quad + 4 (1+a) \big( (1 + 3a) r \coth r - 1 \big) \sinh^4 r 
    + 2 (1+a)^2 ( 4 r \coth r - 1 ) \sinh^2 r,
\end{align*}
and observe that all terms are positive for all $0 \leq a \leq 1$ and all $r > 0$. Hence, $(V - W_a)'' > 0$ on $(0, \infty)$ for all $0 \leq a \leq 1$.


Finally, we prove (iv) by computing $\widehat{W_a}$ explicitly. First, we focus on $0 \leq a < 1$. We use several common Fourier calculus relations to decompose $\widehat{W_a}$ in several functions to which we can apply [Erd\'elyi I \S 1.9 (6)]:
  \begin{align*} 
    \mathcal F \Big( \frac1{ \cosh (2r) + a } \Big) (\omega)
    = \frac{\pi}{\sqrt{1 - a^2}} \frac{ \sinh ( \pi \arccos (a) \omega ) }{ \sinh ( \pi^2 \omega ) }
    > 0.
  \end{align*}
We compute
  \begin{multline} \label{pf:prop:props:Wa:1}
    \widehat{W_a (r)} 
    = \frac1{2 \pi^2 \omega} \bigg[ \mathcal F \Big( \frac{ W_a' (r) }{2r} \Big) \bigg]'
    = \frac{-1}{2 \pi^2 \omega} \bigg[ \mathcal F \Big( \frac a{ \cosh (2r) + a } + \frac{ 1 - a^2 }{ ( \cosh (2r) + a )^2 } \Big) \bigg]' \\
    = \frac{-1}{2\pi \omega} \frac a{\sqrt{1 - a^2}} \frac d{d \omega} \Big( \frac{ \sinh ( \alpha \omega ) }{ \sinh ( \beta \omega ) } \Big) - \frac1{2\omega} \frac d{d \omega} \Big( \frac{ \sinh ( \alpha \omega ) }{ \sinh ( \beta \omega ) } \Big) \ast \Big( \frac{ \sinh ( \alpha \omega ) }{ \sinh ( \beta \omega ) } \Big), 
  \end{multline} 
  where $\alpha := \pi \arccos (a) < \pi^2 =: \beta$. Since $f(\omega) := \sinh ( \alpha \omega ) / \sinh ( \beta \omega )$ is even, it is sufficient to show that both terms in the right-hand side of \eqref{pf:prop:props:Wa:1} are positive for $\omega > 0$. To this aim, we compute
  \begin{align*}
    f' (\omega) 
    &= \frac{ \sinh ( \beta \omega ) \, \alpha \cosh ( \alpha \omega ) - \sinh ( \alpha \omega ) \, \beta \cosh ( \beta \omega ) }{ \sinh^2 ( \beta \omega ) } \\
    &= \alpha \frac{ \cosh ( \alpha \omega ) }{ \sinh ( \beta \omega ) } \Big( 1 - \frac{ \tanh ( \alpha \omega ) / ( \alpha \omega ) }{ \tanh ( \beta \omega ) / ( \beta \omega ) } \Big),
  \end{align*}
  which is negative for $\omega > 0$, since $\alpha < \beta$ and $x \mapsto \tanh (x) / x$ is decreasing on $(0, \infty)$. 
  Hence, the first term in the right-hand side of \eqref{pf:prop:props:Wa:1} yields a positive contribution. Since $f$ is moreover integrable, it holds that $f \ast f$ is decreasing on $(0,\infty)$, 
  from which we infer positivity of the second term in the right-hand side of \eqref{pf:prop:props:Wa:1}.

The case $a = 1$ simply follows from \eqref{for:der:V:W0:W1}:
\begin{equation*}
  \widehat{W_1 (r)} = 2 \widehat{W_0 (\tfrac r2)}  = \widehat{W_0} (\tfrac\omega2) > 0. \qedhere
\end{equation*}
\end{proof}

\begin{proof}[Proof of \eqref{f:Vprop}] 
  We start by deriving a few estimates on $V$. \cite[Prop.~A.3.1]{VanMeurs15} states that
  \begin{equation*}
    V(r) = 2 r e^{-2r} + \mathcal O (r e^{-4 r}).
  \end{equation*}
  Hence, for all $r$ large enough, we have $V(r) > \exp(- \tfrac32 r)$, and thus
  \begin{equation} \label{fp:V:UB}
    \Veff (r)  \leq \sum_{k=1}^\infty \exp(- \tfrac32 r)^k = \frac{ \exp(- \tfrac32 r) }{ 1 - \exp(- \tfrac32 r) } \leq e^{-r}
    \quad \text{for all $r$ large enough.} 
  \end{equation}

We also prove the following lower bound:
\begin{equation} \label{fp:V:LB}
  V(r) - e^{-2r} > 0
  \quad \text{for all } r > 0.
\end{equation}
Since $V(r) - e^{-2r} \to 0$ as $r \to \infty$, it is enough to show that $\frac d{dr} (V(r) - e^{-2r}) < 0$ on $(0, \infty)$. Recalling \eqref{for:der:V:W0:W1}, we compute
\begin{multline*}
  \frac d{dr} (V(r) - e^{-2r})
  = - \frac r { \sinh^2 r } + 2 e^{-2r} 
  = \frac{ 2(1 - e^{-2r})^2 - 4r }{(e^r - e^{-r})^2} \\
  \leq \frac2{(e^r - e^{-r})^2} \left\{ \begin{array}{ll}
  (2r)^2 - 2r, & 0 < r < \tfrac12 \\
  (1 - \tfrac1e)^2 - 1, & r = \tfrac12 \\
  1 - 2r, & \tfrac12 < r \rule[-0.5ex]{0pt}{0pt} 
  \end{array} \right\}
  < 0.
\end{multline*}

Next we prove \eqref{f:Vprop}. We fix any $C > 0$, set $\alpha_0 \geq 2 C \log C$ such that \eqref{fp:V:UB} holds for all $r \geq \alpha_0 / C$, and take any $0 < x < 1/(4C)$. Then, by \eqref{fp:V:UB} and \eqref{fp:V:LB} we have
\begin{equation} \label{fp:VeffV}
  \sup_{\alpha > \alpha_0} \Big[ C \Veff \Big( \frac \alpha C \Big) - V (\alpha x) \Big]
  \leq \sup_{\alpha > \alpha_0} \Big[ \underbrace{ \exp \Big( \log C - \frac \alpha C \Big) - \exp (- 2 \alpha x) }_{ \varphi (\alpha, x) } \Big].
\end{equation}
Since $x < 1/(2C)$, the second exponential in $\varphi (\alpha, x)$ decays slower as $\alpha \to \infty$ than the first exponential. In particular, if $C \leq 1$, then $\sup_{\alpha > 0} \varphi (\alpha) = 0$, and thus it suffices to assume $C > 1$. For $C > 1$, we note that the equation $\varphi (\alpha, x) = 0$ has a unique solution $\alpha^x$, which is given by
\begin{equation*}
  \alpha^x := \frac{C \log C}{ 1 - 2Cx } < 2 C \log C \leq \alpha_0.
\end{equation*}
Moreover, $\sup_{\alpha > \alpha^x} \varphi(\alpha, x) = 0$. Hence, if we further assume $x < 1/(4C)$, then $\alpha^x \leq 2 C \log C \leq \alpha_0$, and thus
\begin{equation*}
  0
  = \sup_{\alpha > \alpha^x} \varphi(\alpha, x)
  \geq  \sup_{\alpha > \alpha_0} \varphi (\alpha, x).
\end{equation*}
Together with \eqref{fp:VeffV}, we conclude \eqref{f:Vprop}.
\end{proof}

\section{Proof of Lemma \ref{l:MbW}}
\label{app:pf:of:l}

\begin{proof}[Proof of Lemma \ref{l:MbW}]
The symmetry property stated in Lemma \ref{l:MbW}.\eqref{l:MbW:sym} follows directly from the definition. We continue by showing that $\bW$ as defined in \eqref{fd:bW:genl} is a metric. Symmetry, non-negativity and the property that $\bW (\bmu, \bnu) = 0$ implies $\bmu = \bnu$ are easily checked. To prove the triangle inequality, we consider arbitrary $\bmu_i \in M([0,1])$ for $i = 1,2,3$, and let $\bsigma_i := \bmu_i ([0,1])$ the corresponding masses. For convenience, we set
\begin{equation*}
  \bw_{ij} := \bW^2 (\bmu_i, \bmu_j), \quad
  w_{ij}^\pm := W^2 \bigg( \frac{\mu_i^\pm}{\sigma_i^\pm}, \frac{\mu_j^\pm}{\sigma_j^\pm} \bigg), \quad
  \sigma_{ij}^\pm := \sigma_i^\pm \wedge \sigma_j^\pm, \quad
  d_{ij}^\pm := \big| \sigma_i^\pm - \sigma_j^\pm \big|
\end{equation*}
for all $1 \leq i < j \leq 3$, and set $\sum_{p = \pm} \beta^p := \beta^+ + \beta^-$ for any symbol $\beta$. We prove the triangle inequality in squared form:
\begin{equation} \label{fp:ti:sqd}
   \bw_{12} + \bw_{23} + 2 \sqrt{ \bw_{12} \bw_{23} } - \bw_{13} \geq 0.
\end{equation}

We start with a few observations. \eqref{fd:bW:genl} reads
\begin{equation*}
  \bw_{ij} = \sum_{p = \pm} \big( \sigma_{ij}^p w_{ij}^p + d_{ij}^p \big).
\end{equation*}
Since $W$ satisfies the triangle inequality, we have
\begin{equation} \label{fp:ti:W}
  w_{13}^\pm \leq w_{12}^\pm + w_{23}^\pm + 2 \sqrt{ w_{12}^\pm w_{23}^\pm }.
\end{equation}
For the constants $\sigma_{ij}^\pm$ and $d_{ij}^\pm$, it holds that
\begin{equation} \label{fp:s13:UB}
  \sigma_{ij}^\pm + d_{ij}^\pm
  = \sigma_i^\pm \vee \sigma_j^\pm \geq \sigma_{13}^\pm
  \quad \text{for all } 1 \leq i < j \leq 3.
\end{equation} 
Finally, for $p \in \{+,-\}$,
\begin{equation} \label{fp:s1:s2:s3}
  \sigma_2^p \geq \sigma_{13}^p
  \quad \Longrightarrow \quad
  \sigma_{12}^p \wedge \sigma_{23}^p = \sigma_{13}^p 
  \text{ and } d_{12}^p + d_{23}^p \geq d_{13}^p,
\end{equation}
and, 
\begin{equation} \label{fp:s2:s1:s3}
  \sigma_2^p < \sigma_{13}^p
  \quad \Longrightarrow \quad
  \left\{ \begin{aligned}
    \sigma_2^p - \sigma_{13}^p + d_{12}^p = 0
    &\text{ and } d_{12}^p + d_{23}^p = 2 d_{12}^p + d_{13}^p,
    && \text{if }  \sigma_1^p \leq \sigma_3^p, \\
    \sigma_2^p - \sigma_{13}^p + d_{23}^p = 0
    &\text{ and } d_{12}^p + d_{23}^p = 2 d_{23}^p + d_{13}^p,
    && \text{if }  \sigma_3^p \leq \sigma_1^p.
  \end{aligned} \right.
\end{equation}

Next we prove \eqref{fp:ti:sqd}. Using \eqref{fp:ti:W}, we estimate
\begin{equation*}
  \bw_{13} 
  = \sum_{p = \pm} \big( \sigma_{13}^p w_{13}^p + d_{13}^p \big)
  \leq \sum_{p = \pm} \big( \sigma_{13}^p w_{12}^p + \sigma_{13}^p w_{23}^p + 2 \sigma_{13}^p \sqrt{ w_{12}^p w_{23}^p } + d_{13}^p \big),
\end{equation*}
and find
\begin{multline*}
  \bw_{12} + \bw_{23} + 2 \sqrt{ \bw_{12} \bw_{23} } - \bw_{13} 
  \geq \sum_{p = \pm} \Big[ \big( \sigma_{12}^p - \sigma_{13}^p \big) w_{12}^p + \big( \sigma_{23}^p - \sigma_{13}^p \big) w_{23}^p + d_{12}^p + d_{23}^p - d_{13}^p \Big] \\
  + 
  2 \bigg[ \sqrt{ \sum_{p = \pm} \sum_{q = \pm} \big( \sigma_{12}^p w_{12}^p + d_{12}^p \big) \big( \sigma_{23}^q w_{23}^q + d_{23}^q \big) } - \sum_{p = \pm} \sigma_{13}^p \sqrt{ w_{12}^p w_{23}^p } \bigg].
\end{multline*}
It is enough to prove that both terms in the right-hand side are non-negative. Non-negativity of the first term follow either from \eqref{fp:s1:s2:s3} or from \eqref{fp:s2:s1:s3} with $w_{ij}^p \leq 1$. To prove non-negativity of the second term, we show that the first term within brackets is larger or equal to the second term by using respectively $w_{ij}^p \leq 1$, \eqref{fp:s13:UB} and $a + b \geq 2 \sqrt{ab}$, i.e.,
\begin{align*}
  &\sum_{p = \pm} \sum_{q = \pm} \big( \sigma_{12}^p w_{12}^p + d_{12}^p \big) \big( \sigma_{23}^q w_{23}^q + d_{23}^q \big) 
  \geq \sum_{p = \pm} \sum_{q = \pm} \big( \sigma_{12}^p + d_{12}^p \big) w_{12}^p \big( \sigma_{23}^q + d_{23}^q \big) w_{23}^q \\
  &\geq \sum_{p = \pm} \sum_{q = \pm} \sigma_{13}^p w_{12}^p \sigma_{13}^q w_{23}^q 
  = \sum_{p = \pm} \sigma_{13}^p w_{12}^p \sigma_{13}^p w_{23}^p
    + \sum_{p = \pm} \sigma_{13}^p \sigma_{13}^{-p} w_{12}^p w_{23}^{-p} \\
  &\geq \sum_{p = \pm} \sigma_{13}^p \sigma_{13}^p w_{12}^p w_{23}^p
    + 2 \sigma_{13}^+ \sigma_{13}^- \sqrt{ w_{12}^+  w_{23}^- w_{12}^- w_{23}^+ }
  = \sum_{p = \pm} \sum_{q = \pm} \sigma_{13}^p \sigma_{13}^q \sqrt{ w_{12}^p  w_{23}^p w_{12}^q w_{23}^q } \\
  &= \bigg( \sum_{p = \pm} \sigma_{13}^p \sqrt{ w_{12}^p w_{23}^p } \bigg)^2.
\end{align*}
This completes the proof for $\bW$ satisfying the triangle inequality, which is the last step for proving that $\bW$ is a metric.

\medskip

Separability follows easily from $(P ([0,1]), W)$ being separable. We continue with proving completeness. Let $(\bmu_k)$ be a $\bW$-Cauchy sequence with masses $\bsigma_k := \bmu_k ([0,1])$. From the definition of $\bW$ in \eqref{fd:bW:genl} it follows that $\sigma^\pm_k$ are Cauchy sequences in $[0,1]$, which therefore converge to the limiting values $\sigma^\pm \in [0,1]$, which moreover satisfy $\sigma^+ + \sigma^- = 1$. 

Let us first consider the case in which $\sigma^+ \in (0,1)$. Then, for all $k$ large enough, $\sigma_k^\pm \geq \tfrac12 \sigma^\pm > 0$. It follows from the definition of $\bW$ that $\mu^\pm_k / \sigma_k^\pm$ are $W$-Cauchy sequences in $\mathcal P ([0,1])$, and hence $(\mu^\pm_k / \sigma_k^\pm)$ converges to some $\tilde \mu^\pm \in \mathcal P ([0,1])$ with respect to $W$. Setting $\mu^\pm := \sigma^\pm \tilde \mu^\pm$, we conclude that $\bW(\bmu_k, \bmu) \to 0$ as $k \to \infty$.

The final step of the proof for completeness is to treat the case $\sigma^+ \in \{0,1\}$. By symmetry between the positive and the negative parts, it is enough to consider $\sigma^+ = 0$. By the previous argument, we obtain that $\mu^-_k / \sigma_k^-$ converges to $\mu^- \in \mathcal P ([0,1])$ with respect to $W$. Consequently, we need to set $\mu^+ = 0$ in order for $\bmu \in M([0,1])$. we find from \eqref{fd:bW:genl:si0} that 
\begin{equation*}
  \bW^2 (\bmu_k, \bmu) 
  = \sigma_k^+ + \sigma_k^- W^2 \bigg( \frac{ \mu^-_k }{ \sigma_k^- }, \mu^- \bigg) + |\sigma_k^- - 1|
  \xto{k \to \infty} 0.
\end{equation*}

\medskip 

Next we prove the equivalence between the topology induced by $\bW$ and the narrow topology as in Lemma \ref{l:MbW}.\eqref{l:MbW:NT}. We show that this is a consequence of \cite[Prop.~7.1.5]{AmbrosioGigliSavare08}, which states that for $(\mu_k) \subset \mathcal P([0,1])$,
\begin{equation} \label{fp:equiv:ntW:P}
  \mu_k \weakto \mu 
  \quad \Longleftrightarrow \quad
  W(\mu_k, \mu) \to 0.
\end{equation}
We first prove that $\bmu_k \weakto \bmu$ implies $\bW (\bmu_k, \bmu) \to 0$. Choosing $\varphi^\pm \equiv 1$ in \eqref{for:defn:narrow:conv}, we find that the masses satisfy $\sigma^\pm_k \to \sigma^\pm$. By Lemma \ref{l:MbW}.\eqref{l:MbW:sym}, it is therefore enough to show that
\begin{equation} \label{fp:Wp:to:0}
  (\sigma_k^+ \wedge \sigma^+) W^2 \bigg( \frac{ \mu^+_k }{ \sigma^+_k }, \frac{ \mu^+ }{ \sigma^+ } \bigg) \to 0.
\end{equation}
This is trivial for $\sigma^+ = 0$, so let us assume $\sigma^+ > 0$. We take $k$ large enough such that $\sigma_k^+ > \sigma^+/2 > 0$. Then, $(\sigma_k^+)^{-1} \to (\sigma^+)^{-1}$, and thus 
\begin{equation} \label{fp:musi:conv}
  \frac{ \mu^+_k }{ \sigma^+_k } \weakto \frac{ \mu^+ }{ \sigma^+ }.
\end{equation}
We conclude from \eqref{fp:equiv:ntW:P} that \eqref{fp:Wp:to:0} holds.

Next we prove the opposite implication. Let $\bW (\bmu_k, \bmu) \to 0$. Again, $\sigma^\pm_k \to \sigma^\pm$, and by Lemma \ref{l:MbW}.\eqref{l:MbW:sym} it is enough to show $\mu_k^+ \weakto \mu^+$. If $\sigma^+ = 0$, then $\mu^+ = 0$ and $\sigma_k^+ \to 0$, and thus 
\begin{equation*}
  \bigg| \int \varphi \, d\mu_k^+ \bigg| 
  \leq \sigma_k^+ \|\varphi\|_\infty
  \xto{ k \to \infty } 0
  = \int \varphi \, d\mu^+
\end{equation*}
for any test function $\varphi \in C([0,1])$. If $\sigma^+ > 0$, then we take $k$ large enough such that $\sigma_k^+ > \sigma^+/2 > 0$. Then, $\bW (\bmu_k, \bmu) \to 0$ implies
\begin{equation*}
  W^2 \bigg( \frac{ \mu^+_k }{ \sigma^+_k }, \frac{ \mu^+ }{ \sigma^+ } \bigg) \to 0.
\end{equation*}
By \eqref{fp:equiv:ntW:P} we obtain \eqref{fp:musi:conv}, and we conclude $\mu_k^+ \weakto \mu^+$ from $(\sigma_k^+)^{-1} \to (\sigma^+)^{-1}$.

With Lemma \ref{l:MbW}.\eqref{l:MbW:NT} established, sequential compactness follows from Prokhorov, which states that any sequence $\bmu_k$ has a narrowly converging subsequence.

\medskip

Next we prove that $M_{\sigma^+} ([0,1])$ is non-positively curved (see \eqref{fd:npcurved}), i.e., for all $\bmu_0, \bmu_1, \bnu \in M_{\sigma^+} ([0,1])$,
\begin{equation} \label{fp:pcurved}
  \bW^2 (\bmu_t, \bnu) 
  \leq (1-t) \bW^2 (\bmu_0, \bnu) + t \bW^2 (\bmu_1, \bnu) - t (1-t) \bW^2 (\bmu_0, \bmu_1)
  \quad \text{for all } 0 \leq t \leq 1,
\end{equation}
where $\bmu_t$ is a geodesic connecting $\bmu_0$ and $\bmu_1$. Since the case $\sigma^+ \in \{0,1\}$ follows by a simplification of the argument below, we assume $\sigma^+ \in (0,1)$. To characterise $\bmu_t$, let $\tilde \mu_t^\pm$ be a $W$-geodesic in $\mathcal P([0,1])$ connecting $\mu_0^\pm / \sigma^\pm$ with $\mu_1^\pm / \sigma^\pm$. We observe that $\bmu_t := ( \sigma^+ \tilde \mu_t^+, \, \sigma^- \tilde \mu_t^- ) \in M_{\sigma^+} ([0,1])$ is a $\bW$-geodesic (see \eqref{fd:geods}) from
\begin{equation*}
  \bW^2 (\bmu_s, \bmu_t)
  = \sum_{p = \pm} \frac1{\sigma^p} W^2 \big( \tilde \mu_s^\pm, \tilde \mu_t^\pm \big)
  = \sum_{p = \pm} \frac1{\sigma^p} (t-s)^2 W^2 \big( \tilde \mu_0^\pm, \tilde \mu_1^\pm \big)
  = (t-s)^2 \bW^2 (\bmu_0, \bmu_1),
\end{equation*}
which holds for all $0 \leq s \leq t \leq 1$. Then, by the additive structure of $\bW$ as in \eqref{fd:bW:si} which separates the dependence on $\mu^+$ and $\mu^-$, \eqref{fp:pcurved} follows directly from $(\mathcal P([0,1]), W)$ being non-positively curved (which is a consequence of \cite[(7.2.8)]{AmbrosioGigliSavare08}).

Next we prove the characterisation in Lemma \ref{l:MbW}.\eqref{l:MbW:Rn} of $\bW^2 (\bmu_n, \bnu_n)$ for the empirical measures $\bmu_n$ and $\bnu_n$ corresponding to $x^n, y^n \in \Omega_n$. For any such measures, we compute
\begin{multline} \label{f:bW:ito:mun}
   \bW^2 (\bmu_n, \bnu_n)
   = \frac{n^+}n W^2 \bigg( \frac1{n^+} \sum_{i=1}^{n^+} \delta_{x_i^+} , \frac1{n^+} \sum_{i=1}^{n^+} \delta_{y_i^+} \bigg) + \frac{n^-}n W^2 \bigg( \frac1{n^-} \sum_{i=1}^{n^-} \delta_{x_i^-}, \frac1{n^-} \sum_{i=1}^{n^-} \delta_{y_i^-} \bigg) \\
   = \frac{n^+}n \frac1{n^+} \sum_{i=1}^{n^+} \big( x_i^+ - y_i^+ \big)^2 + \frac{n^-}n \frac1{n^-} \sum_{i=1}^{n^-} \big( x_i^- - y_i^- \big)^2
   = \frac1n |x^n - y^n|^2.
\end{multline} 

Finally, we prove the property of geodesics in $M([0,1])$ stated in Lemma \ref{l:MbW}.\eqref{l:MbW:geods}. We reason by contradiction. Let $t \mapsto \bmu_t$ be any $\bW$-geodesic for which there exists a $t \in (0, 1)$ such that $\bmu_t \notin M_{\sigma^+} ([0,1])$. Setting $\sigma_t^\pm := \mu_t^\pm ([0,1])$, we define $\nu^\pm := \sigma^\pm \mu_t^\pm / \sigma_t^\pm$,
and note that $\bnu \in M_\sigma ([0,1])$. We compute
\begin{equation*}
  \bW^2( \bmu_0, \bmu_t ) 
  = \bW^2 ( \bmu_0, \bnu ) + \bW^2( \bnu, \bmu_t )
  > \bW^2 ( \bmu_0, \bnu ),
\end{equation*}
and, similarly, $\bW( \bmu_1, \bmu_t ) > \bW ( \bmu_1, \bnu )$. Hence, $t \mapsto \bmu_t$ is not a geodesic. 
\end{proof}

\bibliographystyle{alpha}

\begin{thebibliography}{GvMPS16}

\bibitem[AGS08]{AmbrosioGigliSavare08}
L.~Ambrosio, N.~Gigli, and G.~Savar{\'e}.
\newblock {\em Gradient Flows: In Metric Spaces and in the Space of Probability
  Measures}.
\newblock Birkh\"auser Verlag, New York, 2008.

\bibitem[BBP16]{BerendsenBurgerPietschmann16ArXiv}
J.~Berendsen, M.~Burger, and J.-F. Pietschmann.
\newblock On a cross-diffusion model for multiple species with nonlocal
  interaction and size exclusion.
\newblock {\em ArXiv: 1609.05024}, 2016.

\bibitem[BG04]{BraidesGelli04}
A.~Braides and M.~S. Gelli.
\newblock The passage from discrete to continuous variational problems: a
  nonlinear homogenization process.
\newblock In {\em Nonlinear Homogenization and its Applications to Composites,
  Polycrystals and Smart Materials}, pages 45--63. Springer, 2004.

\bibitem[Bil68]{Billingsley68}
P.~Billingsley.
\newblock {\em Convergence of probability measures}.
\newblock John Wiley \& Sons, New York, 1968.

\bibitem[BLBL07]{BlancLeBrisLions07}
X.~Blanc, C.~Le~Bris, and P.-L. Lions.
\newblock Atomistic to continuum limits for computational materials science.
\newblock {\em ESAIM: Mathematical Modelling and Numerical Analysis},
  41(2):391--426, 2007.

\bibitem[CP16]{CanizoPatacchini16ArXiv}
J.~A. Ca{\~n}izo and F.~S. Patacchini.
\newblock Discrete minimisers are close to continuum minimisers for the
  interaction energy.
\newblock {\em ArXiv: 1612.09233}, 2016.

\bibitem[CXZ16]{ChapmanXiangZhu15}
S.~J. Chapman, Y.~Xiang, and Y.~Zhu.
\newblock Homogenization of a row of dislocation dipoles from discrete
  dislocation dynamics.
\newblock {\em SIAM Journal on Applied Mathematics}, 76(2):750--775, 2016.

\bibitem[DFF13]{DiFrancescoFagioli13}
M.~Di~Francesco and S.~Fagioli.
\newblock Measure solutions for non-local interaction pdes with two species.
\newblock {\em Nonlinearity}, 26(10):2777, 2013.

\bibitem[DFF16]{DiFrancescoFagioli16}
M.~Di~Francesco and S.~Fagioli.
\newblock A nonlocal swarm model for predators--prey interactions.
\newblock {\em Mathematical Models and Methods in Applied Sciences},
  26(02):319--355, 2016.

\bibitem[DPG15]{DoggePeerlingsGeers15b}
M.~M.~W. Dogge, R.~H.~J. Peerlings, and M.~G.~D. Geers.
\newblock Extended modelling of dislocation transport--formulation and finite
  element implementation.
\newblock {\em Advanced Modeling and Simulation in Engineering Sciences},
  2(1):29, 2015.

\bibitem[DS10]{DaneriSavareLN10}
S.~Daneri and G.~Savar{\'e}.
\newblock Lecture notes on gradient flows and optimal transport.
\newblock {\em ArXiv: 1009.3737}, 2010.

\bibitem[Dud66]{Dudley66}
R.~Dudley.
\newblock Convergence of baire measures.
\newblock {\em Studia Mathematica}, 27(3):251--268, 1966.

\bibitem[Due16]{Duerinckx16}
M.~Duerinckx.
\newblock Mean-field limits for some {R}iesz interaction gradient flows.
\newblock {\em SIAM Journal on Mathematical Analysis}, 48(3):2269--2300, 2016.

\bibitem[EFK16]{EversFetecauKolokolnikov16ArXiv}
J.~H.~M. Evers, R.~C. Fetecau, and T.~Kolokolnikov.
\newblock Equilibria for an aggregation model with two species.
\newblock {\em ArXiv: 1612.08074}, 2016.

\bibitem[FIM09]{ForcadelImbertMonneau09}
N.~Forcadel, C.~Imbert, and R.~Monneau.
\newblock Homogenization of the dislocation dynamics and of some particle
  systems with two-body interactions.
\newblock {\em Discrete and Continuous Dynamical Systems A}, 23(3):785--826,
  2009.

\bibitem[GB99]{GromaBalogh99}
I.~Groma and P.~Balogh.
\newblock Investigation of dislocation pattern formation in a two-dimensional
  self-consistent field approximation.
\newblock {\em Acta Materialia}, 47(13):3647--3654, 1999.

\bibitem[GCZ03]{GromaCsikorZaiser03}
I.~Groma, F.~F. Csikor, and M.~Zaiser.
\newblock Spatial correlations and higher-order gradient terms in a continuum
  description of dislocation dynamics.
\newblock {\em Acta Materialia}, 51(5):1271--1281, 2003.

\bibitem[GLP10]{GarroniLeoniPonsiglione10}
A.~Garroni, G.~Leoni, and M.~Ponsiglione.
\newblock Gradient theory for plasticity via homogenization of discrete
  dislocations.
\newblock {\em Journal European Mathematical Society}, 12(5):1231--1266, 2010.

\bibitem[GPPS13]{GeersPeerlingsPeletierScardia13}
M.~G.~D. Geers, R.~H.~J. Peerlings, M.~A. Peletier, and L.~Scardia.
\newblock Asymptotic behaviour of a pile-up of infinite walls of edge
  dislocations.
\newblock {\em Archive for Rational Mechanics and Analysis}, 209:495--539,
  2013.

\bibitem[GvMPS16]{GarroniVanMeursPeletierScardia16}
A.~Garroni, P.~{v}an Meurs, M.~A. Peletier, and L.~Scardia.
\newblock Boundary-layer analysis of a pile-up of walls of edge dislocations at
  a lock.
\newblock {\em Mathematical Models and Methods in Applied Sciences},
  26(14):2735--2768, 2016.

\bibitem[Hau09]{Hauray09}
M.~Hauray.
\newblock Wasserstein distances for vortices approximation of {E}uler-type
  equations.
\newblock {\em Mathematical Models and Methods in Applied Sciences},
  19(08):1357--1384, 2009.

\bibitem[HB01]{HullBacon01}
D.~Hull and D.~J. Bacon.
\newblock {\em Introduction to Dislocations}.
\newblock Butterworth Heinemann, Oxford, 2001.

\bibitem[HCO10]{HallChapmanOckendon10}
C.~L. Hall, S.~J. Chapman, and J.~R. Ockendon.
\newblock Asymptotic analysis of a system of algebraic equations arising in
  dislocation theory.
\newblock {\em SIAM Journal on Applied Mathematics}, 70(7):2729--2749, 2010.

\bibitem[Hea59]{Head59}
A.~K. Head.
\newblock The positions of dislocations in arrays.
\newblock {\em Philosophical Magazine}, 4(39):295--302, 1959.

\bibitem[HHvM16]{HallHudsonVanMeurs16ArXiv}
C.~L. Hall, T.~Hudson, and P.~van Meurs.
\newblock Asymptotic analysis of boundary layers in a repulsive particle
  system.
\newblock {\em ArXiv: 1609.03236}, 2016.

\bibitem[HL82]{HirthLothe82}
J.~P. Hirth and J.~Lothe.
\newblock {\em Theory of Dislocations}.
\newblock John Wiley \& Sons, New York, 1982.

\bibitem[HO14]{HudsonOrtner14}
T.~Hudson and C.~Ortner.
\newblock Existence and stability of a screw dislocation under anti-plane
  deformation.
\newblock {\em Archive for Rational Mechanics and Analysis}, 213(3):887--929,
  2014.

\bibitem[Hud13]{Hudson13}
T.~Hudson.
\newblock Gamma-expansion for a 1{D} confined {L}ennard-{J}ones model with
  point defect.
\newblock {\em Networks \& Heterogeneous Media}, 8(2):501--527, 2013.

\bibitem[HZG07]{HochrainerZaiserGumbsch07}
T.~Hochrainer, M.~Zaiser, and P.~Gumbsch.
\newblock A three-dimensional continuum theory of dislocation systems:
  kinematics and mean-field formulation.
\newblock {\em Philosophical Magazine}, 87(8-9):1261--1282, 2007.

\bibitem[KHG15]{KooimanHuetterGeers15}
M.~Kooiman, M.~H{\"u}tter, and M.~G.~D. Geers.
\newblock Microscopically derived free energy of dislocations.
\newblock {\em Journal of the Mechanics and Physics of Solids}, 78:186--209,
  2015.

\bibitem[MPS14]{MoraPeletierScardia14ArXiv}
M.~G. Mora, M.~A. Peletier, and L.~Scardia.
\newblock Convergence of interaction-driven evolutions of dislocations with
  {W}asserstein dissipation and slip-plane confinement.
\newblock {\em ArXiv: 1409.4236}, 2014.

\bibitem[Oel90]{Oelschlager90}
K.~Oelschl{\"a}ger.
\newblock Large systems of interacting particles and the porous medium
  equation.
\newblock {\em Journal of Differential Equations}, 88(2):294--346, 1990.

\bibitem[Ort05]{Ortner05}
C.~Ortner.
\newblock Two variational techniques for the approximation of curves of maximal
  slope.
\newblock Technical report, Oxford University Computing Laboratory, 2005.

\bibitem[PS14]{PetracheSerfaty14}
M.~Petrache and S.~Serfaty.
\newblock Next order asymptotics and renormalized energy for {R}iesz
  interactions.
\newblock {\em Journal of the Institute of Mathematics of Jussieu}, pages
  1--69, 2014.

\bibitem[Sch96]{Schochet96}
S.~Schochet.
\newblock The point-vortex method for periodic weak solutions of the 2-d euler
  equations.
\newblock {\em Communications on Pure and Applied Mathematics}, 49(9):911--965,
  1996.

\bibitem[SPPG14]{ScardiaPeerlingsPeletierGeers14}
L.~Scardia, R.~H.~J. Peerlings, M.~A. Peletier, and M.~G.~D. Geers.
\newblock Mechanics of dislocation pile-ups: a unification of scaling regimes.
\newblock {\em Journal of the Mechanics and Physics of Solids}, 70:42--61,
  2014.

\bibitem[SR97]{ShampineReichelt97}
L.~F. Shampine and M.~W. Reichelt.
\newblock The {M}atlab ode suite.
\newblock {\em SIAM Journal on Scientific Computing}, 18(1):1--22, 1997.

\bibitem[SS04]{SandierSerfaty04}
E.~Sandier and S.~Serfaty.
\newblock Gamma-convergence of gradient flows with applications to
  {G}inzburg-{L}andau.
\newblock {\em Communications on Pure and Applied Mathematics}, 57:1627--1672,
  2004.

\bibitem[Vil03]{Villani03}
C.~Villani.
\newblock {\em {T}opics in {O}ptimal {T}ransportation}.
\newblock American Mathematical Society, Providence Rhode Island, 2003.

\bibitem[vM15]{VanMeurs15}
P.~{v}an Meurs.
\newblock {\em Discrete-to-Continuum Limits of Interacting Dislocations}.
\newblock PhD thesis, Eindhoven University of Technology, 2015.

\bibitem[vMM14]{VanMeursMuntean14}
P.~{v}an Meurs and A.~Muntean.
\newblock Upscaling of the dynamics of dislocation walls.
\newblock {\em Advances in Mathematical Sciences and Applications},
  24(2):401--414, 2014.

\bibitem[vMMP14]{VanMeursMunteanPeletier14}
P.~{v}an Meurs, A.~Muntean, and M.~A. Peletier.
\newblock Upscaling of dislocation walls in finite domains.
\newblock {\em European Journal of Applied Mathematics}, 25(6):749--781, 2014.

\bibitem[Zin16]{Zinsl16}
J.~Zinsl.
\newblock Geodesically convex energies and confinement of solutions for a
  multi-component system of nonlocal interaction equations.
\newblock {\em Nonlinear Differential Equations and Applications NoDEA},
  23(4):43, 2016.

\end{thebibliography}

\end{document}